\documentclass[11pt]{amsart}
\usepackage{geometry}                % See geometry.pdf to learn the layout options. There are lots.
\usepackage{graphicx}
\usepackage{amsmath,amsthm,amssymb}
\usepackage{epstopdf}
\usepackage {tikz}
\usepackage{bm}
\usepackage{hyperref}
\usepackage{ytableau} % easier for making skew-tableaux with empty boxes \begin{ytableaux} & \\ \end{ytableaux} or \ydiagram{1+2,2}
\usepackage{young}  % type \begin{Young} a&b \cr c \cr\end{Young}
\usepackage[enableskew]{youngtab} % type \yng(2,1)
\usepackage{pgfplots}
\usepackage{caption}
\usepgfplotslibrary{ternary}
\usepackage{subcaption}
\usepackage{mathtools} 
\usetikzlibrary{decorations.markings}
\usetikzlibrary{arrows.meta}
\usetikzlibrary{shapes.geometric}
\usepackage{enumitem}

\tikzset{->-/.style={decoration={
  markings,
  mark=at position .5 with {\arrow{>}}},postaction={decorate}}}
  
\pgfplotsset{compat=1.13}

\newcommand\scalemath[2]{\scalebox{#1}{\mbox{\ensuremath{\displaystyle #2}}}}

\DeclareMathOperator{\Hom}{Hom}

\newtheorem*{Thm*}{Theorem A}
\newtheorem*{Thm**}{Theorem B}
\newtheorem{Thm}{Theorem}
\newtheorem{Prop}[Thm]{Proposition}
\newtheorem{Lem}[Thm]{Lemma}
\newtheorem{Cor}[Thm]{Corollary}

\theoremstyle{remark}
\newtheorem{Rem}[Thm]{Remark}

\theoremstyle{definition}
\newtheorem{Def}[Thm]{Definition}

\begin{document}

\title{Growth Diagrams and Minuscule Polygon Configurations in the Affine Grassmannian}
\author{Tair Akhmejanov} %include [T. Akhmejanov] for a footnote
%\date{}                                           % Activate to display a given date or no date
  
\begin{abstract}
We define affine growth diagrams consisting of $GL_m$ dominant weights that label the vertices of a staircase-shaped grid. These are also called cylindrical growth diagrams as defined by Speyer and White in the case of partitions. The weights labelling each adjacent pair of vertices differ by a vertical strip and the weights around each unit square satisfy a local condition that appeared in van Leeuwen's work on the Littelmann path model for crystals.

We prove two main results. For a sequence of minuscule weights $\vec\lambda=(\lambda^1,\ldots,\lambda^n)$ let Poly$(\vec\lambda)$ denote the configuration space of $n$-tuples of points $(g_1,\ldots,g_n)$ in the affine Grassmannian such that the weight-valued distances satisfy $d(g_i,g_{i+1})=\lambda^i$.  This is the convolution variety arising in the geometric Satake correspondence. We show that for a generic point $(g_1,\ldots,g_n)$ of a component the distances $d(g_i,g_j)$ form an affine growth diagram and that this gives a bijection between components of Poly$(\vec\lambda)$ and affine growth diagrams of type $\vec\lambda$. The main tool used in the proof is the Knutson--Tao hive. 

In the second part, we give a purely combinatorial construction of affine growth diagrams from natural number entries by applying Greene's theorem to certain subrectangles of the staircase. From this construction it follows that affine growth diagrams contain the classical Fomin growth diagrams and realize the RS-correspondence when $\vec\lambda=(\omega_1,\ldots,\omega_1,\omega_1^*,\ldots,\omega_1^*)$. 
\end{abstract}

%\setcounter{tocdepth}{1} %stops at section, i.e. doesn't show subsection, subsubsection
%\addcontentsline{toc}{section}{Introduction}
\maketitle
\tableofcontents

%%%%%%%%%%%%%%%%%%%%%%%%%%%%%%%%%%%%%%%%%%%%%%%%%%%%
\section{Introduction}
\subsection{The Main Definition}
The central objects of this paper are staircase-shaped diagrams consisting of $GL_m$ dominant weights. Recall that the dominant-weight lattice of $GL_m$ is the set of weakly decreasing sequences of integers $\lambda_1\geq \ldots\geq\lambda_m$. Let $\omega_k=(1,\ldots,1,0,\ldots,0)$ denote the $k$th fundamental weight of $GL_m$, where there are $k$ many $1$'s, and $\omega_k^*=(0,\ldots,0,-1,\ldots,-1)$ its dual. The dual of a weight $\lambda=(\lambda_1,\ldots,\lambda_m)$  is the negated reversal $\lambda^*=(-\lambda_m,\ldots,-\lambda_1)$. Define the \emph{minuscule weights} as the set of weights $\omega_k$ and $\omega_k^*$ for $1\leq k< m$. We will use $\vec{0}$ to denote the zero weight. Let
\begin{align*}
St_n=\{(i,j)\in \mathbb Z^2\mid i\leq j\leq i+n\}
\end{align*}
be the subset of points in $\mathbb Z^2$ forming an infinite staircase of width $n$. The point $(i,j)$ will be viewed as the vertex in row $i$ and column $j$ as in matrix notation.
\begin{Def}\label{diagram}
Let $\vec\lambda=(\lambda^1,\ldots,\lambda^n)$ be a sequence of minuscule $GL_m$ weights. An \emph{affine growth diagram of type $\vec{\bf \lambda}$} is a labelling of the lattice points in $St_n$ by dominant weights $\{\gamma_{i,j}\}_{(i,j)\in St_n}$, such that for all $i$ we have $\gamma_{i,i}=\vec{0}=\gamma_{i,i+n}$ and $\gamma_{i,i+1}=\lambda^{i}$ where the indices $i$ in $\lambda^i$ are taken modulo $n$. Furthermore, each of the differences $\gamma_{i,j+1}-\gamma_{i,j}$ and $\gamma_{i+1,j}-\gamma_{i,j}$ is a vertical strip, and for each unit square the following local condition is satisfied.
\begin{align}\label{local condition}
\gamma_{i+1,j+1}=\text{sort}(\gamma_{i+1,j}+\gamma_{i,j+1}-\gamma_{i,j})
\end{align}
\end{Def}
Since weights are $m$-tuples of integers, the addition and subtraction operations are defined component-wise. A \emph{positive vertical strip} (resp. \emph{negative vertical strip}) is a sequence of integers $(z_1,\ldots,z_n)$, not necessarily weakly decreasing, where each $z_i$ is equal to $0$ or $1$ (resp. $0$ or $-1$). By \emph{vertical strip} we mean either a positive or negative vertical strip. The sort operator puts any sequence of integers into weakly decreasing order, thereby making the weight dominant. 

Growth diagrams on a staircase were defined in \cite{Spe, Whi} for the case that each vertex is labelled by a partition and neighboring partitions differ by a box. There they are called cylindrical diagrams due to their periodicity property. Here we call them affine growth diagrams due to their connection to the geometry of configurations in the affine Grassmannian. In the setup of \cite{Spe,Whi}, the labels $\gamma_{i,i}$ are the zero partition and the labels $\gamma_{i,i+n}$ are a fixed rectangular shape for all $i$. Our main combinatorial result is that allowing for $GL_m$ weights and vertical strips leads to the realization of classical Fomin growth diagrams \cite{Fom} and the RSK-correspondence within these diagrams. See also \cite{PRW, Wes} for connections to coboundary categories. The local condition (\ref{local condition}) originally appeared in the work of van Leeuwen \cite{vLee2}, as explained in \S\ref{sec:crystals}.

A convenient way to visualize a dominant $GL_m$ weight $\gamma$ is as $m$ rows of boxes where the positive entries of $\gamma$ correspond to ``positive'' boxes and the negative entries to ``negative'' boxes. Since $\gamma$ is weakly decreasing this gives a positive Young diagram above a $90$-degree rotated, ``negative'' one. See Figures \ref{empty} and \ref{first example} for an example of an affine growth diagram with weights drawn in this way. Although the definition consists of infinitely many dominant weights, it will follow from the connection to geometry discussed in the next subsection of the introduction that these diagrams have period $n$ in the rows, so consist of a finite amount of data. Throughout the paper, $n$ will refer to the width of the staircase diagrams and $m$ will refer to $GL_m$. 

This paper contains two parts that are largely independent of each other, but have affine growth diagrams as the common object. The first part draws a connection to the geometry of the affine Grassmannian. In \S\ref{geometry} we prove that the components of the convolution varieties arising in the geometric Satake correspondence for minuscule $GL_m$ weights are in bijection with affine growth diagrams. The methods are combinatorial, once we recall the result of Goncharov-Shen \cite{GS} giving a bijection between components and Knutson--Tao hives. 

In \S\ref{combinatorics} we study affine growth diagrams purely combinatorially, independent to their relation to geometry, although it is interesting that such a connection exists. We prove that classical Fomin growth diagrams appear as subdiagrams within affine growth diagrams, recovering the Robinson--Schensted correspondence. Fomin growth diagrams have an interpretation in terms of Greene's theorem. The combinatorial construction of \S\ref{sec:main combinatorial theorem} generalizes this interpretation to staircase diagrams and may be of independent interest. We discuss both the geometric and combinatorial results in the remainder of this introduction. 

\begin{figure}
\centering
\begin{minipage}{.49\textwidth}
\centering
\scalemath{.7}{
\begin{tikzpicture}
\draw[step=1cm] (-5,2) grid (0,3);
\draw[step=1cm] (-4,1) grid (1,2);
\draw[step=1cm] (-3,0) grid (2,1);
\draw[step=1cm] (-2,-1) grid (3,0);
\draw[step=1cm] (-1,-2) grid (4,-1);
\draw[step=1cm] (0,-3) grid (5,-2);
\draw (0,-3)--(6,-3);
\draw (-6,3)--(0,3);

\Yboxdim{2mm}
\node at (-5.8,3.2) {$\vec{0}$}; 
\node at (-4.8,3.2) {$\lambda^1$};
\node at (0.2,3.2) {$\vec{0}$};

\node at (-4.8,2.2) {$\vec{0}$}; 
\node at (-3.8,2.2) {$\lambda^2$};
\node at (1.2,2.2) {$\vec{0}$};

\node at (-3.8,1.2) {$\vec{0}$};  
\node at (-2.8,1.2) {$\lambda^3$};
\node at (2.2,1.2) {$\vec{0}$};

\node at (-2.8,.2) {$\vec{0}$};  
\node at (-1.8,.2) {$\lambda^4$};
\node at (3.2,.2) {$\vec{0}$};

\node at (-1.8,-.8) {$\vec{0}$};  
\node at (-.8,-.8) {$\lambda^5$};
\node at (4.2,-.8) {$\vec{0}$};

\node at (-.8,-1.8) {$\vec{0}$};  
\node at (0.2,-1.8) {$\lambda^6$};
\node at (5.2,-1.8) {$\vec{0}$};

\node at (.2,-2.8) {$\vec{0}$};
\node at (1.2,-2.8) {$\lambda^1$};
\node at (6.2,-2.8) {$\vec{0}$};
%\node[circle,fill,inner sep=2pt] at (0,0) {}; %center dot
\end{tikzpicture}}
  \captionsetup{width=.8\linewidth}
\caption{An empty affine growth diagram of type $\vec\lambda=(\lambda^1,\ldots,\lambda^6).$ \label{empty}}
\end{minipage}
%\hspace{5mm}
\begin{minipage}{.49\textwidth}
\centering
\scalemath{.7}{
\begin{tikzpicture}
\draw[step=1cm] (-5,2) grid (0,3);
\draw[step=1cm] (-4,1) grid (1,2);
\draw[step=1cm] (-3,0) grid (2,1);
\draw[step=1cm] (-2,-1) grid (3,0);
\draw[step=1cm] (-1,-2) grid (4,-1);
\draw[step=1cm] (0,-3) grid (5,-2);
\draw (0,-3)--(6,-3);
\draw (-6,3)--(0,3);

\Yboxdim{2mm}
\node at (-5.8,3.2) {$\vec{0}$}; 
\node at (-4.8,3.2) {$\yng(1)$};
\node at (-3.8,3.25) {$\yng(1,1)$};
\node at (-2.7,3.35) {$\young(:~,:~,~)$};
\node at (-1.7,3.35) {$\young(:~,:\empty,~)$};
\node at (0-.7,3.2) {$\yng(1)$};
\node at (0.2,3.2) {$\vec{0}$};
%\node at (-3.5,2.5) {X};
\node at (-4.8,2.2) {$\vec{0}$}; 
\node at (-3.8,2.2) {$\yng(1)$};
\node at (-2.75,2.35) {$\young(:~,:\empty,~)$};
\node at (-1.7,2.35) {$\young(:~,~,~)$};
\node at (-.725,2.35) {$\young(:~,:\empty,~)$};
\node at (0.175,2.35) {$\young(:\empty,:\empty,~)$};
\node at (1.2,2.2) {$\vec{0}$};
%\node at (-2.4,1.5) {X};
\node at (-3.8,1.2) {$\vec{0}$};  
\node at (-2.8,1.35) {$\young(:\empty,:\empty,~)$};
\node at (-1.725,1.35) {$\young(:\empty,~,~)$};
\node at (-.75,1.35) {$\young(:\empty,:\empty,~)$};
\node at (0.25,1.35) {$\young(:\empty,~,~)$};
\node at (1.2,1.35) {$\young(:\empty,:\empty,~)$};
\node at (2.2,1.2) {$\vec{0}$};
%\node at (.65,.5) {X};
\node at (-2.8,.2) {$\vec{0}$};  
\node at (-1.8,.35) {$\young(:\empty,:\empty,\empty)$};
\node at (-.75,.35) {$\young(:~,:\empty,~)$};
\node at (0.3,.35) {$\young(:~,~,~)$};
\node at (1.25,.35) {$\young(:~,:\empty,~)$};
\node at (2.2,.2) {$\yng(1)$};
\node at (3.2,.2) {$\vec{0}$};
%\node at (1.5,-.5) {X};
\node at (-1.8,-.8) {$\vec{0}$};  
\node at (-.8,-.8) {$\yng(1)$};
\node at (0.25,-.65) {$\young(:~,:\empty,~)$};
\node at (1.25,-.65) {$\young(:~,:~,~)$};
\node at (2.2,-.75) {$\yng(1,1)$};
\node at (3.2,-.8) {\yng(1)};
\node at (4.2,-.8) {$\vec{0}$};
%\node at (-.5,-1.5) {X};
\node at (-.8,-1.8) {$\vec{0}$};  
\node at (0.2,-1.65) {$\young(:\empty,:\empty,~)$};
\node at (1.275,-1.65) {$\young(:~,:\empty,~)$};
\node at (2.2,-1.8) {$\yng(1)$};
\node at (3.275,-1.65) {$\young(:~,:\empty,~)$};
\node at (4.175,-1.65) {$\young(:\empty,:\empty,~)$};
\node at (5.2,-1.8) {$\vec{0}$};
%\node at (4.6,-2.5) {X};
\node at (.2,-2.8) {$\vec{0}$};
\node at (1.2,-2.8) {\yng(1)};
\node at (2.2,-2.75) {\yng(1,1)};
\node at (3.3,-2.65) {$\young(:~,:~,~)$};
\node at (4.3,-2.65) {$\young(:~,:\empty,~)$};
\node at (5.2,-2.8) {\yng(1)};
\node at (6.2,-2.8) {$\vec{0}$};
\end{tikzpicture}}
  \captionsetup{width=.9\linewidth}
\caption{An affine growth diagram of type $(\omega_1,\omega_1,\omega_1^*,\omega_1^*,\omega_1,\omega_1^*)$ with $m=3$ and $n=6$. \label{first example}}
\end{minipage}
\end{figure}

\subsection{Polygon Configuration Spaces}\label{sec:intro geometry}
In \S \ref{geometry} we relate affine growth diagrams to configuration spaces of polygons with minuscule side lengths in the affine Grassmannian. See \S\ref{background} for the definitions briefly introduced here. The affine Grassmannian is the quotient $Gr=GL_m(\mathcal K)/GL_m(\mathcal O)$ where $\mathcal K=\mathbb C((t))$ and $\mathcal O=\mathbb C[[t]]$, and is a direct limit of varieties of increasing dimension. The affine Grassmannian has a metric $d(g,h)$, in the sense of Kapovich, Leeb, and Milson \cite{KLM1, KLM2, KLM3}, that takes values in dominant $GL_m$ weights and satisfies $d(g,h)=d(h,g)^*$. For a fixed sequence of minuscule weights $\vec\lambda=(\lambda^1,\ldots,\lambda^n)$ the set of polygon configurations,
\begin{align*}
\text{Poly}(\vec\lambda)=\left\{(g_1=g_{n+1},g_2,\ldots,g_{n})\in Gr^{n}\mid d(g_i,g_{i+1})=\lambda^i, g_1=[1]=g_{n+1}\right\},
\end{align*}
forms a reducible algebraic variety, which we call \emph{the polygon space}. It is shown to be equidimensional in \cite{Hai2}.

These configuration spaces arise in the geometric Satake correspondence of Lusztig \cite{Lus}, Ginzburg \cite{Gin}, Beilinson--Drinfeld \cite{BD}, and Mirkovi\'c--Vilonen \cite{MV}. The geometric Satake correspondence is an equivalence between the tensor category of perverse sheaves on Gr and the tensor category of representations of the Langlands dual group, which is $GL_m$ in this case. In general each $\lambda^i$ can be any dominant weight, not necessarily minuscule, and the polygon space is usually called the \emph{Satake fiber} or \emph{convolution variety}. Under this equivalence the top Borel--Moore homology of the Satake fiber for $\vec\lambda$ is isomorphic to the invariant space $(V_{\lambda^1}\otimes \cdots\otimes V_{\lambda^n})^{GL_m}$ of the corresponding irreducible representations of $GL_m$. Here $V_\lambda$ denotes the finite-dimensional irreducible representation of $GL_m$ of highest weight $\lambda$. The classes of the components form the \emph{Satake basis}. 

This paper is concerned with the combinatorics of the pairwise distances $d(g_i,g_j)$ in the case that each $\lambda^i$ is a minuscule weight. The main result of \S \ref{geometry} is Theorem \ref{thm:geometry theorem}, reproduced here.
\begin{Thm*}
Let $Z$ be a component of Poly$(\vec\lambda)$ and $p=(g_1=[1]=g_{n+1},g_1,\ldots,g_n)$ a generic point in $Z$. The dominant weights $\gamma_{i,j}=d(g_{\overline{i}},g_{\overline{j}})$ for $(i,j)$ in $St_n$ form an affine growth diagram of type $\vec\lambda$ and do not depend on the choice of generic $p$ in $Z$. Furthermore, this is a bijection, that is, the affine growth diagrams of type $\vec\lambda$ index the components of Poly$(\vec\lambda)$.
\end{Thm*}
The indices $(i,j)$ range over $St_n$, so they are taken modulo $n$ in $g_i$ and $g_j$. There are two consequences of the geometric interpretation of affine growth diagrams. For any affine growth diagram the rows are periodic with period $n$ and the dual symmetry statement $\gamma_{i,j}^*=\gamma_{j,i+n}$ follows from $d(g_i,g_j)=d(g_j,g_i)^*$. This is Corollary \ref{cor:symmetry}. We apply these two facts in \S \ref{combinatorics} to study the combinatorics of these diagrams. 

With some work one can see that the first part of the theorem is implicit in the geometric arguments of Fontaine and Kamnitzer \cite{FK}. Their results however are stated in the setting of the cyclic sieving phenomenon and are formulated in terms of the action of a rotation operator $R:\text{Poly}(\lambda^1,\ldots,\lambda^n)\rightarrow \text{Poly}(\lambda^2,\ldots,\lambda^n,\lambda^1)$ on the Satake basis. We are instead concerned with the components of a fixed polygon space Poly$(\vec\lambda)$ and give a proof of the above theorem using Knutson--Tao hives. The argument is entirely combinatorial once we recall the result of Goncharov--Shen \cite{GS} that gives a bijection between the components of Poly$(\vec\lambda)$ and hives of type $\vec\lambda$. 

Another motivation for understanding the combinatorics of the weight-valued distances is the work of Fontaine, Kamnitzer, and Kuperberg \cite{FKK}. They study the relation between the Satake basis and Kuperberg's basis of non-elliptic webs \cite{Kup}, which is a different basis for the invariant space $(V_{\lambda^1}\otimes \cdots \otimes V_{\lambda^n})^{GL_3}$ for $\lambda^i$ minuscule. They view an element of $Gr$ as a vertex of the corresponding Bruhat--Tits building, an infinite dimensional simplicial complex. They then use the CAT(0) geometry of the Bruhat--Tits building and embeddings of polygon configurations to show that the change of basis matrix between the Satake basis and the non-elliptic web basis is upper-unitriangular. We hope that understanding the metric properties of polygon configurations in a rotation-invariant manner can help understand Kuperberg webs in higher rank, along the lines of \cite{FKK}. Note that in higher rank, there is no known, rotation-invariant construction of minimal web bases. See \cite{Fon} for a construction of a web basis that is not rotation invariant, nor minimal. 

We briefly sketch the proof of the main theorem of \S 2. Goncharov and Shen \cite{GS} show that components of the polygon space Poly$(\vec\lambda)$ are in bijection with $n$-hives of type $\vec\lambda$ via an explicit constructible function defined by Kamnitzer in \cite{Kam}. An $n$-hive is an assignment of integers to the following set of lattice points in an $(n-1)$-dimensional tetrahedron of size $m$ with some additional conditions.
\begin{align*}
\Delta_m^n=\{(i_1,\ldots,i_n)\in \mathbb Z_{\geq 0}^n\mid i_1+\cdots i_n=m\}
\end{align*}
The precise conditions satisfied by the integers are given in \S\ref{sec:hives}, one of which is the octahedron recurrence for every unit octahedron \cite{RR}, as depicted in Figure \ref{fig:octahedron}. In \S\ref{sec:Kamnitzer function} we recall Kamnitzer's constructible function from the polygon space Poly$(\vec\lambda)$ to $\mathbb Z^{\Delta_m^n}$. It was conjectured in \cite{Kam} that the generic value of this function on each component of the Satake fiber is a Knutson--Tao hive of type $\vec\lambda$ and that this is a bijection. This was subsequently proved in \cite{GS}, building on the work of \cite{FG}. 

\begin{figure}[b]
\begin{center}
\begin{tikzpicture}[line join=bevel,z=-7]
\coordinate (A1) at (0,0,-1);
\coordinate (A2) at (-1,0,0);
\coordinate (A3) at (0,0,1);
\coordinate (A4) at (1,0,0);
\coordinate (B1) at (0,1,0);
\coordinate (C1) at (0,-1,0);

\node at (A1) {b};
\node[left] at (A2) {a};
\node[below,left] at (A3) {d};
\node[right] at (A4) {c};
\node[above] at (B1) {e};
\node[below] at (C1) {f=max(a+c,b+d)-e};
\draw[dotted] (A1) -- (A2) -- (B1) -- cycle;
\draw[dotted] (A4) -- (A1) -- (B1) -- cycle;
\draw[dotted] (A1) -- (A2) -- (C1) -- cycle;
\draw[dotted] (A4) -- (A1) -- (C1) -- cycle;
\draw (A2) -- (A3) -- (B1) -- cycle;
\draw (A3) -- (A4) -- (B1) -- cycle;
\draw (A2) -- (A3) -- (C1) -- cycle;
\draw (A3) -- (A4) -- (C1) -- cycle;
\end{tikzpicture}
\caption{The octahedron recurrence. \label{fig:octahedron}}
\end{center}
\end{figure}

Using Kamnitzer's function we observe in Proposition \ref{prop:hive edges} (also noted in \cite{LO}) that the differences of consecutive labels along the edges of an $n$-hive give the generic distances $d(g_i,g_j)$ of the corresponding component. By edge label we mean a hive value for an index $(i_1,\ldots,i_n)$ with at most two nonzero entries. To prove the main theorem it then suffices to analyze the edges of the corresponding $n$-hives. It is immediate from the definition of Poly$(\vec\lambda)$ that $d(g_i,g_{i+1})=\lambda^i$ for all $i$. That the differences $d(g_i,g_j)-d(g_i,g_{j+1})$ and $d(g_i,g_j)-d(g_{i+1},g_j)$ are vertical strips also follows easily from the hive interpretation. 

The main technical part is in proving that the distances $d(g_i,g_j)$ satisfy the local condition (\ref{local condition}). To do so it suffices to analyze $4$-subhives that have two opposing minuscule edges. This is done by repeated application of the octahedron recurrence to determine the values on the bottom two faces of the $4$-hive given the value on the top two faces. See Figures \ref{quadrilateral} and \ref{fig:excavate}. The same method is used in \cite{KT, HK1, Zin}. The proof is elementary and combinatorial, but somewhat involved so is relegated to \S \ref{main proof}. See also \cite{HK2}, where Henriques and Kamnitzer define a tensor category in terms of hives and show that it is equivalent to the category of $\mathfrak{gl}_n$-crystals.

To prove bijectivity we show that the weights $\gamma_{1,i}=d(g_1,g_i)$ for $1\leq i\leq n$ given by the edge labels of the corresponding $n$-hive form a minuscule path of type $\vec\lambda$, and are enough to recover the entire $n$-hive. A \emph{minuscule path} of type $\vec\lambda$ is a sequence of dominant weights $(\mu^1=\vec{0}, \mu^2,\ldots,\mu^{n},\mu^{n+1}=\vec{0})$ such that $\mu^i-\mu^{i-1}=w\cdot\lambda^i$ for some $w\in S_n$. It is well known that the number of minuscule paths of type $\vec\lambda$ is equal to $\dim(V_{\lambda^1}\otimes\cdots\otimes V_{\lambda^n})^{GL_m}$. These weights label the vertices on the first horizontal line of the corresponding affine growth diagram. We show that the first line of any affine growth diagram of type $\vec\lambda$ must form a minuscule path of type $\vec\lambda$. We then show that the local condition is reversible, meaning that $\gamma_{i,j}=\text{sort}(\gamma_{i+1,j}+\gamma_{i,j+1}-\gamma_{i+1,j+1})$ (this was already proved in \cite{vLee2}). This implies that an affine growth diagram can be recovered from the first horizontal line by application of the local condition to the southeast and to the northwest. This establishes bijectivity between affine growth diagrams of type $\vec\lambda$ and components of Poly$(\vec\lambda)$. 

It was already shown in \cite{FKK} by geometric arguments that the components of Poly$(\vec\lambda)$ can be indexed by minuscule paths by measuring the distances $d(g_1,g_i)$ for all $1\leq i\leq n$. More generally, let $(i_1,j_1),\ldots,(i_{n+1},j_{n+1})$ be a path in $St_n$ such that $(i_1,j_1)=(i,i)$ for some $i$ and $(i_k,j_k)-(i_{k+1},j_{k+1})=(1,0)$ or $(0,-1)$ for all $k$. This is a path through $St_n$ consisting of eastward and northward steps. Then an affine growth diagram can be recovered from the weights $\gamma_{i_k,j_k}$ along this path. Hence, any such path specifies a way to index the components of Poly$(\vec\lambda)$. In this sense affine growth diagrams are overdetermined, but are rotationally invariant since they measure all pairwise distances. We give a proof of bijectivity using the $n$-hive as discussed rather than appealing to \cite{FKK}.

As mentioned, Fontaine and Kamnitzer \cite{FK} also studied rotation of the Satake basis in the context of the cyclic sieving phenomenon. Given a sequence of minuscule weights $\vec\lambda=(\lambda^1,\ldots,\lambda^n)$ for any complex semisimple $G$, define the rotated sequences by $\lambda^{(i)}=(\lambda_{1+i},\ldots,\lambda_n,\lambda_1,\ldots,\lambda_i)$. They defined a geometric rotation $R:\text{Poly}(\vec\lambda)\rightarrow \text{Poly}(\vec\lambda^{(1)})$ that takes components to components. They also define a rotation on minuscule paths, so that for a minuscule path $\vec\mu$ of type $\vec\lambda$, the minuscule path $R(\mu)$ is of type $\lambda^{(1)}$. Given a component $Z_{\vec\mu}$ of Poly$(\vec\lambda)$ indexed by a minuscule path $\vec\mu$ of type $\vec\lambda$, they show that $R(Z_{\vec\mu})=Z_{R(\vec\mu)}$. In the present $GL_m$ setting, applying the local condition at a single unit square to the a minuscule path labelling the first line of an affine growth diagram can be seen as a refinement of the procedure given in \cite{FK}. They observe that when $G=SL_n$ a minuscule path can be interpreted as a rectangular column-strict tableau. Under this equivalence the rotation of minuscule paths in \cite{FK} is promotion on rectangular tableaux. See \S\ref{sec:crystals} for an interpretation in terms of crystals. 

Polygon configurations also play a role in the study of higher laminations and higher Teichm\"uller spaces of Fock--Goncharov  \cite{FG}. Polygon configurations in the affine Grassmannian can be thought of as the tropicalization of the configuration space of principal affine flags. Certain positive configurations are used to define higher laminations, as described in \cite{GS}, and \cite{Le1}. We only briefly discuss this point of view as motivation.

Let $\mathcal A=GL_m/U$ be the variety of principal flags where $U$ is the subgroup of unipotent upper-triangular matrices. Such a flag is given by an ordered basis $v_1,\ldots,v_m$ together with volume forms $v_1\wedge\cdots v_k$, such that the span of $v_1,\ldots, v_k$ is the $k$-dimensional subspace of the flag and $v_1\wedge\cdots v_k$ is a volume form on the subspace spanned by $v_1,\ldots, v_k$. Two bases represent the same flag if they give the same $k$-forms for all $k\leq m$. Consider the configuration space of $n$ principal flags up to left diagonal action, $Conf_n(\mathcal A)=G\backslash(G/U)^n$. Let $F_1,\ldots,F_n$ be flags where $F_i$ has basis $v_{i1},\ldots,v_{im}$. For each $(i_1,\ldots,i_n)$ define the Fock--Goncharov functions
\begin{align*}
h_{i_1,\ldots,i_n}(F_1,\ldots,F_n)=\det(v_{11},\ldots,v_{1i_1},\ldots,v_{n1},\ldots,v_{ni_n}).
\end{align*}
If all but three of the indices $i_j$ are zero, then the function depends on three flags, and such a function is called a face function. 

It was shown in \cite{FG} that $Conf_n(\mathcal A)$ has a cluster structure. Consider the flags as labelling the vertices of an $n$-gon and fix a triangulation of the $n$-gon into $n-2$ triangles. The face functions corresponding to the triangles of the fixed triangulation give coordinates on $Conf_n(\mathcal A)$ that form a cluster. Clusters corresponding to different triangulations can be reached by certain sequences of mutations, each of which is an application of the (nontropical) octahedron recurrence (see \cite{FG, FL} for a description of the associated quiver). Kamnitzer's function can be viewed as a tropicalization of the Fock--Goncharov coordinates on the configuration space of $n$ principal flags. Our proof technique to establish that the local condition holds in Theorem A can be interpreted as an explicit analysis of a sequence of cluster mutations in the tropical setting. 

Higher laminations are studied in  \cite{Le2, LO} by interpreting the tropicalized functions in terms of the geometry of minimal spanning networks of polygon configurations in the affine building. See \cite{Le1, Le2, LO, FL} for further details on this perspective. 

\subsection{Classical Growth Diagrams and the Robinson--Schensted Correspondence}
\S \ref{combinatorics} concerns the purely combinatorial study of affine growth diagrams. The first main result of \S \ref{combinatorics} is that the classical Fomin growth diagrams, and hence the Robinson--Schensted correspondence, appear within affine growth diagrams when $\vec\lambda=(\omega_1,\ldots,\omega_1,\omega_1^*,\ldots,\omega_1^*)$. Recall that the \emph{Robinson--Schensted correspondence} is a bijection between permutations in $S_k$ and pairs of same-shape standard Young tableaux with $k$ boxes. 

Fomin growth diagrams are one of many ways to realize the Robinson--Schensted correspondence and are defined as follows (see for example \cite{Sag}). For a permutation $\pi \in S_k$ consider the corresponding $k\times k$ permutation matrix with the $i$th row containing a $1$ in position $\pi(i)$. View this permutation matrix inside of a $k\times k$ grid of squares. Each vertex of the grid will be labelled by a partition, which we identify with its Young diagram. Begin by labelling the vertices along the top and left boundary with the empty Young diagram. The remaining vertex labels are filled in one-by-one from northwest to southeast according to the following local rules applied at each unit square. Let $\alpha,\beta,\gamma$ be three of the partitions labelling a unit square as in the following diagram. 
\begin{align}
\begin{array}{|cc|c}
\cline{1-2}
\alpha&&\beta\\
&  & \\
\cline{1-2}
 \multicolumn{1}{c}{\gamma}&\multicolumn{2}{r}{\delta} \label{square setup}
\end{array}
\end{align}

Then $\delta$ is determined from $\alpha,\beta,\gamma$ according to the following local rules \cite{Fom}. These local rules will not be very important for our purposes, as we will recall a more global interpretation, but we list them for completeness. Here $\alpha\prec \beta$ means that $\alpha$ precedes $\beta$ in Young's lattice of partitions.
\begin{Rem}
Since affine growth diagrams are defined in terms of vertical strips, the local rules that we state below are transpose of the usual rules for Fomin growth diagrams. The rules stated here give the column-insertion Robinson--Schensted correspondence.
\end{Rem}
\begin{enumerate}
\item
If $\beta\not=\gamma$, then $\delta=\beta\cup\gamma$.
\item
If $\alpha \prec \beta=\gamma$, then $\beta$ must have been obtained from $\alpha$ by adding a box in some column $i$. Let $\delta$ be obtained from $\beta=\gamma$ by adding a box in column $i+1$.
\item
If $\alpha=\beta=\gamma$, then let $\delta$ be obtained from $\alpha$ by adding a box in the first column if the current square contains a $1$, and otherwise let $\delta=\alpha$.
\end{enumerate}
Reading the partitions left to right along the bottom row and down the right column gives a pair of chains of partitions ending at a common Young diagram. This information is equivalent to a pair of same-shape standard Young tableaux. Conversely, starting with such a pair labelling an empty $k\times k$ grid, the diagram can be filled-in from southeast to northwest, one square at a time by reversing the local rules, and hence revealing the $1$ entries of the corresponding permutation matrix.

Here is a full example with $\pi=1423$ corresponding to $P(\pi)=\scalebox{.6}{\begin{ytableau}1&4\\2\\3\end{ytableau}},Q(\pi)=\scalebox{.6}{\begin{ytableau}1&3\\2\\4\end{ytableau}}$.
\begin{align*}
\scalemath{.5}{
\Ylinethick1.2pt
\begin{array}{|cc|cc|cc|cc|c}
\cline{1-8}
 \emptyset&& \emptyset&   & \emptyset& & \emptyset& &\emptyset\\
  &1& & & & & & & \\
\cline{1-8}
  \emptyset&&&&&&&&\\
    && & & & & & 1& \\
\cline{1-8}
 \emptyset&&&&&&&&\\
   && & 1& & & & & \\
\cline{1-8}
 \emptyset&&&&&&&&\\
  && & & & 1& & & \\
\cline{1-8}
 \multicolumn{1}{c}{\emptyset}& \multicolumn{2}{r}{}& \multicolumn{2}{r}{}& \multicolumn{2}{r}{}& \multicolumn{2}{r}{}
\end{array}
\hspace{8mm}%%
\begin{array}{|cc|cc|cc|cc|c}
\cline{1-8}
 \emptyset&& \emptyset&   & \emptyset& & \emptyset& &\emptyset\\
  &1& & & & & & & \\
\cline{1-8}
  \emptyset&&\scalebox{.5}{\yng(1)}&&\scalebox{.5}{\yng(1)}&&\scalebox{.5}{\yng(1)}&&\scalebox{.5}{\yng(1)}\\
    && & & & & & 1& \\
\cline{1-8}
 \emptyset&&&&&&&&\\
   && & 1& & & & & \\
\cline{1-8}
 \emptyset&&&&&&&&\\
  && & & & 1& & & \\
\cline{1-8}
 \multicolumn{1}{c}{\emptyset}& \multicolumn{2}{r}{}& \multicolumn{2}{r}{}& \multicolumn{2}{r}{}& \multicolumn{2}{r}{}
\end{array}
\hspace{8mm}%%
\begin{array}{|cc|cc|cc|cc|c}
\cline{1-8}
 \emptyset&& \emptyset&   & \emptyset& & \emptyset& &\emptyset\\
  &1& & & & & & & \\
\cline{1-8}
  \emptyset&&\scalebox{.5}{\yng(1)}&&\scalebox{.5}{\yng(1)}&&\scalebox{.5}{\yng(1)}&&\scalebox{.5}{\yng(1)}\\
    && & & & & & 1& \\
\cline{1-8}
 \emptyset&&\scalebox{.5}{\yng(1)}&&\scalebox{.5}{\yng(1)}&&\scalebox{.5}{\yng(1)}&&\scalebox{.5}{\yng(1,1)}\\
   && & 1& & & & & \\
\cline{1-8}
 \emptyset&&&&&&&&\\
  && & & & 1& & & \\
\cline{1-8}
 \multicolumn{1}{c}{\emptyset}& \multicolumn{2}{r}{}& \multicolumn{2}{r}{}& \multicolumn{2}{r}{}& \multicolumn{2}{r}{}
\end{array}
\hspace{8mm}%%
\begin{array}{|cc|cc|cc|cc|c}
\cline{1-8}
 \emptyset&& \emptyset&   & \emptyset& & \emptyset& &\emptyset\\
  &1& & & & & & & \\
\cline{1-8}
  \emptyset&&\scalebox{.5}{\yng(1)}&&\scalebox{.5}{\yng(1)}&&\scalebox{.5}{\yng(1)}&&\scalebox{.5}{\yng(1)}\\
    && & & & & & 1& \\
\cline{1-8}
 \emptyset&&\scalebox{.5}{\yng(1)}&&\scalebox{.5}{\yng(1)}&&\scalebox{.5}{\yng(1)}&&\scalebox{.5}{\yng(1,1)}\\
   && & 1& & & & & \\
\cline{1-8}
 \emptyset&&\scalebox{.5}{\yng(1)}&&\scalebox{.5}{\yng(1,1)}&&\scalebox{.5}{\yng(1,1)}&&\scalebox{.5}{\yng(2,1)}\\
  && & & & 1& & & \\
\cline{1-8}
 \multicolumn{1}{c}{\emptyset}& \multicolumn{2}{r}{}& \multicolumn{2}{r}{}& \multicolumn{2}{r}{}& \multicolumn{2}{r}{}
\end{array}
\hspace{8mm}%%
\begin{array}{|cc|cc|cc|cc|c}
\cline{1-8}
 \emptyset&& \emptyset&   & \emptyset& & \emptyset& &\emptyset\\
  &1& & & & & & & \\
\cline{1-8}
  \emptyset&&\scalebox{.5}{\yng(1)}&&\scalebox{.5}{\yng(1)}&&\scalebox{.5}{\yng(1)}&&\scalebox{.5}{\yng(1)}\\
    && & & & & & 1& \\
\cline{1-8}
 \emptyset&&\scalebox{.5}{\yng(1)}&&\scalebox{.5}{\yng(1)}&&\scalebox{.5}{\yng(1)}&&\scalebox{.5}{\yng(1,1)}\\
   && & 1& & & & & \\
\cline{1-8}
 \emptyset&&\scalebox{.5}{\yng(1)}&&\scalebox{.5}{\yng(1,1)}&&\scalebox{.5}{\yng(1,1)}&&\scalebox{.5}{\yng(2,1)}\\
  && & & & 1& & & \\
\cline{1-8}
 \multicolumn{1}{c}{\emptyset}& \multicolumn{2}{r}{\scalebox{.5}{\yng(1)}}& \multicolumn{2}{r}{\scalebox{.5}{\yng(1,1)}}& \multicolumn{2}{r}{\scalebox{.5}{\yng(1,1,1)}}& \multicolumn{2}{r}{\scalebox{.5}{\yng(2,1,1)}}
\end{array}}%\label{Fomin growth}
\end{align*}

The main result of \S\ref{combinatorics} is Theorem \ref{thm:Fomin contained}, reproduced here. For a dominant weight $\gamma$, let $\gamma^+$ (resp. $\gamma^-$) be the weight consisting of the positive (resp. negative) entries of $\gamma$. For example, for the dominant weight $\gamma=(3,1,1,0,0,-2,-4)$, we have $\gamma^+=(3,1,1,0,0,0,0)$ and $\gamma^-=(0,0,0,0,0,-2,-4)$. Interpret each as a partition. See Figure \ref{fig:RS} for an example of this theorem for $k=3$.
\begin{Thm**}
Let $n=2k$ and $\vec\lambda=(\omega_1,\ldots,\omega_1,\omega_1^*,\ldots,\omega_1^*)$ where $\omega_1$ and $\omega_1^*$ each appear $k$ times. Let $\{\gamma_{i,j}\}_{(i,j)\in St_n}$ be an affine growth diagram of type $\vec\lambda$ for $m\geq k$. Then the partitions $\gamma_{i,j}^+$ for $1\leq i\leq k+1$ and $k+1\leq j\leq n+1$ form a Fomin growth diagram, growing from the southeast to the northwest, and the partitions $\gamma_{i,j}^-$ for the same indices form a Fomin growth diagram, growing from the northwest to the southeast. Similarly for the indices $k+1 \leq i\leq n+1$ and $n+1\leq j\leq n+k+1$ the partitions $\gamma_{i,j}^+$ form a Fomin growth diagram from the northwest to southeast and the $\gamma_{i,j}^-$ from southeast to northwest.
\end{Thm**} 

To prove this theorem we actually prove something stronger. Recall the following Sundaram bijection, attributed to Stanley in \cite{Sun}, and also studied in \cite{Rob}. An \emph{oscillating tableaux} is a sequence of partitions $\left(\mu^1=\vec{0},\mu^2,\ldots,\mu^{n},\mu^{n+1}=\vec{0}\right)$ such that for all $i$ the partitions $\mu^i$ and $\mu^{i+1}$ differ by exactly one box.

\begin{Thm}[\cite{Sun, Rob}]\label{thm:Sundaram}
For $n$ even, there is a bijection between fixed-point-free involutions in $S_n$ and length $n$ oscillating tableaux.
\end{Thm}

Fixed-point-free involutions in $S_n$ can be viewed as $n\times n$ symmetric permutation matrices with zeros on the diagonal. If an oscillating tableau $\left(\mu^1=\vec{0},\mu^2,\ldots,\mu^{n},\mu^{n+1}=\vec{0}\right)$ increases for the first $k=n/2$ steps, then the pair of chains $(\vec{0}=\mu^1,\ldots, \mu^{k+1})$ and $(\vec{0}=\mu^{m+1},\ldots,\mu^{k+1})$ can be interpreted as a pair of standard Young tableau of same shape $\mu^k$. Restricted to such oscillating tableaux, the Sundaram bijection maps them to permutations of $S_k$ embedded in $n\times n$ symmetric permutation matrices as the $k\times k$ northeastern submatrix. 
\begin{Rem}
This is not exactly the RS-correspondence. For $\pi\in S_k$, if $\pi$ corresponds to $ \left(P(\pi),Q(\pi)\right)$ under the column-insertion RS-correspondence, then the Sundaram bijection sends $\pi$ (thought of as embedded in an $n\times n$ symmetric fixed-point-free matrix) to $\left(ev(P(\pi)),Q(\pi)\right)$ (thought of as an oscillating tableu) where $ev(P(\pi))$ is the evacuation tableau of $P(\pi)$.
\end{Rem}

In \S\ref{sec:combinatorial background} we recall Greene's theorem, which assigns a partition to a partially ordered set. In our setting this will be applied to partially ordered sets defined by partial permutation matrices. We then recall the interpretation of the partitions in Fomin growth diagrams in terms of Greene's theorem.

In \S\ref{sec:filling} we give a procedure for assigning a natural number to each unit square of an affine growth diagram without any restrictions on $\vec\lambda$. We then specialize to the case that each $\lambda^i$ is $\omega_1$ or $\omega_1^*$ in which case the entries are all $0$ and $1$. When $m$ is large enough for fixed $n$, the periodicity of the diagram and dual symmetry imply that the resulting $0,1$ entries can be interpreted as an $n\times n$ symmetric permutation matrix. 

In \S\ref{sec:main combinatorial theorem}, starting with such a matrix, we give a procedure to construct an affine growth diagram whose first line is labelled by an oscillating tableaux. The construction defines the weight $\gamma_{i,j}$ at each vertex of $St_n$ by adding a positive partition $\gamma_{i,j}^+$ and a negative partition $\gamma_{i,j}^-$ where the negative partition is thought of as a dominant weight with negative entries. The two partitions $\gamma_{i,j}^+$ and $\gamma_{i,j}^-$ are defined by applying Greene's theorem to two certain submatrices of the staircase diagram depending on the vertex $(i,j)$. This construction may be of independent interest. 

Finally, we show that these two constructions are inverses when $m\geq n/2$, thereby recovering the Sundaram bijection. This is Theorem \ref{thm:bijection}. Theorem \ref{thm:Fomin contained} follows since Fomin diagrams can be interpreted using Greene's theorem.

Recall that the \emph{Robinson--Schensted--Knuth correspondence} is a generalization of the Robinson--Schensted correspondence. It is a bijection between $n\times n$ matrices with natural-number entries and pairs of same-shape semistandard Young tableaux with entries from $[n]$. Fomin diagrams also have a generalization to this setting in the sense of Roby \cite{Rob}. 

In \S\ref{sec:RSK} we drop the assumption that each $\lambda^i$ is $\omega_1$ or $\omega_1^*$ to show that affine growth diagrams also contain these generalized Fomin growth diagrams when $\vec\lambda$ consists of a sequence of fundamental weights, followed by dual-fundamental weights, thereby recovering the full RSK-correspondence. This is Theorem \ref{thm:generalized Fomin}, which we don't reproduce here. As with the RS-correspondence, this will be a special case of the most general setting where there are no restrictions on $\vec\lambda$, other than $m$ being large enough with respect to fixed $\vec\lambda$. 

\subsection{Relation to Crystals}\label{sec:crystals}
The local condition (\ref{local condition}) arises in the context of crystals, first appearing in \cite{vLee2} where it is called the local move. For now let $\mathfrak{g}$ be a semisimple finite-dimensional complex Lie algebra. Recall that for any two $\mathfrak{g}$-crystals $B$ and $C$ if $b\otimes c$ is a highest (resp. lowest) weight crystal element in $B\otimes C$, then $b$ is highest weight in $B$ (resp. $c$ is lowest weight in C). Henriques and Kamnitzer \cite{HK1} define a commutor $\sigma_{B,C}:B\otimes C\rightarrow C\otimes B$ for any two crystals $B$ and $C$ by $\sigma_{B,C}(b\otimes c)=\xi(\xi(c)\otimes \xi(b))$. Here $\xi:B\rightarrow B$ is the Lusztig involution such that such that $\xi(b)$ is highest (resp. lowest) weight if $b$ is lowest (highest) weight and $\xi(b)$. For the tableaux model of $\mathfrak{sl}_m$-crystals this is just Sch\"utzenberger evacuation on tableaux. 

Let $B(\lambda)$ be the highest weight $\mathfrak{g}$-crystal of highest weight $\lambda$. Recall that the Littelmann path model \cite{Lit} is a model for crystals. In this setting, van Leeuwen \cite{vLee2} realized the commutor $\sigma_{B(\lambda),B(\mu)}:B(\lambda)\otimes B(\mu)\rightarrow B(\mu)\otimes B(\lambda)$ as an iterated application of the local move by embedding $B(\lambda)$ and $B(\mu)$ into tensor products of minuscule highest weight crystals. For weights $\lambda,\mu,\nu,\rho$ for general $\mathfrak{g}$ the local rule is $\rho=\text{dom}_W(\lambda+\mu-\nu)$ where the operator dom$_W$ returns the dominant representative in the Weyl group orbit. For $\mathfrak{g}=\mathfrak{sl}_m$ Littelmann paths can be identified with tableaux, in which case the local rule is \emph{jeu de taquin}. See \cite{Len1,Len2} for a nice discussion of the commutor and local moves.

For the present context we briefly summarize the discussion in \cite[section 2.3]{FK}. For a sequence of minuscule weights let $B(\vec\lambda)=B(\lambda^1)\otimes\cdots\otimes B(\lambda^n)$. Let $B(\vec\lambda)^{\mathfrak g}$ be the set of crystal elements that are both highest weight and lowest weight, which is equal to the set of minuscule paths of type $\vec\lambda$. For $b_1\otimes\cdots\otimes b_n\in \left(B(\lambda^1)\otimes \left(B(\lambda^2)\otimes \cdots \otimes B(\lambda^n)\right)\right)^{\mathfrak{g}}$, we have that $b_1$ is highest weight and $b_2\otimes\cdots\otimes b_n$ is lowest weight. Fontaine and Kamnitzer \cite{FK} define a rotation operation on $B(\vec\lambda)^{\mathfrak{g}}$ by $R(b_1\otimes \cdots\otimes b_n)=\xi(b_2\otimes\cdots\otimes b_n)\otimes \xi(b_1)$. The rotation map can also be written as $R(b_1\otimes\cdots\otimes b_n)=\sigma_{B(\lambda^1),B(\lambda^2)\otimes\cdots\otimes B(\lambda^n)}(b_1\otimes\left(b_2\otimes\cdots\otimes b_n\right))$. They then show that the bijection between minuscule paths and $B(\vec\lambda)^{\mathfrak g}$ intertwines the rotation of minuscule paths and crystal elements. They finally observe that in the $\mathfrak{sl}_m$ case, minuscule paths are in bijection with row-strict semistandard tableaux and the rotation operation becomes promotion of tableaux. 

The promotion operation can be decomposed into a sequence of Bender--Knuth operations and this is precisely the local move at each unit square. See \cite{Wes} for a complete discussion of the commutor, local moves, crystals, and their relation to growth diagrams. We have chosen to focus on $\mathfrak{gl}_m$ minuscule paths rather than tableaux since this leads to the classical Fomin growth diagrams of \S\ref{combinatorics}. 

As mentioned previously, \cite{HK2} define a coboundary tensor category in terms of hives with an associator and commutor defined in terms of a modified octahedron recurrence. They then give an equivalence with the tensor category of $\mathfrak{gl}_m$-crystals that respects the associator and commutor. Although we have not done so, it should be possible to establish Theorem A by following the Littelmann path model for crystals along this equivalence of categories, taking their definition of the commutor in terms of the modified octahedron recurrence, and applying the decomposition of the commutor in terms of local moves as in \cite{vLee2}. 

\subsection{Acknowledgements}
I would like to thank Allen Knutson for suggesting the question that led to this work and for many helpful discussions. I would also like to thank Pasha Pylyavskyy, David Speyer and Alex Yong for helpful conversations. Finally, I would like to thank Bruce Westbury for pointing out the occurrence of the local condition in \cite{vLee2}. I thank the anonymous referee for a careful reading of this paper. An extended abstract of these results appeared in \cite{Akh}.

\section{Polygon Configurations in the Affine Grassmannian \label{geometry}}
In \S2.1 we recall the definitions of the affine Grassmannian and polygon spaces. In \S2.2 we recall the definition of Knutson--Tao hives. In \S2.3 we define Kamnitzer's function and state the result of \cite{GS} that this function gives a bijection between hives and components of the polygon space. We then show in Proposition \ref{prop:hive edges} that the edge values of the hives give the distances $d(g_i,g_j)$, thereby establishing a link between the geometry of the polygon space and the combinatorics of hives. In \S2.4 we use this result to prove our main theorem, the bijection between affine growth diagrams and components of the polygon space for minuscule weights. 

\subsection{The Affine Grassmannian and Polygon Spaces \label{background}}
Let $\mathcal K=\mathbb C((t))$ and $\mathcal O=\mathbb C[[t]]$. The \emph{affine Grassmannian} is the quotient
\begin{align*}
Gr=GL_m(\mathcal K)/GL_m(\mathcal O)
\end{align*}
and has the structure of an ind-variety. The points of $Gr$ correspond to finitely generated rank-$m$ $\mathcal O$-submodules in $\mathcal K^m$ which are also called lattices. The group $GL_m(\mathcal K)$ acts on the space of lattices and the stabilizer of any lattice is isomorphic to $GL_m(\mathcal O)$. See \cite{AG2} for a nice exposition of this interpretation of the points as lattices.

The affine Grassmannian has a coweight-valued metric, defined as follows. Recall that the coweight lattice of $GL_m$ is $\Lambda=\Hom(\mathbb G_m,T)\cong \mathbb Z^m$ where $T$ is a maximal torus. A coweight $\mu=(\mu_1,\ldots,\mu_m)$ in $\Lambda$ corresponds to the element $\mu(t)$ in $T(\mathcal K)\subset GL_m(\mathcal K)$ that has $t^{\mu_i}$ along the diagonal. Let $t^\mu$ denote the projection of $\mu(t)$ to the quotient $Gr$. The cone of dominant coweights is $\Lambda_+=\{(\lambda_1,\ldots,\lambda_m)\in \mathbb Z^m\mid \lambda_1\geq \cdots \geq \lambda_m\}$. The double cosets 
\begin{align*}
GL_m(\mathcal O)\backslash GL_m(\mathcal K)/GL_m(\mathcal O)
\end{align*}
are indexed by dominant coweights of $GL_m$, that is, the elements $t^\mu$ for $\mu\in \Lambda_+$ form representatives for the left cosets of $GL_m(\mathcal O)$ on $Gr$. The orbits $GL_m(\mathcal O)t^\mu$ for $\mu\in \Lambda_+$ form a stratification of $Gr$.

For any group $G$ and subgroup $H$, the left cosets on $G/H$ are in bijection with the cosets of $G$ on $G/H\times G/H$. In the present setup of $G=GL_m(\mathcal K)$ and $H=GL_m(\mathcal O)$ this implies that any pair of points $(p,q)$ in $Gr$ can be moved via the $GL_m(\mathcal K)$ action to the pair $(t^0,t^\lambda)$ for a unique dominant coweight $\lambda$. In this case, the distance from $p$ to $q$ is defined to be $\lambda$, denoted $d(p,q)=\lambda$. The distance from $q$ to $p$ is given by the dual coweight, $d(q,p)=\lambda^*$, where the dual of a coweight $\lambda=(\lambda_1,\ldots,\lambda_m)$ is the negation of the reversal, that is $\lambda^*=(-\lambda_m,\ldots,-\lambda_1)$. We have stated everything in terms of $GL_m$ for simplicity, but the above definitions and the following theorem hold for any complex connected reductive group $G$. Note that we will identify the weight and coweight lattices of $GL_m$, so in the remainder of the paper we speak of $GL_m$ weights.

The polygon variety arises in the geometric Satake correspondence \cite{Lus}, \cite{Gin}, \cite{BD}, \cite{MV}, which we briefly recall now for motivation even though we do not make explicit use of these results. Let $G$ be a connected reductive group over $\mathbb C$ and $Gr_G$ the affine Grassmannian for $G$. The $G(\mathcal O)$-orbits on $Gr_G$ correspond to dominant coweights of $G$ and form a stratification of $Gr_G$. Let Perv$_{G(\mathcal O)}(Gr_G)$ denote the category of $G(\mathcal O)$-equivariant perverse sheaves on $Gr_G$.

\begin{Thm}[\cite{Lus,Gin,BD,MV} \label{geometric Satake}]
Let $G$ be a connected reductive group over $\mathbb C$. The tensor category Perv$_{G(\mathcal O)}(Gr_G)$ is equivalent to the tensor category of representations of $G^L$.
\end{Thm}
Here $G^L$ denotes the Langlands dual group of $G$, but for the case we are interested in, $G^L=GL_m$. Let $\vec\lambda=(\lambda^1,\ldots,\lambda^n)$ be any sequence of dominant coweights of $G$. The \emph{polygon space with side lengths $\vec\lambda$} is defined as
\begin{align*}
\text{Poly}(\vec\lambda)=\{(g_1=g_{n+1},\ldots,g_n)\in Gr_G^n \mid d(g_i,g_{i+1})=\lambda^i, g_1=t^0=g_{n+1}\}.
\end{align*}
We will only be concerned with the case of $G=GL_m$ and the $\lambda^i$ minuscule. It was shown in \cite{Hai} that this variety is equidimensional for $G=GL_m$ (see also \cite{Hai2} for the minuscule case for general $G$). See \cite[Lemma 4.6]{HS} for a proof of the following corollary. 
\begin{Cor}\label{cor:Satake}
Under the equivalence of the previous theorem, there is an isomorphism $\left(V_{\lambda^1}\otimes \cdots \otimes V_{\lambda^n}\right)^{G^L}\cong H_{\text{top}}(\text{Poly}(\vec\lambda))$ where $H_{\text{top}}(\text{Poly}(\vec\lambda))$ is the top Borel--Moore homology of Poly$(\vec\lambda)$. Hence, the set of top components of Poly$(\vec\lambda)$ give a basis for $\left(V_{\lambda^1}\otimes \cdots V_{\lambda^n}\right)^G$.
\end{Cor}

\subsection{Knutson--Tao Hives \label{sec:hives}}
In this section we recall the definition of $n$-hives. Consider the set 
\begin{align*}
\Delta_m^3=\{(i,j,k)\in \mathbb Z_{\geq0}^3\mid i+j+k=m\}, 
\end{align*}
a triangle of lattice points of size $m$. 
\begin{Def}
A \emph{hive}, or \emph{3-hive}, is an assignment $f_{i,j,k}$ of an integer to each point of $\Delta_m^3$ that satisfies the following rhombus inequalities, 
\begin{align*}
f_{i,j,k}+f_{i,j+1,k-1}&\geq f_{i+1,j,k-1}+f_{i-1,j+1,k}\\
f_{i,j,k}+f_{i+1,j-1,k}&\geq f_{i+1,j,k-1}+f_{i,j-1,k+1}\\
f_{i,j,k}+f_{i+1,j,k-1}&\geq f_{i,j+1,k-1}+f_{i+1,j-1,k},
\end{align*}
for all $(i,j,k)$ in $\Delta_m^3$. There are three types of unit rhombi in $\Delta_m^3$, vertical, slanted to the left, and slanted to the right. Each inequality corresponds to a unit rhombus in the triangular lattice, and it says that the sum across the short diagonal is at least the sum across the long diagonal. Two hives $f_{i,j,k}$ and $g_{i,j,k}$ are considered equivalent if there is a constant $c$ such that $f_{i,j,k}=g_{i,j,k}+c$ for all $(i,j,k)\in \Delta_m^3$. Throughout we will assume that we are dealing with the representative with $f_{m,0,0}=0$. See Figure \ref{hive examples} for two examples of $3$-hives for $m=4$.
\end{Def}

\begin{figure}
\centering
\scalemath{.5}{
\begin{tikzpicture}
  \begin{ternaryaxis}[
  %%separate axis lines,
%%y axis line style= { draw opacity=0 },
  ylabel=$\scalemath{2}{\lambda=(2,1,0,0)}$,xlabel=$\scalemath{2}{\mu=(2,1,0,0)}$,zlabel=$\scalemath{2}{\nu=(0,0,-2,-4)}$,
    clip=false,xticklabels={},yticklabels={},zticklabels={},
    xmin=0,
    ymin=0,
    zmin=0,
    xmax=8,
    ymax=8,
    zmax=8,
    xtick={2,4,6},
    ytick={2,4,6},
    ztick={2,4,6},
    major tick length=0,
    %axis .style={opacity=.1}
    %%opacity=0.8,
    %%clip=false
]    
  	\node at (axis cs:0,24,0) {$\scalemath{2}{0}$};
	\node at (axis cs:2,22,0) {$\scalemath{2}{2}$};
	\node at (axis cs:4,20,0) {$\scalemath{2}{3}$};
	\node at (axis cs:6,18,0) {$\scalemath{2}{3}$};
	\node at (axis cs:8,16,0) {$\scalemath{2}{3}$};
	
	\node at (axis cs:0,22,2) {$\scalemath{2}{4}$};
	\node at (axis cs:2,20,2) {$\scalemath{2}{5}$};
	\node at (axis cs:4,18,2) {$\scalemath{2}{5}$};
	\node at (axis cs:6,16,2) {$\scalemath{2}{5}$};
	
	\node at (axis cs:0,20,4) {$\scalemath{2}{6}$};
	\node at (axis cs:2,18,4) {$\scalemath{2}{6}$};
	\node at (axis cs:4,16,4) {$\scalemath{2}{6}$};
	
	\node at (axis cs:0,18,6) {$\scalemath{2}{6}$};
	\node at (axis cs:2,16,6) {$\scalemath{2}{6}$};
	
	\node at (axis cs:0,16,8) {$\scalemath{2}{6}$};
		
  \end{ternaryaxis}
\end{tikzpicture}
}
\hspace{5mm}
\scalemath{.5}{
\begin{tikzpicture}
  \begin{ternaryaxis}[
  %%separate axis lines,
%%y axis line style= { draw opacity=0 },
  ylabel=$\scalemath{2}{\lambda=(2,1,0,0)}$,xlabel=$\scalemath{2}{\mu=(2,1,0,0)}$,zlabel=$\scalemath{2}{\nu=(0,-1,-1,-4)}$,
    clip=false,xticklabels={},yticklabels={},zticklabels={},
    xmin=0,
    ymin=0,
    zmin=0,
    xmax=8,
    ymax=8,
    zmax=8,
    xtick={2,4,6},
    ytick={2,4,6},
    ztick={2,4,6},
    major tick length=0,
    %axis .style={opacity=.1}
    %%opacity=0.8,
    %%clip=false
]    
  	\node at (axis cs:0,24,0) {$\scalemath{2}{0}$};
	\node at (axis cs:2,22,0) {$\scalemath{2}{2}$};
	\node at (axis cs:4,20,0) {$\scalemath{2}{3}$};
	\node at (axis cs:6,18,0) {$\scalemath{2}{3}$};
	\node at (axis cs:8,16,0) {$\scalemath{2}{3}$};
	
	\node at (axis cs:0,22,2) {$\scalemath{2}{4}$};
	\node at (axis cs:2,20,2) {$\scalemath{2}{5}$};
	\node at (axis cs:4,18,2) {$\scalemath{2}{5}$};
	\node at (axis cs:6,16,2) {$\scalemath{2}{5}$};
	
	\node at (axis cs:0,20,4) {$\scalemath{2}{5}$};
	\node at (axis cs:2,18,4) {$\scalemath{2}{6}$};
	\node at (axis cs:4,16,4) {$\scalemath{2}{6}$};
	
	\node at (axis cs:0,18,6) {$\scalemath{2}{6}$};
	\node at (axis cs:2,16,6) {$\scalemath{2}{6}$};
	
	\node at (axis cs:0,16,8) {$\scalemath{2}{6}$};
		
  \end{ternaryaxis}
\end{tikzpicture}
}
\caption{Two examples of $3$-hives with $m=4$ \label{hive examples}}
\end{figure}

A simple consequence of the hive inequalities is that the sequence of edge labels $(f_{m,0,0},f_{m-1,1,0}\ldots,f_{0,m,0})$ is concave, that is, 
\begin{align*}
(f_{m-1,1,0}-f_{m,0,0},f_{m-2,2,0}-f_{m-1,2,0},\ldots,f_{1,m-1,0}-f_{0,m,0})
\end{align*}
is weakly decreasing, hence a dominant weight $\lambda$. Similarly, the successive differences of the edge labels along the other two edges of the hive give dominant weights $\mu$ and $\nu$ where differences are taken clockwise. Notice that taking successive differences of edge labels in the opposite order gives the dual weights $\lambda^*, \mu^*, \nu^*$. We adopt the clockwise convention throughout and say that the edges of the $3$-hive are labeled by $\lambda, \mu, \nu$, or that the hive has type $\lambda, \mu, \nu$. 

The main result about hives is that they count tensor product multiplicities of $GL_m$ representations, or equivalently the dimension of invariant spaces. That is, let $V_{\lambda}$ denote the finite-dimensional irreducible representation of $GL_m$ with highest weight $\lambda$. Then the following theorem was proved in \cite{KT} and used to establish the saturation conjecture for $GL_m$. See \cite{Buc} for a self-contained treatment of hives and the saturation conjecture.

\begin{Thm}[\cite{KT}]\label{thm:3hives}
The number of $3$-hives of type $\lambda, \mu, \nu$ is equal to the dimension of the invariant space $(V_{\lambda}\otimes V_{\mu}\otimes V_{\nu})^{GL_m}$.
\end{Thm}

We will need the following lemma in \S\ref{sec:main combinatorics}. 
\begin{Lem}\label{lem:pieri}
Let $\lambda,\mu,\nu$ be dominant weights such that $\lambda$ is minuscule. There is exactly one $3$-hive with boundary labeled clockwise by $\lambda,\mu,\nu$ if $\nu^*=\mu+w\cdot \lambda$ for some $w\in S_n$, and none otherwise. In other words, the boundary hive values completely determine the interior face values if one exists.
\end{Lem}
\begin{proof}
Recall that the tensor product with a minuscule weight is multiplicity free, that is, $V_{\lambda}\otimes V_{\mu}=\oplus_{\nu} c_{\lambda,\mu}^{\nu^*} V_{\nu^*}$, where $c_{\lambda,\mu}^{\nu^*}$ is either $0$ or $1$. In particular, by the Pieri rule, the only irreducible representations that have multiplicity $1$ correspond to the dominant $\nu^*$ gotten from $\mu$ by adding (resp. subtracting) a vertical strip of $i$ boxes if $\lambda=\omega_i$ (resp. $\omega_i^*$). Apply Theorem \ref{thm:3hives} and the fact that $c_{\lambda,\mu}^{\nu^*}=\dim(V_\lambda\otimes V_\mu\otimes V_\nu)^{GL_m}$.
\end{proof}

Theorem \ref{thm:3hives} and Corollary \ref{cor:Satake} imply that $3$-hives are equinumerous with components of Poly$(\lambda,\mu,\nu)$. Kamnitzer gave an explicit bijection in \cite{Kam} and generalized the definition of hives to more than three tensor factors. To define $n$-hives let
\begin{align*}
\Delta_m^n=\{(i_1,\ldots,i_n)\in \mathbb Z_{\geq 0}^n\mid i_1+\cdots i_n=m\}
\end{align*}
be an $(n-1)$-dimensional, size-$m$ simplex of lattice points. 
\begin{Def}
An \emph{$n$-hive of size $m$} is an assignment of integers $f_{i_1,\ldots,i_n}$ to each point of $\Delta_{m}^n$ such that the restriction to every two dimensional face of $\Delta_m^n$ is a $3$-hive as defined above, and the following condition is satisfied. For all $\vec{\jmath}=(j_1,\ldots,j_n)$ such that $\sum_{k=1}^n j_k=m-2$ and all $1\leq a<b<c<d \leq n$ we have
\begin{align}\label{oct rec}
f_{\vec{\jmath}+e_a+e_c}+f_{\vec{\jmath}+e_b+e_d}=\max\left(f_{\vec{\jmath}+e_a+e_b}+f_{\vec{\jmath}+e_c+e_d},f_{\vec{\jmath}+e_a+e_d}+f_{\vec{\jmath}+e_b+e_c}\right)
\end{align}
where $e_k$ is the vector with a single $1$ in the $k$th position and all other entries $0$. As before, two hives are considered equivalent if they differ by a global scaling by some constant. We assume that hives are scaled so that $f_{m,0,\ldots,0}=0$. The hive values corresponding to lattice points $(i_1,\ldots,i_n)$ with at most three nonzero entries are called \emph{face values} or \emph{face labels}. Those at lattice points with at most two nonzero entries are called \emph{edge values} or \emph{edge labels}.
\end{Def}
Condition (\ref{oct rec}) in the definition is called the \emph{octahedron recurrence} \cite{RR}. The octahedron recurrence preserves the validity of the rhombus inequalities from one $3$-subhive to another.  As mentioned in the introduction, it is the tropical version of the relations on the Fock--Goncharov coordinates on the configuration space of principal flags for $GL_m$ \cite[\S3]{GS} (see also Remark \ref{rem:canonical coordinates}).

Consider the cycle of lattice points $me_1,\ldots,me_n$ in $\Delta^n_m$. For an  $n$-hive the successive differences $f_{(m-j,j,0,\ldots,0)}-f_{(m-j-1,j+1,0,\ldots,0)}$ along the first edge of this cycle give a dominant weight $\lambda^1$. Likewise, let $\lambda^i$ record the differences of the edge labels along the edge from $me_i$ to $me_{i+1}$. In this case, we say that the hive has \emph{type} $\vec\lambda=(\lambda^1,\ldots,\lambda^n)$, or that the $n$-gon of the $n$-hive is labelled by $\vec\lambda$. The analogue of Theorem \ref{thm:3hives} holds for $n$ tensor factors, which was already shown for $n=4$ in \cite{KTW}. 

\begin{Thm}[\cite{Kam, GS}]\label{thm:n-hive LR}
Let $(\lambda^1,\ldots,\lambda^n)$ be a sequence of dominant weights, not necessarily minuscule. The number of $n$-hives of type $(\lambda^1,\ldots,\lambda^n)$ is equal to the dimension of the invariant space $(V_{\lambda^1}\otimes\cdots\otimes V_{\lambda^n})^{GL_m}$.
\end{Thm}

This theorem was proved using a constructible function $H:\text{Poly}(\vec\lambda)\rightarrow \mathbb Z^{\Delta_m^n}$ defined in \cite{Kam}. It was conjectured in \cite{Kam} that the generic values of this constructible function are $n$-hives and that this is a bijection between components of Poly$(\vec\lambda)$ and $n$-hives of type $\vec\lambda$. This was proved in \cite{GS}, building on the work of \cite{FG}. The previous theorem then follows from Corollary \ref{cor:Satake} of the geometric Satake correspondence, although the explicit bijection is more important for our purposes.

\subsection{Bijection Between Components and Hives \label{sec:Kamnitzer function}}
A hive can be interpreted as a collection of generic values of constructible functions on the polygon space. For each $\vec{\imath}=(i_1,\ldots,i_n)$ in $\Delta_m^n$ denote the corresponding function by
\begin{align*}
H_{\vec{\imath}}:\text{Poly}(\vec\lambda)\rightarrow \mathbb Z.
\end{align*}
These were defined in \cite{Kam} with a suggestion from Speyer as follows. For any representation $V$ of $GL_m(\mathbb C)$, the group $GL_m(\mathcal K)$ acts on $\mathcal K\otimes V$. Think of this as a right action. Recall that $\mathcal O=\mathbb C[[t]]$, so the vector space $\mathcal K\otimes V$ has the filtration
\begin{align*}
\cdots \subset t\mathcal O\otimes V\subset\mathcal O\otimes V\subset t^{-1}\mathcal O\otimes V\subset \cdots.
\end{align*}
Define the valuation $\text{val}:\mathcal K\otimes V\rightarrow \mathbb Z$ by $\text{val}(v)=-k$ if $v\in t^k\mathcal O\otimes V$, but $v\not\in t^{k+1}\mathcal O\otimes V$. The negative here is due to the fact that we are using the max convention in the definition of octahedron recurrence.

Given $\vec{\imath}=(i_1,\ldots,i_n)$ in $\Delta_m^n$, let $V_{\vec\imath}=V_{\omega_{i_1}}\otimes\cdots\otimes V_{\omega_{i_n}}$ where $V_{\omega_{j}}$ is the $j$th fundamental representation of $GL_m$. Then $V_{\vec\imath}$ is a representation of $GL_m^n$, and restricting to the diagonal subgroup $GL_m^{\Delta}$ we see that $V_{\vec\imath}$ contains a unique copy of the determinant representation. Let $\xi_{\vec\imath}$ be a vector in this one-dimensional subrepresentation of $V_{\vec\imath}$. 

Since $GL_m(\mathcal K)^n$ acts on $\mathcal K\otimes V_{\vec\imath}$, for any $(h_1,\ldots,h_n)\in G(\mathcal K)^n$ we can evaluate val$\left(\xi_{\vec\imath}\cdot (h_1,\ldots,h_n)\right)$. This descends to the quotient $Gr^n=GL_m(\mathcal K)^n/GL_m(\mathcal O)^n$ because the action of $GL_m(\mathcal O)^n$ on $\mathcal K\otimes V_{\vec\imath}$ preserves the filtration and hence the valuation of any vector. Since $\text{Poly}(\vec\lambda)$ is a subset of $Gr^n$, this gives a well-defined function.

The following theorem was proved for $n=3$ in \cite[Theorem 1.4]{Kam} and in full generality in \cite{GS}. We state it without proof.
\begin{Thm}[\cite{Kam, GS}]\label{thm:Goncharov-Shen}
The values $H_{\vec\imath}$ on a generic point of a component of $\text{Poly}(\vec\lambda)$ form an $n$-hive of type $\vec\lambda$. This is a bijection between the components of the polygon space and $n$-hives. In other words, each component of the polygon space has an open, dense set on which the values of $H_{\vec\imath}$ are constant and give a hive.
\end{Thm}

\begin{Rem}\label{rem:canonical coordinates}
Ian Le gave an equivalent interpretation of these functions in \cite{Le1}. Given a point $(g_1,\ldots,g_{n})\in Gr^n$ consider each $g_j$ as a lattice. Consider the values
\begin{align*}
-\text{val}(v_{11},\ldots,v_{1i_1},v_{21},\ldots,v_{1i_2},\ldots,v_{n1},\ldots,v_{1i_n})
\end{align*}
where $v_{j1},\ldots,v_{ji_j}$ range over the elements of the lattice $g_j$. Then $H_{i_1,\ldots,i_n}$ is the maximum of these values. As mentioned at the end of \S\ref{sec:intro geometry} of the introduction, the functions $H_{i_1,\ldots,i_n}$ are the tropicalization of the functions $h_{i_1,\ldots,i_n}$ defined there.
\end{Rem}

Restricting to the minuscule case, we will use this theorem to establish the bijection between components of Poly$(\vec\lambda)$ and affine growth diagrams of type $\vec\lambda$. When referring to a generic point of a component of Poly$(\vec\lambda)$ we mean a point in the open, dense set of the previous theorem. As noted, Theorem \ref{thm:n-hive LR} is a consequence of combining this theorem with geometric Satake.

The following proposition relates the edge labels of a hive to the pairwise distances $d(g_i,g_j)$ for generic points of the corresponding component of the polygon space. It is also mentioned in \cite[Remark 2]{LO}. We provide a proof for completeness.
\begin{Prop}\label{prop:hive edges}
Let $p=(g_1=[1]=g_{n+1},g_2,\ldots,g_n)$ be a generic point of a component of Poly$(\vec\lambda)$, so that the function values $f_{\vec{\imath}}=H_{\vec{\imath}}(p)$ form an $n$-hive. Consider the edge values $f_{\vec{s}_k}$ at the lattice points $\vec{s}_k=(0,\ldots,0,m-k,0,\ldots,0,k,0,\ldots,0)$ for all $0\leq k\leq m$ where the nonzero entries appear at indices $i$ and $j$. Then the dominant weight given by the differences $(f_{\vec{s}_1}-f_{\vec{s}_0}, f_{\vec{s}_2}-f_{\vec{s}_1},\ldots,f_{\vec{s}_m}-f_{\vec{s}_{m-1}})$ is $d(g_i,g_j)$. 
\end{Prop}
\begin{proof}
For any $k$ the value $H_{\vec{s}_k}(g_1,\ldots,g_i,\ldots,g_j,\ldots,g_n)$ depends only on $g_i$ and $g_j$, since the $0$ entries of $\vec{s}_k$ correspond to the trivial representation in the definition of $V_{\vec{s}_k}$, so denote this value by $H_{\vec{s}_k}(g_i,g_j)$. There exists an element $h\in G(\mathcal K)$ such that $(h\cdot g_i,h\cdot g_j)=(t^0,t^\lambda)$ where $d(g_i,g_j)=\lambda$. For any $h\in GL_m(\mathcal K)$, on checks from the definition that $H_{\vec{s}_k}(h\cdot g_i,h\cdot g_j)=H_{\vec{s}_k}(g_a,g_b)+\text{val}(\det(h))$ (note that $\xi_{\vec{s}_k}$ is an eigenvector for the action of $(h,\ldots,h)$). Together with the fact that $H_{\vec{s}_k}(t^0,t^\lambda)=\lambda_1+\cdots+\lambda_j$ this shows that the successive differences give $(\lambda_1,\ldots,\lambda_m)$. 
\end{proof}

\subsection{The Main Theorem Related to Geometry}\label{sec:main combinatorics}
The following theorem is the main result relating the combinatorics of growth diagrams to the geometry of polygon spaces. Recall that $St_n$ is the set of vertices of the staircase diagram, and let $\overline{i}$ denote $i$ modulo $n$.
\begin{Thm}\label{thm:geometry theorem}
Let $Z$ be a component of Poly$(\vec\lambda)$ and $p=(g_1=[1]=g_{n+1},g_1,\ldots,g_n)$ a generic point in $Z$. The dominant weights $\gamma_{i,j}=d(g_{\overline{i}},g_{\overline{j}})$ for $(i,j)$ in $St_n$ form an affine growth diagram of type $\vec\lambda$ and do not depend on the choice of generic $p$ in $Z$. Furthermore, this is a bijection, that is, the affine growth diagrams of type $\vec\lambda$ index the components of Poly$(\vec\lambda)$.
\end{Thm}

This geometric interpretation of affine growth diagrams has the following two immediate combinatorial consequences. See \cite{Hen} for related work on the periodicity of the octahedron recurrence.

\begin{Cor}\label{cor:symmetry}
Any affine growth diagram $\{\gamma_{i,j}\}_{(i,j)\in St_n}$ is periodic with period $n$, that is $\gamma_{i,j}=\gamma_{i+n,j+n}$. Furthermore, there is a dual symmetry, $\gamma_{i,j}^*=\gamma_{j,i+n}$.
\end{Cor}
\begin{proof}
Any affine growth diagram of type $\vec\lambda$ corresponds to a component of Poly$(\vec\lambda)$, so for a generic point $p=(g_1,g_2,\ldots,g_n)$, $\gamma_{i,j}=d(g_{\overline{i}},g_{\overline{j}})=\gamma_{i+n,j+n}$. Similarly, for a fixed index $i$ let $j=i+k$ for some $k$ such that $0\leq k\leq n$. Then $\gamma_{i,j}^*=d(g_{\overline{i}},g_{\overline{j}})^*=d(g_{\overline{j}},g_{\overline{i}})=\gamma_{j,i+n}$. Note $j\leq j+n-k=i+n\leq j+n$, so this is a valid vertex of the infinite staircase.
\end{proof}

The last part of the corollary says that the $i$th row, $\{\gamma_{i,j}\mid i\leq j\leq i+n\}$, read left to right is dual to the $(i+n)$th column, $\{\gamma_{j,i+n}\mid i\leq j\leq i+n\}$, read top to bottom. Put differently, consider the triangle of weights
\begin{align*}
\left\{\gamma_{i,j}\mid 1\leq i\leq n,\; i\leq j\leq n+1\right\}, 
\end{align*}
which we will call the \emph{fundamental triangular region} of the staircase or just the \emph{fundamental region}. Transposing the fundamental region and taking duals gives the triangle of weights
\begin{align*}
\left\{\gamma_{i,j}\mid 1\leq i\leq n\;, n+1\leq j\leq i+n\right\},
\end{align*}
\emph{the dual fundamental region}. Hence, the weight labels of the fundamental triangle determine the rest of the diagram by dual symmetry and periodicity.

We now discuss the proof of Theorem \ref{thm:geometry theorem}. Let $p=(g_1,\ldots,g_n)$ be a generic point in a component $Z$ of Poly$(\vec\lambda)$, and let $\gamma_{i,j}=d(g_{\overline{i}},g_{\overline{j}})$. It is immediate from the definition that $\gamma_{i,i}=\vec{0}$ and $\gamma_{i,i+1}=d(g_{\overline{i}},g_{\overline{i+1}})=\lambda^i$. To prove the first part of Theorem \ref{thm:geometry theorem} we must show that each of the neighboring pairs $\gamma_{i,j},\gamma_{i,j+1}$ and $\gamma_{i,j},\gamma_{i+1,j}$ differ by a vertical strip and that the local condition 
\begin{align*}
\gamma_{i+1,j+1}=\text{sort}(\gamma_{i+1,j}+\gamma_{i,j+1}-\gamma_{i,j})
\end{align*}
is satisfied for all $(i,j)\in St_n$ with $j\not=i$ and $j\not=i+n$. Proposition \ref{prop:hive edges} shows that the weights $\gamma_{i,j}$ are given by the edge labels of the corresponding $n$-hive $f_{\vec{\imath}}=H_{\vec{\imath}}(p)$, so we can work directly with the hive. The vertical strip condition follows from Lemma \ref{lem:pieri} by setting $\mu=\gamma_{i,j}, \nu^*=\gamma_{i,j+1}$ and $\lambda=\lambda^{i}$ because the face values $f_{\vec\imath}$ spanned by the vertices $i,j,j+1$ form a $3$-hive. Similar reasoning applies to vertices $i,i+1,j$. 

\begin{figure}[b]
\center
\begin{tikzpicture}
 \node[regular polygon, regular polygon sides=10, minimum size=4cm, draw,
        outer sep=0pt] (a) {};
\draw[decoration={markings,mark=at position 0.5 with {\arrow{>}}},postaction={decorate}] (a.corner 2) -- (a.corner 8) node[midway, above, shift={(.25,-.1)}] {\scalemath{1}{\gamma_{3,7}}};
\draw[decoration={markings,mark=at position 0.7 with {\arrow{>}}},postaction={decorate}] (a.corner 2) -- (a.corner 7) node[near end, above, shift={(.3,-.1)}] {\scalemath{1}{\gamma_{3,8}}};
\draw[decoration={markings,mark=at position 0.3 with {\arrow{>}}},postaction={decorate}] (a.corner 3) -- (a.corner 8) node[near start,above] {\scalemath{1}{\gamma_{2,7}}};
\draw[decoration={markings,mark=at position 0.5 with {\arrow{>}}}, postaction={decorate}] (a.corner 3) -- (a.corner 7) node[midway, below, shift={(-.2,.1)}] {\scalemath{1}{\gamma_{2,8}}};

\draw [decoration={markings,mark=at position 0.5 with {\arrow{>}}}, postaction={decorate}] (a.corner 2) -- (a.corner 1) node[midway, above] {$\lambda^{3}$};
\draw [decoration={markings,mark=at position 0.5 with {\arrow{>}}}, postaction={decorate}] (a.corner 3) -- (a.corner 2) node[midway, above] {$\lambda^{2}$};
\draw [decoration={markings,mark=at position 0.5 with {\arrow{>}}}, postaction={decorate}] (a.corner 4) -- (a.corner 3) node[midway, left] {$\lambda^{1}$};
\draw [decoration={markings,mark=at position 0.5 with {\arrow{>}}}, postaction={decorate}] (a.corner 5) -- (a.corner 4) node[midway, left] {$\lambda^{10}$};
\draw [decoration={markings,mark=at position 0.5 with {\arrow{>}}}, postaction={decorate}] (a.corner 6) -- (a.corner 5) node[midway, below] {$\lambda^{9}$};
\draw [decoration={markings,mark=at position 0.5 with {\arrow{>}}}, postaction={decorate}] (a.corner 7) -- (a.corner 6) node[midway, below] {$\lambda^{8}$};
\draw [decoration={markings,mark=at position 0.5 with {\arrow{>}}}, postaction={decorate}] (a.corner 8) -- (a.corner 7) node[midway, below] {$\lambda^{7}$};
\draw [decoration={markings,mark=at position 0.5 with {\arrow{>}}}, postaction={decorate}] (a.corner 9) -- (a.corner 8) node[midway, right] {$\lambda^{6}$};
\draw [decoration={markings,mark=at position 0.5 with {\arrow{>}}}, postaction={decorate}] (a.corner 10) -- (a.corner 9) node[midway, right] {$\lambda^{5}$};
\draw [decoration={markings,mark=at position 0.5 with {\arrow{>}}}, postaction={decorate}] (a.corner 1) -- (a.corner 10) node[midway, above] {$\lambda^{4}$};
\end{tikzpicture}
\caption{A 10-gon with edges indicated for $i=2$, $j=7$.\label{quadrilateral}}
\end{figure}

We are left to prove that the weights given by the edge labels $f_{\vec\imath}$ between the $i,i+1,j,j+1$ vertices of the hive satisfy the local condition (\ref{local condition}). Figure \ref{quadrilateral} shows the corresponding edges in the $n$-gon for $n=10$ and $(i,j)=(2,7)$. For fixed indices $i$ and $j$ the hive values spanned by the vertices $me_i,me_{i+1},me_j,me_{j+1}$ form a $4$-hive, so we assume that $i=1$ and $j=3$ in the statement of the following proposition.

\begin{figure}
\[
\begin{tikzpicture}[baseline=-3pt, scale=.8]
\shadedraw[top color=white!25!black, bottom color=white] (-2,0) -- (0,2) -- (2,0) -- (-2,0);
\shadedraw[top color=white, bottom color=white!25!black] (-2,0) -- (0,-2) -- (2,0) -- (-2,0);
%\draw (-1,0) -- (-1,1);
%\draw (-1.5,0) -- (-1.5,.5);
%\draw (-.5,0) -- (-.5,1.5);
%\draw (0,0) -- (0,2);
%\draw (.5,0) -- (.5,1.5);
%\draw (1.5,0) -- (1.5,.5);
%\draw (1,0) -- (1,1);

\node at (-.15,-.25) {$\gamma_{1,3}$};
\node at (1.25,1.25) {$\gamma_{2,3}$};
\node at (-1.25,-1.25) {$\gamma_{1,4}$};
\node at (-1.25,1.25) {$\gamma_{1,2}$};
\node at (1.25,-1.25) {$\gamma_{3,4}$};

%%%%%%top triangle%%%%%%%
\draw (-1.5,.5) -- (1.5,.5);
\draw (-1,1) -- (1,1);
\draw (-.5,1.5) -- (.5,1.5);

\draw (1,0) -- (-.5,1.5);
\draw (0,0) -- (-1,1);
\draw (-1,0) -- (-1.5,.5);
\draw (1,0) -- (1.5,.5);
\draw (0,0) -- (1,1);
\draw (-1,0) -- (.5,1.5);

%%%%%%bottom triangle%%%%%%%
\draw (-1.5,-.5) -- (1.5,-.5);
\draw (-1,-1) -- (1,-1);
\draw (-.5,-1.5) -- (.5,-1.5);

\draw (1,0) -- (-.5,-1.5);
\draw (0,0) -- (-1,-1);
\draw (-1,0) -- (-1.5,-.5);
\draw (1,0) -- (1.5,-.5);
\draw (0,0) -- (1,-1);
\draw (-1,0) -- (.5,-1.5);

\end{tikzpicture}
\;\longrightarrow\;
\begin{tikzpicture}[baseline=-3pt, scale=.8]
\shadedraw[left color=white!25!black, right color=white] (0,-2) -- (2,0) -- (0,2) -- (0,-2);
\shadedraw[left color=white, right color=white!25!black] (0,-2) -- (-2,0) -- (0,2) -- (0,-2);

%%%%%%%%%%%%%
\draw  (-1.5,-.5) -- (-1.5,.5);
\draw (-1,-1) -- (-1,1);
\draw (-.5,-1.5) -- (-.5,1.5);

\draw  (1.5,-.5) -- (1.5,.5);
\draw (1,-1) -- (1,1);
\draw (.5,-1.5) -- (.5,1.5);

\draw (-.5,1.5) -- (1.5,-.5);
\draw (-1,1) -- (1,-1);
\draw (-1.5,.5) -- (.5,-1.5);
\draw (1.5,.5) -- (-.5,-1.5);
\draw (1,1) -- (-1,-1);
\draw (.5,1.5) -- (-1.5,-.5);

\node at (.6,0) {$\gamma_{2,4}$};
\node at (1.25,1.25) {$\gamma_{2,3}$};
\node at (-1.25,-1.25) {$\gamma_{1,4}$};
\node at (-1.25,1.25) {$\gamma_{1,2}$};
\node at (1.25,-1.25) {$\gamma_{3,4}$};
\end{tikzpicture}
\]

\[
\begin{tikzpicture}[baseline=-3pt, scale=.8]
\shadedraw[top color=white!25!black, bottom color=white] (-2,0) -- (0,2) -- (2,0) -- (-2,0);
\shadedraw[top color=white, bottom color=white!25!black] (-2,0) -- (0,-2) -- (2,0) -- (-2,0);
%%%%%%top triangle%%%%%%%
\draw[color=white, thick] (-1,0) -- (1,0);

\draw (-1.5,.5) -- (1.5,.5);
\draw (-1,1) -- (1,1);
\draw (-.5,1.5) -- (.5,1.5);

\draw (1,0) -- (-.5,1.5);
\draw (0,0) -- (-1,1);
\draw (-1,0) -- (-1.5,.5);
\draw (1,0) -- (1.5,.5);
\draw (0,0) -- (1,1);
\draw (-1,0) -- (.5,1.5);

\shadedraw[left color=white, right color=white!65!black] (0,0) -- (.5,.5) -- (.5,-.5) -- cycle;
\shadedraw[right color=white, left color=white!65!black] (1,0) -- (.5,.5) -- (.5,-.5) -- cycle;
\shadedraw[left color=white, right color=white!65!black] (-1,0) -- (-.5,.5) -- (-.5,-.5) -- cycle;
\shadedraw[right color=white, left color=white!65!black] (0,0) -- (-.5,.5) -- (-.5,-.5) -- cycle;
\draw (.5,.5) -- (.5,-.5);
\draw (-.5,.5) -- (-.5,-.5);

%%%%%%bottom triangle%%%%%%%
\draw (-1.5,-.5) -- (1.5,-.5);
\draw (-1,-1) -- (1,-1);
\draw (-.5,-1.5) -- (.5,-1.5);

\draw (1,0) -- (-.5,-1.5);
\draw (0,0) -- (-1,-1);
\draw (-1,0) -- (-1.5,-.5);
\draw (1,0) -- (1.5,-.5);
\draw (0,0) -- (1,-1);
\draw (-1,0) -- (.5,-1.5);
\end{tikzpicture} %%%%%%%%%%%%%%%%%%%%%%%%%%%%%%%%%%%%%%%%%%%
\;\longrightarrow\;
\begin{tikzpicture}[baseline=-3pt, scale=.8]
\shadedraw[top color=white!25!black, bottom color=white] (-2,0) -- (0,2) -- (2,0) -- (-2,0);
\shadedraw[top color=white, bottom color=white!25!black] (-2,0) -- (0,-2) -- (2,0) -- (-2,0);
%%%%%%top triangle%%%%%%%
\draw[color=white, thick] (-1,0) -- (1,0);

\draw (-1.5,.5) -- (1.5,.5);
\draw (-1,1) -- (1,1);
\draw (-.5,1.5) -- (.5,1.5);

\draw (1,0) -- (-.5,1.5);
\draw (0,0) -- (-1,1);
\draw (-1,0) -- (-1.5,.5);
\draw (1,0) -- (1.5,.5);
\draw (0,0) -- (1,1);
\draw (-1,0) -- (.5,1.5);

\shadedraw[right color=white!75!black, left color=white!25!black] (0,0) -- (.5,.5) -- (.5,-.5) -- cycle;
\shadedraw[right color=white, left color=white!65!black] (1,0) -- (.5,.5) -- (.5,-.5) -- cycle;
\shadedraw[left color=white, right color=white!65!black] (-1,0) -- (-.5,.5) -- (-.5,-.5) -- cycle;
\shadedraw[left color=white!75!black, right color=white!25!black] (0,0) -- (-.5,.5) -- (-.5,-.5) -- cycle;
\shadedraw[top color=white!75!black, bottom color=white!25!black] (0,0) -- (.5,.5) -- (-.5,.5) -- cycle;
\shadedraw[bottom color=white!75!black, top color=white!25!black] (0,0) -- (.5,-.5) -- (-.5,-.5) -- cycle;
\draw (.5,.5) -- (.5,-.5);
\draw (-.5,.5) -- (-.5,-.5);

%if you want a blue circle on the revealed vertex
%\node at (0,0) [circle,fill,inner sep=1.5pt,blue]{};

%%%%%%bottom triangle%%%%%%%
\draw (-1.5,-.5) -- (1.5,-.5);
\draw (-1,-1) -- (1,-1);
\draw (-.5,-1.5) -- (.5,-1.5);

\draw (1,0) -- (-.5,-1.5);
\draw (0,0) -- (-1,-1);
\draw (-1,0) -- (-1.5,-.5);
\draw (1,0) -- (1.5,-.5);
\draw (0,0) -- (1,-1);
\draw (-1,0) -- (.5,-1.5);
\end{tikzpicture}
\]
\caption{A $4$-hive balanced on its bottom edge labelled by $\gamma_{2,4}$ and a diagrammatic depiction of an application of the octahedron recurrence at a single unit octahedron. \label{fig:excavate}}
\end{figure}

\begin{Prop}\label{prop:hive excavation}
Consider a $4$-hive such that the successive differences of the edge values along the edge from $me_i$ to $me_j$ give the dominant weight $\gamma_{i,j}$ for $1\leq i,j\leq 4$. Suppose that $\gamma_{1,2}$ and $\gamma_{2,3}$ are minuscule. Then $\gamma_{2,4}=\text{sort}(\gamma_{2,3}+\gamma_{1,4}-\gamma_{1,3})$.
\end{Prop}

The proof is the main technical part of establishing Theorem \ref{thm:geometry theorem} and appears in \S \ref{main proof}. Although involved, it is elementary and combinatorial. The idea is as follows. Picture the $4$-hive balanced on the edge labelled by $\gamma_{2,4}$, so that the top two faces are visible from above. The view from above shows the edge corresponding to $\gamma_{1,3}$, but the bottom edge labelled by $\gamma_{2,4}$ is obscured from view by the hive. 

The goal is to use the octahedron recurrence to excavate from the top two faces to the bottom two faces, revealing the hive values along the bottom edge. The octahedron recurrence can be viewed as removing a unit octahedron to reveal the integer label at a previously obscured lattice point of the hive. See Figure \ref{fig:octahedron} of the introduction and Figure \ref{fig:excavate}. The shaded tetrahedron figures were inspired by the figures in \cite{KTW} and \cite{Zin}. A similar proof technique was used in \cite{KTW} to prove associativity of a ring multiplication defined via hives, implying that $3$-hives count tensor product multiplicities. The argument is also extended to Hall polynomials in \cite{Zin}. 

\subsubsection{Proof of Bijectivity}
The rest of this section is concerned with proving the bijectivity statement of Theorem \ref{thm:geometry theorem}. We know by Theorem \ref{thm:Goncharov-Shen} that components of Poly$(\vec\lambda)$ are in bijection with $n$-hives of type $\vec\lambda$, and by Proposition \ref{prop:hive edges} that the distances $d(g_i,g_j)$ coincide with the weights given by the $(i,j)$-edge hive values. The previous proposition gives a map from $n$-hives of type $\vec\lambda$ to affine growth diagrams of type $\vec\lambda$, so it suffices to show that this map is a bijection. We first show that this map is injective by showing that $n$-hives are overdetermined.

\begin{Lem}\label{lem:triangulation}
Let $T$ be a triangulation of the $n$-gon by diagonal edges. The face values of an $n$-hive corresponding to the faces of $T$ determine the entire $n$-hive.
\end{Lem}
\begin{proof}
Given two neighboring triangles of a triangulation $T$, one can remove the common edge of the two triangles and replace it with the other diagonal of the resulting quadrilateral, giving a new triangulation $T'$. This corresponds to applying the octahedron recurrence to the face values of the two triangles in the $n$-hive to determine all of the labels $f_{\vec{\imath}}$ in the $4$-subhive determined by the four vertices, as in Proposition \ref{prop:hive excavation}. Since any triangulation can be reached from a fixed triangulation $T$ via such diagonal flips, every hive value $f_{\vec{\imath}}$ where $\vec{\imath}$ has at most four nonzero entries is determined. It is easy to see by induction that the rest of the hive values are also determined. Suppose $\vec{\imath}=(i_1,i_2,\ldots,i_{k},0,\ldots,0)$ such that $i_j\not=0$ for $1\leq j\leq k$ and $i_{k-1}$ and $i_{k}$ are the two smallest values of the $i_j$ (possibly equal). Then applying the octahedron recurrence (\ref{oct rec}) for indices $a=k-1, b=1, c=k, d=2$ and $\vec{\jmath}=(i_1,i_2,\ldots,i_{k-1}-1,i_{k}-1,0,\ldots,0)$ expresses $f_{\vec{\imath}}$ in terms of labels corresponding to indices with either both $i_{k-1}, i_k$ decreased by one, or at least one of $i_{k-1}, i_{k}$ decreased by one and the other unchanged.  
\end{proof}

The previous Lemma holds for an $n$-hive without any minuscule assumptions. It can be refined further with the added assumption that the $n$-gon weights are minuscule. More specifically, the hive values along the edges of certain triangulations fully determine the hive values. Recall that a \emph{minuscule path of type $\vec{\lambda}$} is a sequence of dominant weights $(\mu^1=\vec{0},\mu^2,\ldots,\mu^{n},\mu^{n+1}=\vec{0})$ such that for all $i$ there exists a $w$ in $S_n$ satisfying $\mu^i-\mu^{i-1}=w\cdot \lambda^i$.

\begin{Prop}\label{prop:overdetermined}
Let $\vec\lambda$ be a sequence of minuscule weights and $f_{\vec{\imath}}$ an $n$-hive of type $\vec\lambda$. Let $T$ be the fan triangulation of the $n$-gon consisting of all diagonal edges adjacent to the first vertex $me_1$. Then the entire hive $f_{\vec{\imath}}$ is determined by the edge values along the edges of $T$. Furthermore, let $\mu^j$ be the corresponding weight given by the hive values along the edge $(1,j)$ of $T$. Then $(\vec{0},\mu^2,\ldots,\mu^{n},\vec{0})$ is a minuscule path of type $\vec\lambda$ and this is a bijection.
\end{Prop}
\begin{proof}
Since each of the triangles of $T$ has an edge labeled by a minuscule weight $\lambda^i$, by Lemma \ref{lem:pieri} the face values corresponding to the faces of $T$ are determined by the hive edge values of $T$. The rest of the hive values are determined by Lemma \ref{lem:triangulation}. That the $\mu^i$ form a minuscule path of type $\vec\lambda$ also follows from the Lemma \ref{lem:pieri}. This, together with the assumption that $f_{(m,0,\ldots,0)}=0$, gives the bijection.
\end{proof}

\begin{Rem}\label{rem:extroverted triangulation}
The proof of the previous proposition applies to any \emph{extroverted triangulation}, a triangulation of the $n$-gon such that each triangle contains an external edge of the $n$-gon.
\end{Rem}

The weights $\gamma_{1,j}=\mu^j$ given by the edges of the fan triangulation correspond to the labels along the first line of the corresponding affine growth diagram. Since the edge values determine the entire hive, the map from $n$-hives to growth diagrams $\gamma_{i,j}$ is injective. The previous proposition also shows that the image consists of growth diagrams such that the first row is a minuscule path of type $\vec\lambda$. We will show that this is true for any affine growth diagram of type $\vec\lambda$ and that the first row completely determines the rest of the diagram. 

To see that an affine growth diagram can be recovered from a single row we need the following lemma. This was shown in \cite{vLee2}, but we sketch a direct proof for $GL_m$.

\begin{Lem}\label{lem:reversible}
Let $\gamma_{i,j}$ be an affine growth diagram. Then
\begin{align*}
\gamma_{i,j}=\text{sort}(\gamma_{i+1,j}+\gamma_{i,j+1}-\gamma_{i+1,j+1}), 
\end{align*}
i.e. the local condition in a growth diagram is symmetric in $\gamma_{i,j}$ and $\gamma_{i+1,j+1}$ for fixed $\gamma_{i+1,j}$ and $\gamma_{i,j+1}$.
\end{Lem}
\begin{proof}
The local condition states that $\gamma_{i+1,j+1}=\text{sort}(\gamma_{i+1,j}+\gamma_{i,j+1}-\gamma_{i,j})$. We claim that there exists a $w\in S_n$ such that $\gamma_{i+1,j+1}=w\cdot(\gamma_{i+1,j}+\gamma_{i,j+1}-\gamma_{i,j})$, and $w$ stabilizes $\gamma_{i+1,j}$ and $\gamma_{i,j+1}$. Then $\gamma_{i+1,j+1}=\gamma_{i+1,j}+\gamma_{i,j+1}-w\cdot(\gamma_{i,j})$ implies $\gamma_{i,j}=w^{-1}\cdot(\gamma_{i+1,j}+\gamma_{i,j+1}-\gamma_{i+1,j+1})$, which is dominant since $\gamma_{i,j}$ is dominant.

To prove the claim let $\nu=\gamma_{i,j}$, $\lambda=\gamma_{i,j+1}$, $\mu=\gamma_{i,j+1}$, and $\rho=\gamma_{i+1,j+1}$. There are four cases depending on whether the differences $\lambda-\nu$ and $\mu-\nu$ are positive or negative vertical strips. We'll check the case that $\lambda-\nu=w'\cdot \omega_k$ and $\mu-\nu=w''\cdot\omega_l$, the others being similar. Hence, $\nu=(\nu_1,\ldots,\nu_n)$ is a fixed dominant weight such that adding a vector of $k$ $1$'s to $\nu$ gives $\lambda$, and adding a vector of $l$ $1$'s to $\nu$ gives $\mu$. Then $\mu+\lambda-\nu=\mu+w'\cdot\omega_k$ can be viewed as adding $k$ $1$'s to $\mu$ in the same entries as the nonzero entries of $\lambda-\nu$. The result is non-dominant only if a $1$ is added to $\mu_\alpha$ for an index $\alpha$ such that $\mu_{\alpha}=\mu_{\alpha-1}$. Since the difference $\mu-\nu$ is also a vertical strip, this can only happen if the $\alpha$ entry in $\mu-\nu=w''\cdot \omega_l$ contains a $1$. Here is an example with $k=2$ and $l=2$ where a dominant weight $\lambda$ is depicted as an outline of a partition with infinitely many boxes in row $i$ to the left of column $\lambda_i$.

$\nu=$\begin{tikzpicture}[baseline=5ex,scale=.35]
\draw[thick] (0,0) -- (1,0) -- (1,3) -- (2,3) -- (2,6) -- (4,6);
\end{tikzpicture}
$\lambda=$\begin{tikzpicture}[baseline=5ex,scale=.35]
\draw[thick] (0,0) -- (1,0) -- (1,3) -- (2,3) -- (2,6) -- (4,6);
\draw (1,3) rectangle (2,2);
\draw (1,3) rectangle (2,1);
\end{tikzpicture}
$\mu=$\begin{tikzpicture}[baseline=5ex,scale=.35]
\draw[thick] (0,0) -- (1,0) -- (1,3) -- (2,3) -- (2,6) -- (4,6);
\draw (1,3) rectangle (2,2);
\draw (2,6) rectangle (3,5);
\end{tikzpicture}
$(\mu+\lambda-\nu)=$\begin{tikzpicture}[baseline=5ex,scale=.35]
\draw[thick] (0,0) -- (1,0) -- (1,3) -- (2,3) -- (2,6) -- (4,6);
\draw (1,3) rectangle (2,1);
\draw (1,3) rectangle (3,2);
\draw (2,6) rectangle (3,5);
\end{tikzpicture}\;$\xrightarrow{sort}$\;
$\rho=$\begin{tikzpicture}[baseline=5ex,scale=.35]
\draw[thick] (0,0) -- (1,0) -- (1,3) -- (2,3) -- (2,6) -- (4,6);
\draw (1,3) rectangle (2,2);
\draw (1,2) rectangle (2,1);
\draw (2,6) rectangle (3,5);
\draw (2,5) rectangle (3,4);
\end{tikzpicture}

The necessary sorting only occurs within a range of indices for which the components of $\mu$ are equal, i.e. $\mu_{\alpha}=\mu_{\alpha-1}=\cdots=\mu_{\alpha-\beta}$, and likewise the corresponding entries of $\lambda$ are equal, so there exists a $w\in S_n$ stabilizing $\lambda$ and $\mu$ such that $\rho=w\cdot(\mu+\lambda-\nu)$. Of course, $\mu+\lambda-\nu$ may already be dominant, so taking $w$ to be the identity element in this case satisfies the statement.
\end{proof}

As a consequence, an affine growth diagram can be recovered from a single row by applying the deterministic local condition (\ref{local condition}) from the northwest to the southeast and from the southeast to the northwest. We now show that the first row must be a minuscule path of type $\vec\lambda$. As always, the indices $i$ in $\lambda^i$ are taken modulo $n$.

\begin{Lem}{\label{Lem:edge differences}}
Let $\gamma_{i,j}$ be an affine growth diagram of type $\vec{\lambda}$. For all $i$ and $j$ such that $i+1\leq j\leq i+n$ there exists a $w\in S_n$ such that $\gamma_{i,j}-\gamma_{i+1,j}=w\cdot \lambda^i$. In particular, $\gamma_{i+1,i+n}={\lambda^i}^*$. Likewise, for all $j$ and $i$ such that $j-n\leq i\leq j-1$ there exists $w\in S_n$ such that $\gamma_{i,j}-\gamma_{i,j-1}=w\cdot \lambda^j$. 
\end{Lem}

\begin{proof}
Fix a row index $i$ and suppose that $\lambda^i=\omega_k$. The first statement says that the differences across row $i$ and $i+1$ are always a positive vertical strip of size $k$. To see this, propagate differences to the right along a row using the local condition. For $j=i+1$, $\gamma_{i,i+1}-\gamma_{i+1,i+1}=\omega_k-\vec{0}=\omega_k$ obviously satisfies the statement. The local condition states that $\gamma_{i+1,i+2}=\text{sort}(\gamma_{i,i+2}+\gamma_{i+1,i+1}-\gamma_{i,i+1})=\text{sort}(\gamma_{i,i+2}-\omega_k)$, that is, $\gamma_{i+1,i+2}$ and $\gamma_{i,i+2}$ differ by a vertical strip of size $k$. Hence, $\gamma_{i,i+2}-\gamma_{i+1,i+2}=w\cdot \omega_k$ for some $w\in S_n$. Continuing along by induction gives the result for all $j$. Then for $j=i+n$, $\gamma_{i,i+n}-\gamma_{i+1,i+n}=w\cdot \omega_k$, but $\gamma_{i,i+n}=\vec{0}$ implies $\gamma_{i+1,i+n}=-w\cdot \omega_k$. Since $\gamma_{i+1,i+n}$ is dominant this implies that $\gamma_{i+1,i+n}=\omega_k^*$. For the final statement, propagate down columns and use the observation that $\gamma_{i+1,i+n}={\lambda^i}^*$, or alternatively propagate up columns using the local condition applied northwestward.
\end{proof}

\begin{Prop}\label{prop:diagram bijection}
The first row of an affine growth diagram of type $\vec\lambda$ is a minuscule path of type $\vec\lambda$. More generally, the $i$th row is a minuscule path of type $(\lambda^i,\lambda^{i+1},\ldots,\lambda^n,\lambda^1,\ldots,\lambda^{i-1})$. There is a bijection between affine growth diagrams of type $\vec\lambda$ and minuscule paths of type $\vec\lambda$.
\end{Prop}
\begin{proof}
The first two statements follows from the previous lemma. The last statement follows from the first together with the observation that a growth diagram may be recovered from the first row by applying the local condition to the southeast and northwest.
\end{proof}

\begin{proof}[Proof of Theorem \ref{thm:geometry theorem}]
The first statement follows from Propositions \ref{prop:hive edges} and \ref{prop:hive excavation}. Each component corresponds to a hive such that the weights along the edges $(1,i)$ of the hive are $d(g_1,g_i)$. Bijectivity follows from Proposition \ref{prop:overdetermined} and Proposition \ref{prop:diagram bijection}.
\end{proof}

\begin{Rem}
It was shown in \cite{FKK} via a geometric argument that the components of Poly$(\vec\lambda)$ are indexed by minuscule paths given by the distances $d(g_1,g_j)$ along the fan triangulation. These distances comprise the first row of the corresponding affine growth diagram. As mentioned in Remark \ref{rem:extroverted triangulation}, the $n$-hive is determined by the edges of any extroverted triangulation. Likewise, an affine growth diagram is determined by the weights along any path from southwest diagonal to northeast diagonal. That is, any path $(i_1,j_1),\ldots,(i_{n+1},j_{n+1})$ in $St_n$ such that $(i_1,j_1)=(i,i)$ for some $i$ and for all $k$, $(i_k,j_k)-(i_{k+1},j_{k+1})=(1,0)$ or $(0,-1)$. The weights along such path form a minuscule path whose type is determined by the order of eastward and northward steps. 
\end{Rem}

%%%%%%%%%%%%%%%%% Part II %%%%%%%%%%%%%%%%%%%%%%

\section{Combinatorics of Affine Growth Diagrams \label{combinatorics}}
This section is devoted to showing that the classical Robinson--Schensted correspondence and Fomin growth diagrams appear within affine growth diagrams.

\subsection{Greene's Theorem Interpretation of Fomin Diagrams}\label{sec:combinatorial background}
Recall from the introduction that Fomin growth diagrams are defined by local growth rules. The partitions in a Fomin growth diagram also have a more global interpretation. To see this, recall the following theorem due to Greene \cite{Gre2} (see also \cite{Gre1} and \cite{Sag}). Let $P$ be a finite partially ordered set and let $c_k(P)$ (resp. $a_k(P)$) be the largest size of a union of $k$ disjoint chains (resp. antichains) in $P$.

\begin{Thm}[\cite{Gre2}]\label{thm:Greene}
For any finite partially ordered set $P$ there exists a partition $\lambda=\lambda(P)$ such that $a_k(P)=\lambda_1+\cdots+\lambda_k$ and $c_k(P)=\lambda^t_1+\cdots+\lambda^t_k$ for all $k$, where $\lambda^t$ denotes the transpose of $\lambda$.
\end{Thm}

In what follows we will be interested in applying Greene's theorem to partially ordered sets defined by permutation matrices or partial permutation matrices. For a partial permutation matrix define the partially ordered set on the set of $1$ entries as follows. For two locations $A$ and $B$, each containing a $1$, define $A<B$ if $B$ is southeast of $A$. For a matrix $M$ let $M_\searrow$ denote this poset and $\lambda(M_{\searrow})$ the associated partition given by Theorem \ref{thm:Greene}. Define similarly the partitions $\lambda(M_{\nwarrow}), \lambda(M_{\nearrow}), \lambda(M_{\swarrow})$ associated to the poset on the set of $1$ entires as before, but the larger elements are to the northwest, northeast, southeast respectively. 

\begin{Prop}\label{prop:reflect}
Let $M$ be a permutation matrix and set $\lambda=\lambda(M_{\searrow})$. Then $\lambda=\lambda(M_{\nwarrow})$ and $\lambda^t=\lambda(M_{\swarrow})=\lambda(M_{\nearrow})$.
\end{Prop}
\begin{proof}
Reversing the ordering doesn't affect chains and antichains, whereas the other two reflections swap chains and antichains.
\end{proof}

The following proposition from \cite{Rob} interprets the partitions of a Fomin growth diagram in terms of Greene's theorem.

\begin{Prop}[\cite{Rob}]\label{prop:Roby original}
Let $\alpha_{i,j}$ denote the partition labeling the vertex in row $i$, column $j$ of a Fomin growth diagram and $\pi$ the corresponding permutation matrix. Let $\pi(i,j)$ be the $(i-1) \times (j-1)$ northwest-justified submatrix of $\pi$. Then $\alpha_{i,j}=\lambda(\pi(i,j)_{\searrow})$ for all $i,j$.
\end{Prop}

Hence, Fomin's local rules govern the monotonic growth of the partition $\lambda(\pi_{\searrow})$ as more of the permutation matrix is revealed. Note, in particular, taking a horizontal step from vertex $(i,j)$ to $(i,j+1)$ reveals a new column in the corresponding matrices $\pi(i,j)$ and $\pi(i,j+1)$ and possibly a new maximal element of the associated poset, thereby growing $\alpha_{i,j}$ to $\alpha_{i,j+1}$ by at most one box. Similarly, a vertical step from $(i,j)$ to $(i+1,j)$ reveals a new row.

As mentioned in the introduction, the RS-correspondence can be viewed as a special case of the RSK-correspondence and the Sundaram bijection. We will show how to realize the Sundaram bijection in terms of affine growth diagrams. In \S\ref{sec:filling} we give a rule to assign natural numbers to the unit squares of an affine growth diagrams. For an affine growth diagram of type $\vec\lambda$ where each $\lambda^i$ is $\omega_1$ or $\omega_1^*$ this assignment gives $0,1$ entries, and when $m$ is large this can be interpreted as a fixed-point-free involution.

In \S\ref{sec:main combinatorial theorem} we define a construction of an affine growth diagram starting from natural number entries by applying Greene's theorem to certain submatrices. We then establish the main combinatorial result, Theorem \ref{thm:bijection}, which says that these constructions are inverses for $m$ large. From the construction it follows that the first row of the resulting diagram consists of weights that have no negative parts, hence an oscillating tableau. It follows from the construction that Fomin growth diagrams, and the RS-correspondence, appear within affine growth diagrams when $\vec\lambda=(\omega_1,\ldots,\omega_1,\omega_1^*,\ldots,\omega_1^*)$.

In section \ref{sec:RSK} we drop the $\omega_1,\omega_1^*$ assumption to realize Roby's generalized Fomin diagrams within affine growth diagrams. These diagrams correspond to the RSK-correspondence \cite{Rob}. All three bijections can be seen as specializations of yet a fourth bijection, which we now describe. 
\begin{Def}
Let a \emph{row-strict semistandard oscillating tableau} of length $n$ be a sequence of partitions $(\mu^1=\vec{0},\mu^2,\ldots,\mu^n,\mu^{n+1}=\vec{0})$ such that for all $i$ the partitions $\mu^i$ and $\mu^{i+1}$ differ by a vertical strip. The \emph{content} is the sequence of integers $(a_1,\ldots,a_n)$ where $a_i$ is the signed size of the vertical strip $\mu^{i+1}-\mu^i$. Since we are always dealing with vertical strips, we will just say semistandard oscillating tableau to mean row-strict.
\end{Def}
For an $n\times n$ symmetric matrix $M$ of natural numbers, let $r_i=\sum_{j=i}^n M_{i,j}$ be the sum of the entries in row $i$ from the main diagonal onward. Let $c_i=\sum_{i=1}^j M_{i,j}$ be the sum of the entries in column $i$ from the top entry to the main diagonal. When the first row of an affine growth diagram is a semistandard oscillating tableau, we show in \S\ref{sec:RSK} that the natural number entries give such a matrix. Furthermore, this also has an inverse construction, as stated in Theorem \ref{thm:inverse map}. A corollary is the following theorem, which although stated a little differently, is also proved in \cite{Rob}.

\begin{Thm}[\cite{Rob}]
There is a bijection between semistandard oscillating tableaux of length $n$ and $n\times n$ symmetric matrices such that for all $i$, either $r_i=0$ or $c_i=0$. 
\end{Thm}
Note that the condition $r_i=0$ or $c_i=0$ forces zeroes on the main diagonal. These symmetric matrices can be thought of as the natural-number analogues of fixed-point-free involutions. Although these bijections are not new, we believe that the main construction of \S\ref{sec:filling} is interesting in its own right. As a final corollary, we show in Theorem \ref{thm:generalized Fomin} that the generalized Fomin growth diagrams appear within affine growth diagrams.

\subsection{Natural Number Entries \label{sec:filling}}

To realize the bijections discussed, we assign a natural number to each unit square of the growth diagram that depends only on the four weights labelling the four corners of the square. These entries will be $0$ or $1$ when restricting to the case that each $\lambda^i$ of $\vec\lambda$ is $\omega_1$ or $\omega_1^*$. 

For any dominant weight $\gamma$ let $\gamma^+$ be the weight consisting of positive entries and $\gamma^-$ consisting of negative entries. For example, for the dominant weight $\gamma=(3,1,1,0,0,-2,-4)$, we have $\gamma^+=(3,1,1,0,0,0,0)$ and $\gamma^-=(0,0,0,0,0,-2,-4)$. We will often think of these as a pair of partitions. When defining the inverse construction we will need $m$ to be large enough so that the positive and negative parts do not interact.

Each difference, $\gamma_{i,j}-\gamma_{i+1,j}$, is a vertical strip of some size $k$, so that $\gamma_{i,j}^+, \gamma_{i+1,j}^+$ also differ by a vertical strip of some size $l$ that is at most $k$. Define the integer $v_{i,j}$ to be $l$ if the difference $\gamma_{i,j}^+-\gamma_{i+1,j}^+$ is a positive strip and $-l$ if it is a negative strip. 
\begin{Def}
Let $n_{i,j}=\lvert v_{i,j}-v_{i,j+1}\rvert$ be the natural number associated to the unit square with vertices $(i,j), (i,j+1), (i+1,j), (i+1,j+1)$. Likewise, define $h_{i,j}$ to be the signed size of the vertical strip $\gamma_{i,j}^+-\gamma_{i,j+1}^+$, and $m_{i,j}=\lvert h_{i,j}-h_{i+1,j}\rvert$.
\end{Def}

\begin{Lem}\label{lem:natural entries}
For an affine growth diagram, the natural numbers $n_{i,j}$ and $m_{i,j}$ are equal for all indices $i$ and $j$. The $n_{i,j}$ are $n$-periodic, that is $n_{i,j}=n_{i+n,j+n}$. Furthermore, there is a symmetry, $n_{i,j}=n_{j,i+n}$ for all $i$ and $j$.
\end{Lem}
\begin{proof}
There are two paths from vertex $(i,j)$ to $(i+1,j+1)$ around the corresponding unit square, implying that $v_{i,j}+h_{i+1,j}=h_{i,j}+v_{i,j+1}$. Hence, $v_{i,j}-v_{i,j+1}=h_{i,j}-h_{i+1,j}$, so in particular $n_{i,j}=m_{i,j}$. The periodicity follows from the periodicity of $\gamma_{i,j}$. The symmetry statement follows from the dual symmetry $\gamma_{j,i+n}=\gamma_{i,j}^*$ of Corollary \ref{cor:symmetry}.
\end{proof}

A convenient way to visualize the last symmetry statement of the proposition is to place the dual fundamental triangular region $\{(i,j)\mid 1\leq i\leq n+1,\; n+1\leq j\leq i+n\}$, along with its entries $n_{i,j}$, underneath the fundamental triangular region consisting of the vertices $\{(i,j)\mid 1\leq i\leq n+1,\; i\leq j\leq n+1\}$. This identifies the vertices $(i,i)$ and $(i,n+i)$, giving an $n\times n$ symmetric matrix with zero entries along the main diagonal, \emph{the matrix associated to the affine growth diagram}. Technically, the diagonal $n_{i,i}$ are not even defined, so we set them to zero.

\begin{Rem}
We suspect that it may be interesting to allow negative entries by removing the absolute value in the definitions of $n_{i,j}$ and $m_{i,j}$, but it is not necessary for our goal of realizing the Robinson--Schensted correspondence and we do not explore this in the present work. 
\end{Rem}

When $m$ is large enough, the fundamental weights give the row and column sums of the $n_{i,j}$ entries. Informally speaking, taking $m$ large prevents the positive and negative partitions from interacting other than by a large sorting.

\begin{Prop}
Fix a sequence of nonzero integers $(k_1,\ldots,k_n)$ such that $\sum_{k_i>0}k_i=\sum_{k_i<0}\lvert k_i\rvert$ and let $m$ be at least $\sum_{k_i>0}k_i$. Let $\lambda^i=\omega_{k_i}$ if $k_i$ is positive and let $\lambda^i=\omega_{\lvert k_i\rvert}^*$ otherwise. For an affine growth diagram of type $\vec\lambda$, for all $i$ the row sum $\sum_{j=i+1}^{i+n-1} n_{i,j}$ is equal to $k_i$. Likewise, for all $i$ the column sum is $\sum_{j=i+1}^{i+n-1} n_{j,i+n}=k_i$.
\end{Prop}
\begin{proof}
The condition $\sum_{k_i>0}k_i=\sum_{k_i<0}\lvert k_i\rvert$ is necessary for the existence of an affine growth diagram of type $\vec\lambda$.

Fix an index $i$ and set $k=k_i$. Suppose that $k$ is positive, so $\lambda^i=\omega_k$ (the negative case is analogous). By Lemma \ref{Lem:edge differences}, for all $j$ there exists a $w_j\in S_n$ such that $\gamma_{i,j}-\gamma_{i+1,j}=w_j\cdot \omega_k$ is a positive, vertical strip of size $k$. In particular, for $j=i+1$, $\gamma_{i,i+1}=\lambda^i=\omega_k$ and $\gamma_{i+1,i+1}=\vec{0}$, so $v_{i,i+1}=k$. We claim that $v_{i,j}$ is monotonically decreasing from $k$ to $0$ as $j$ increases from $i+1$ to $i+n-1$, from which the result follows. In particular, $v_{i,i+n-1}=0$ because $\gamma_{i,i+n}=\vec{0}$ and $\gamma_{i+1,i+n}=\omega_k^*$. 

Suppose the claim holds up to index $j$ and that $v_{i,j}=l$. Let $A\subset [m]$ be the set of $l$ indices where $\gamma_{i,j}^+$ and $\gamma_{i+1,j}^+$ differ. Likewise, let $B\subset [m]$ be the set of $k-l$ indices where $\gamma_{i,j}^-$ and $\gamma_{i+1,j}^-$ differ. Then the local condition, $\gamma_{i+1,j+1}=\text{sort}(\gamma_{i,j+1}-(\gamma_{i,j}-\gamma_{i+1,j}))$, implies that $1$ is subtracted from $\gamma_{i,j+1}$ in positions $A$ and positions $B$, followed by a possible sorting. Hence, the only way for $v_{i,j+1}$ to be strictly greater than $v_{i,j}$ is if $\gamma_{i,j+1}$ has a positive entry for some index $b\in B$. For $m$ sufficiently large this is impossible. One can show that $m\geq \sum_{k_i>0}k_i$ suffices. For the column sum statement, notice that $\gamma_{i+1,i+n}={\lambda^i}^*=\omega_k^*$, and apply the same reasoning together with the previous lemma.
\end{proof}

Let us now restrict to the case that each $\lambda^i$ is equal to $\omega_1$ or $\omega_1^*$ and $m$ large. In this case, the previous proposition says that every row and column has exactly one entry equal to $1$. Recall that an \emph{affine permutation} is a bijection $f:\mathbb Z\rightarrow \mathbb Z$ such that $f(x+n)=f(x)+n$. In particular, an affine permutation is determined by the values $f(1),\ldots, f(n)$.

\begin{Cor}
If $\vec\lambda$ is a sequence of weights such that each $\lambda^i$ is $\omega_1$ or $\omega_1^*$, then for $m\geq n/2$ each row and column of an affine growth diagram of type $\vec\lambda$ has exactly one $n_{i,j}$ entry equal to $1$ and all others equal to $0$. The function $f:\mathbb Z\rightarrow \mathbb Z$ defined by $f(i)=j$ if $n_{i,j}=1$ is an affine permutation satisfying $f\circ f(i)=i+n$.
\end{Cor}
\begin{proof}
The first part is just a special case of the previous proposition. The condition $f(i+n)=f(i)+n$ is satisfied by the periodicity of affine growth diagrams. The last part follows from the symmetry statement, $n_{i,j}=n_{j,i+n}$ of Lemma \ref{lem:natural entries}.
\end{proof}

From the affine permutation $f$ define a permutation $\pi \in S_n$ by $\pi(i)=f(i) \mod n$. As previously mentioned, this is the $n\times n$ permutation matrix constructed by placing the dual fundamental triangular region underneath the fundamental region. The permutation $\pi$ is a fixed-point-free involution. It is fixed-point-free because $n_{i,i}=0$ for all $i$, and it is an involution because the corresponding permutation matrix is symmetric.

This defines a map from such affine growth diagrams to fixed-point-free involutions. This map is not a bijection. For example, the two growth diagrams with first line labelled by 
\begin{align*}
(0,\ldots,0),(1,0,\ldots,0),(2,0,\ldots,0,0),(1,0,\ldots,0),(0,\ldots,0)\\
(0,\ldots,0),(1,0,\ldots,0),(1,0,\ldots,0,-1),(1,0,\ldots,0),(0,\ldots,0) 
\end{align*}
both map to the involution $1\leftrightarrow 4$, $2\leftrightarrow 3$. However, the map is a bijection if restricted to affine growth diagrams such that the first line is labelled by an oscillating tableau, that is, if $\gamma_{1,j}=\gamma_{1,j}^+$ for all $j$.

\begin{Def}
Let $\mathcal A_{n,m}$ be the set of $GL_m$ affine growth diagrams of type $(\lambda^1,\ldots,\lambda^n)$ such that each $\lambda^i$ is $\omega_1$ or $\omega_1^*$, and $\gamma_{1,j}=\gamma_{1,j}^+$ for all $j$. Let $\mathcal{FI}_n$ denote the set of fixed-point-free involutions in $S_n$. Let $\Phi_{n,m}:\mathcal A_{n,m}\rightarrow \mathcal{FI}_n$  for $m\geq n/2$ denote the map that associates the above fixed-point-free involution to each affine growth diagram. 
\end{Def}
We will show that this map has an inverse.

\subsection{Constructing Affine Growth Diagrams from Natural Numbers \label{sec:main combinatorial theorem}}

We will now construct an affine growth diagram from a fixed-point-free involution and show that the construction is an inverse to $\Phi_{n,m}$. The construction applies Greene's theorem (Theorem \ref{thm:Greene}) to the posets corresponding to certain submatrices of the staircase diagram defined below. From this construction it will follow that Fomin growth diagrams appear within affine growth diagrams.

\begin{Def}
To define the inverse map $\Psi_{n,m}:\mathcal{FI}_n\rightarrow \mathcal A_{n,m}$ for $m\geq n/2$ begin with an empty staircase diagram and a fixed-point-free involution $\pi$. For indices $(i,i+k)$ such that $1\leq i\leq n$ and $0\leq k\leq n$, set $n_{i,i+k}=1$ if $\pi(i)=k$ and $0$ otherwise. Extend to the infinite staircase periodically. 

To define $\gamma_{i,j}$ we will define $\gamma_{i,j}^+$ and $\gamma_{i,j}^-$ individually. For $\gamma_{i,j}^+$ consider the $(j-i)\times (n+1-j)$ matrix $M(i,j)^+$ consisting of entries $n_{a,b}$ of the rectangular region 
\begin{align*}
\{(a,b)\mid i\leq a\leq j-1,\; j\leq b\leq n \}. 
\end{align*}
This matrix is contained in the largest rectangle with northwestern-most vertex $(i,j)$ that fits inside of the fundamental triangular region (see Figure \ref{rectangle domain}). Since $M(i,j)^+$ is a partial permutation matrix, we can interpret the set of $1$ entries as a partially ordered set where for two locations $(i_1,j_1), (j_1,j_2)$ that contain a $1$ entry, $(i_1,j_1)<(i_2,j_2)$ if $i_1> i_2$ and $j_1> j_2$. That is, larger elements are to the northwest. Define $\gamma_{i,j}^+$ to be the $GL_m$ weight with parts given by the partition $\lambda(M(i,j)^+_{\nwarrow})$ according to Greene's theorem. 

To define $\gamma_{i,j}^-$ consider the matrix $M(i,j)^-$ consisting of entries $n_{a,b}$ of the rectangular region 
\begin{align*}
\{(a,b)\mid 1\leq a\leq i-1,\; j-i\leq b\leq j-1\}. 
\end{align*}
This is an $i\times(j-i)$ matrix, the largest matrix with southeastern-most vertex $(i,j)$ that fits inside the fundamental triangular region. Let $\gamma_{i,j}^-$ be the $GL_m$ weight whose negative entries are given by the partition $\lambda(M(i,j)^-_\searrow)$.

Finally, define $\gamma_{i,j}$ to be $\gamma_{i,j}^++{\gamma_{i,j}^-}$, the dominant $GL_m$ weight with positive (resp. negative) parts given by $\gamma_{i,j}^+$ (resp. $\gamma_{i,j}^-$), which is guaranteed to be well-defined if $m\geq n/2$, so that there are no collisions between positive and negative entries. 
\end{Def}

\begin{figure}
\center
\scalemath{.6}{
\begin{tikzpicture}
\draw[step=1cm] (-5,2) grid (0,3);
\draw[step=1cm] (-4,1) grid (1,2);
\draw[step=1cm] (-3,0) grid (2,1);
\draw[step=1cm] (-2,-1) grid (3,0);
\draw[step=1cm] (-1,-2) grid (4,-1);
\draw[step=1cm] (0,-3) grid (5,-2);
\draw (0,-3)--(6,-3);
\draw (-6,3)--(0,3);

\fill[opacity=.4] (-3,2) -- (-3,3) -- (-5,3) -- (-5,2) -- cycle;
\fill[opacity=.4] (-3,2) -- (0,2) -- (0,0) -- (-3,0) -- cycle;

\node[above right] at (-3,2) {\scalemath{1.4}{\gamma_{2,4}}};
\node[circle, fill] at (-3,2) {};

\node at (-.8,-.3) {$\bf M(2,4)^+$};
\node at (-4,3.3) {$\bf M(2,4)^-$};

\end{tikzpicture}}
\caption{The rectangles corresponding to the partial permutation matrices $M(2,4)^-$ and $M(2,4)^+$ for a staricase diagram with $n=6$.\label{rectangle domain}}
\end{figure}

Theorem \ref{thm:bijection} below states that this construction does indeed produce an affine growth diagram. The proof will require a technical lemma regarding the way the partition associated to a finite poset by Greene's theorem evolves under addition of a new extremal element. This was proved by Thomas Roby in his thesis \cite{Rob}, which is stated below without proof. Let $P_{11}$ be a poset containing $P_{00}$ such that $P_{01}=P_{00}\cup e_1$, $P_{10}=P_{00}\cup e_2$, and $P_{11}=P_{00}\cup e_1\cup e_2$ where $e_1\not=e_2$. For $i,j=0,1$ let $\lambda_{ij}$ be the partition corresponding to the poset $P_{i,j}$ according to Greene's theorem. For two partitions $\mu\subseteq \nu$, let $\nu/\mu$ denote the skew diagram of the cells in $\nu$ but not in $\mu$. When referring to a cell $(x,y)$ of a partition, $x$ denotes the row and $y$ the column.

\begin{Lem}[{\cite[A.3.1]{Rob}}]{\label{lem:Roby lemma}}
Assume $\lambda_{01}=\lambda_{10}$. Let $a=(x_a,y_a)$ be the cell $\lambda_{01}/\lambda_{00}$ and $b=(x_b,y_b)$ be the cell $\lambda_{11}/\lambda_{01}$.
\begin{itemize}
\item
If $e_1$ and $e_2$ are extremal elements of $P_{11}$ of different types, then $x_b=x_a$ or $x_b=x_a+1$.
\item
If $e_1$ and $e_2$ are extremal elements of $P_{11}$ of the same type, then $x_b\leq x_a$.
\end{itemize}
\end{Lem}

The main idea in the proof of the following theorem is a variation on an argument from \cite[A.3.3]{Rob}.

\begin{Thm}\label{thm:bijection}
The dominant weights $\Psi_{n,m}(\pi)=\{\gamma_{i,j}\}_{(i,j)\in St_n}$ defined from a fixed-point-free involution $\pi$ form an affine growth diagram. The type of $\Psi_{n,m}(\pi)$ is $\vec\lambda$ where $\lambda^i=\omega_1$ if $\pi(i)>i$ and $\lambda^i=\omega_1^*$ otherwise. Furthermore, $\Phi_{n,m}\circ \Psi_{n,m}(\pi)=\pi$, and the maps are bijections for $m\geq n/2$.
\end{Thm}
\begin{proof}
We must show that the local condition is satisfied for each unit square, so let $X$ be a unit square with entry $n_{i,j}$ and dominant weights $\nu=\gamma_{i,j}, \mu=\gamma_{i,j+1}, \lambda=\gamma_{i+1,j}, \rho=\gamma_{i+1,j+1}$. Assume without loss of generality that $X$ is contained in the fundamental triangular region, and let $X^*$ denote the dual unit square in the dual fundamental region, with corresponding entry $n_{j,i+n}=n_{i,j}$. 

Consider the row containing $X$. Let $A$ be the portion of the row before $X$, $B$ the portion of the row after $X$ contained in the fundamental region, and $C^*$ the portion of the row after $X$ contained in the dual fundamental region. See Figure \ref{9 cases}. Let $C$ be the column portion in the fundamental region that is dual to the region $C^*$. The row containing $X$ has exactly one entry equal to $1$, which by the dual symmetry is equivalent to the L-shape $C, A, X, B$ in the fundamental region containing exactly one entry equal to $1$. Similarly, consider the column containing $X$. Let $D$ be the portion below $X$, $F$ the portion above $X$ contained in the fundamental region, and $E$ the row in the fundamental region completing the L-shape $F,X,D,E$. The L-shape $F, X, D, E$ contains exactly one entry equal to $1$. 

Either $X$ and $X^*$ both contain a $1$, or neither do. If neither do, then there are $9$ possibilities depending on whether $A,B$ or $C$ contains a $1$ and whether $D,E$ or $F$ contains a $1$. We will check the cases that $X$ or $B$ contain a $1$.

Suppose first that the unit square $X$ contains a $1$, so that the regions $A,B,C,D,E,F$ only contain zeroes. Let $\nu^{\pm}, \lambda^{\pm}, \mu^{\pm}, \rho^{\pm}$ be the weights defined by the submatrices $M(i,j)^{\pm}$, $M(i+1,j)^{\pm}$, $M(i,j+1)^{\pm}$, $M(i+1,j+1)^{\pm}$. We will interpret the weights as partitions. Since $M(i,j)^+$ contains a $1$ at position $X$ that is not contained in any of $M(i+1,j)^+, M(i,j+1)^+, M(i+1,j+1)^+$, it follows that $\nu^+$ covers $\lambda^+=\mu^+=\rho^+$ in Young's lattice. Likewise, $\rho^-$ cover $\lambda^-=\mu^-=\nu^-$ because $M(i+1,j+1)^-$ contains a $1$ at position $X$ not contained in $M(i,j)^-, M(i+1,j)^-, M(i,j+1)^-$.

Since the $1$ entry in square $X$ dominates all other $1$ entries in $M(i,j)^+$, it must be part of the longest chain. Therefore, $\nu^+/\lambda^+$ is the last box in the first column of $\nu^+$, say in row $k$. Similarly, $\rho^-/\lambda^-$ is the last box of the first column of $\rho^-$, say in row $l$. Therefore, the difference $\nu-\lambda$ contains a single $1$ in position $k$ since $\nu^-=\lambda^-$. The weight $\mu$ contains a $0$ in position $k$ since $\mu^+=\lambda^+$, so the difference $\mu-(\nu-\lambda)$ contains $-1$ in position $k$. This $-1$ gets sorted to position $m+1-l$, becoming the corresponding box in $\rho^-/\lambda^-$. Hence, $\rho=\text{sort}(\mu+\lambda-\nu)$.

Now suppose that $X$ does not contain a $1$, but region $B$ does. There are three subcases depending on which region $D, E$ or $F$ contains a $1$. Suppose that region $E$ contains a $1$, so that $A,C,D,F$ only contain zeros. The matrices $M(i,j)^-,M(i,j+1)^-,M(i+1,j)^-,M(i+1,j+1)^-$ induce the same poset, so $\lambda^-=\mu^-=\nu^-=\rho^-$. Hence, we only need to compare the positive partitions to see that the local condition is satisfied. 

Let $e_1$ (resp. $e_2$) represent the $1$ entry of region $B$ (resp. $E$) as an element of the poset $M(i,j+1)^+_\nwarrow$. We claim that 
\begin{align*}
M(i,j)^+_\nwarrow&=M(i+1,j)^+_\nwarrow\cup e_1,\\
M(i+1,j+1)^+_\nwarrow&=M(i+1,j)^+_\nwarrow\cup e_2\\
M(i,j+1)^+_\nwarrow&=M(i+1,j)^+_\nwarrow\cup e_1\cup e_2.
\end{align*}
This follows because the matrices $M(i,j+1)^+, M(i+1,j+1)^+$ contain region $E$ but not $D$ (which is empty), whereas the matrices $M(i,j)^+, M(i+1,j)^+$ contain $D$ but not $E$. Similarly, $M(i,j)^+, M(i,j+1)^+$ contain $B$, but $M(i+1,j)^+, M(i+1,j+1)^+$ do not. In terms of partitions, $\mu^+$ covers both $\nu^+$ and $\rho^+$, which both cover $\lambda^+$.

If $\nu^+$ and $\rho^+$ are not equal, then it must be that $\mu^+=\nu^+\cup \rho^+$ and $\lambda^+=\nu^+\cap \rho^+$ simply by the properties of Young's lattice. This gives two paths of length two in Young's lattice with common start and end points, so if $\nu^+$ and $\lambda^+$ differ in row $i$, then so must $\mu^+$ and $\rho^+$, so the local condition is satisfied.

Now consider the case that $\nu^+=\rho^+$. We have $M(i,j+1)^+_\nwarrow=M(i+1,j)^+_\nwarrow\cup e_1\cup e_2$ where $e_1$ and $e_2$ where $e_1$ is maximal and $e_2$ is minimal. Let $a=(x_a,y_a)$ and $b=(x_b,y_b)$ be the cells of $\nu^+/\lambda^+$ and $\mu^+/\nu^+$ respectively. By the first part of Lemma \ref{lem:Roby lemma}, $x_b=x_a$ or $x_b=x_a+1$. If $x_b=x_a$, then both boxes were added to the same row. Thus, $\nu^+-\lambda^+=\mu^+-\rho^+$, so the local condition is satisfied. Now suppose that $x_b=x_a+1$. Consider the dual order posets, $M(i,j)^+_\nearrow,M(i,j+1)^+_\nearrow,M(i+1,j)^+_\nearrow,M(i+1,j+1)^+_\nearrow$. By Proposition \ref{prop:reflect} the associated partitions given by Greene's theorem are the transposes of the original $\lambda^+, \mu^+,\nu^+,\rho^+$. The extremal elements $e_1$ and $e_2$ are still different types, so by another application of Lemma \ref{lem:Roby lemma}, $y_b=y_a$ or $y_b=y_a+1$. Since $y_b=y_a+1$ and $x_b=x_a+1$ cannot simultaneously hold, we conclude $y_b=y_a$. This means that two boxes were added in the same column, one in row $x_a$ and the next in row $x_b=x_a+1$. Then $\nu^+-\lambda^+$ contains a single $1$ in position $x_a$, implying that row $x_a+1$ of $\mu^+-(\nu^+-\lambda^+)$ is one greater than row $x_a$, which after sorting, gives $\rho^+$.

The case that $B$ and $D$ both contain a $1$ has a similar argument to the previous one, except now the $1$ entries in $B$ and $D$ are both maximal in the poset $M(i,j)^+_{\nwarrow}$. This is the argument given in \cite{Rob} to prove Proposition \ref{prop:Roby original}. Finally, if $B$ and $F$ contain a $1$, then $M(i,j)^+_{\nwarrow}=M(i,j+1)^+_{\nwarrow}$ and $M(i+1,j)^+_{\nwarrow}=M(i+1,j+1)^+_{\nwarrow}$, so $\nu^+=\mu^+$, $\lambda^+=\rho^+$, and hence $\nu^+-\lambda^+=\mu^+-\rho^+$. By similar reasoning, $\nu^-=\lambda^-$ and $\mu^-=\rho^-$, so $\nu-\lambda=\mu-\rho$. 

The other six cases are similar, or follow from dual symmetry. The statement about the $\lambda^i$ follows from considering $M(i,i+1)^+$ and $M(i,i+1)^-$. That $\Phi_{n,m}\circ \Psi_{n,m}(\pi)=\pi$ follows by examining the above cases. Since a diagram is determined by it's first row, these maps must be bijections by Sundaram's bijection.
\end{proof}
 
\begin{Rem}
One can show directly that $\Psi_{n,m}\circ \Phi_{n,m}$ is also the identity, thereby recovering Sundaram's bijection.
\end{Rem}
 
\begin{figure}
\centering
\scalemath{.5}{
\begin{tikzpicture}
\draw (0,0) -- (6,0) -- (12,-6) -- (6,-6) -- (0,0);
\draw (6,0) -- (6,-6);

\draw (1,-1) -- (7,-1);
\draw (1.5,-1.5) -- (7.5,-1.5);

\draw (3.5,-3.5) -- (9.5,-3.5);
\draw (4,-4) -- (10,-4);

\draw (1,0) -- (1,-1);
\draw (1.5,0) -- (1.5,-1.5);

\draw (3.5,0) -- (3.5,-3.5);
\draw (4,0) -- (4,-4);

\draw (7,-1) -- (7,-6);
\draw (7.5,-1.5) -- (7.5,-6);

\draw (9.5,-3.5) -- (9.5,-6);
\draw (10,-4) -- (10,-6);

\node at (3.75,-1.25) {X};
\node at (3.75,-2.5) {D};
\node at (5,-3.75) {E};
\node at (5,-1.25) {B};

\node at (7.25,-3.75) {X$^*$};
\node at (7.25,-2.5) {A$^*$};
\node at (6.5,-3.75) {F$^*$};
\node at (6.5,-1.25) {C$^*$};

\node at (2.5,-1.25) {A};
\node at (1.25,-.5) {C};
\node at (3.75,-.5) {F};

\node at (9.75,-5) {E$^*$};
\node at (7.25,-5) {B$^*$};
\node at (8.5,-3.75) {D$^*$};

\end{tikzpicture}
}
\caption{The regions $A, B, C, D, E, F, X$, and their duals, in the proof of Theorem \ref{thm:bijection}. \label{9 cases}}
\end{figure}

Let $n=2k$ and $\vec\lambda=(\omega_1,\ldots,\omega_1,\omega_1^*,\ldots,\omega_1^*)$. As previously mentioned, in this case the first line of affine growth diagrams is an oscillating tableau that can be interpreted as a pair of same-shape standard Young tableaux, the two increasing chains of partitions meeting in the middle. By the last part of Theorem \ref{thm:bijection}, all of the $1$ entries in the corresponding affine growth diagram appear in two $k\times k$ matrices consisting of the $n_{i,j}$ for the indices $\{(i,j)\mid 1\leq i\leq k, k+1\leq j\leq n\}$ and $\{(i,j)\mid k+1\leq i\leq n, n+1\leq j\leq n+k\}$. These are the $k\times k$ squares in the fundamental triangular region and dual triangular region. See Figure \ref{fig:RS} for an example with $n=6$ of the following theorem.

\begin{Thm}{\label{thm:Fomin contained}}
Let $n=2k$ and $\vec\lambda=(\omega_1,\ldots,\omega_1,\omega_1^*,\ldots,\omega_1^*)$. Let $\{\gamma_{i,j}\}_{(i,j)\in St_n}$ be an affine growth diagram of type $\vec\lambda$ for $m\geq n/2$. Then the partitions $\gamma_{i,j}^+$ for $1\leq i\leq k+1$ and $k+1\leq j\leq n+1$ form a Fomin growth diagram, growing from the southeast to the northwest, and the partitions $\gamma_{i,j}^-$ for the same indices form a Fomin growth diagram, growing from the northwest to the southeast. Similarly for the indices $k+1 \leq i\leq n+1$ and $n+1\leq j\leq n+1+k$ the partitions $\gamma_{i,j}^+$ form a Fomin growth diagram from the northwest to southeast and the $\gamma_{i,j}^-$ from southeast to northwest.
\end{Thm} 
\begin{proof}
By the previous theorem all of the $n_{i,j}$ entries that are equal to $1$ appear in the two $k\times k$ squares contained in the two triangular regions. The definition of the $\gamma_{i,j}^+$ according to $\Psi_{n,m}$ on the vertices of this $k\times k$ array coincides with the global definition of Fomin growth diagrams in Proposition \ref{prop:Roby original}.
\end{proof}

\begin{Rem}
This theorem realizes the Robinson--Schensted bijection as a restriction of the Sundaram bijection. Let $P,Q$ be a pair of same-shape standard Young tableaux interpreted as an oscillating tableau written along the first line of an affine growth diagram. Let $\pi\in S_n$ be the corresponding fixed-point-free involution and $\sigma$ the permutation given by the $k\times k$ permutation matrix in the fundamental region. Then $\sigma$ corresponds to $(ev(P),Q)$ under our notion of the column-insertion Robinson--Schensted correspondence where $ev(P)$ denotes Sch\"utzenberger evacuation. See Figure \ref{fig:RS}.
\end{Rem}

\begin{figure}[h!]
\centering
\begin{minipage}{.49\textwidth}
\scalemath{.7}{
\begin{tikzpicture}
\draw[step=1cm] (-5,2) grid (0,3);
\draw[step=1cm] (-4,1) grid (1,2);
\draw[step=1cm] (-3,0) grid (2,1);
\draw[step=1cm] (-2,-1) grid (3,0);
\draw[step=1cm] (-1,-2) grid (4,-1);
\draw[step=1cm] (0,-3) grid (5,-2);
\draw (0,-3)--(6,-3);
\draw (-6,3)--(0,3);

\draw[blue,ultra thick] (0,0)--(0,-3)--(3,-3)--(3,0)--cycle;
\draw[red,ultra thick] (0,0)--(0,3)--(-3,3)--(-3,0)--cycle;
%\Ylinethick1.5pt
\Yboxdim{2mm}
\node at (6.2,-2.8) {$\scalemath{1}{\vec{0}}$};
\node at (5.2,-2.8) {\yng(1)};
\node at (4.3,-2.8) {\yng(2)};
\node at (3.3,-2.8) {\yng(3)};
\node at (2.3,-2.8) {\yng(2)};
\node at (1.2,-2.8) {\yng(1)};
\node at (.2,-2.8) {$\scalemath{1}{\vec{0}}$};
\node at (.6,-2.5) {\scalemath{1}{1}};
\node at (5.2,-1.8) {$\scalemath{1}{\vec{0}}$};
\node at (4.2,-1.8) {\yng(1)};
\node at (3.3,-1.8) {\yng(2)};
\node at (2.3,-1.8) {\yng(1)};
\node at (1.3,-1.8) {$\scalemath{1}{\vec{0}}$};
\node at (0.2,-1.5) {$\young(:\empty,:\empty,:\empty,\empty)$};
\node at (-.8,-1.8) {$\scalemath{1}{\vec{0}}$};  
\node at (1.6,-1.5) {\scalemath{1}{1}};
\node at (4.2,-.8) {$\scalemath{1}{\vec{0}}$};
\node at (3.2,-.8) {\yng(1)};
\node at (2.3,-.8) {$\scalemath{1}{\vec{0}}$};
\node at (1.3,-.5) {$\young(:\empty,:\empty,:\empty,\empty)$};
\node at (0.3,-.5) {$\young(::\empty,::\empty,::\empty,~~)$};
\node at (-.8,-.5) {$\young(:\empty,:\empty,:\empty,\empty)$};
\node at (-1.8,-.8) {$\scalemath{1}{\vec{0}}$};  
\node at (2.6,-.5) {\scalemath{1}{1}};
\node at (3.2,.2) {$\scalemath{1}{\vec{0}}$};
\node at (2.2,.5) {$\young(:\empty,:\empty,:\empty,~)$};
\node at (1.3,.5) {$\young(::\empty,::\empty,::\empty,~~)$};
\node at (0.325,.5) {$\young(:::\empty,:::\empty,:::\empty,~~~)$};
\node at (-.7,.5) {$\young(::\empty,::\empty,::\empty,~~)$};
\node at (-1.8,.5) {$\young(:\empty,:\empty,:\empty,\empty)$};
\node at (-2.8,.2) {$\scalemath{1}{\vec{0}}$};  
\node at (-2.5,.5) {\scalemath{1}{1}};
%\node[circle,fill,inner sep=2pt] at (0,0) {};
\node at (2.2,1.2) {$\scalemath{1}{\vec{0}}$};
\node at (1.2,1.5) {$\young(:\empty,:\empty,:\empty,~)$};
\node at (0.3,1.5) {$\young(::\empty,::\empty,::\empty,~~)$};
\node at (-.7,1.5) {$\young(:\empty,:\empty,:\empty,~)$};
\node at (-1.7,1.2) {$\scalemath{1}{\vec{0}}$};
\node at (-2.8,1.2) {$\yng(1)$};
\node at (-3.8,1.2) {$\scalemath{1}{\vec{0}}$};  
\node at (-1.5,1.5) {\scalemath{1}{1}};
\node at (1.2,2.2) {$\scalemath{1}{\vec{0}}$};
\node at (0.2,2.5) {$\young(:\empty,:\empty,:\empty,~)$};
\node at (-.7,2.2) {$\scalemath{1}{\vec{0}}$};
\node at (-1.7,2.2) {$\yng(1)$};
\node at (-2.7,2.2) {$\yng(2)$};
\node at (-3.8,2.2) {$\yng(1)$};
\node at (-4.8,2.2) {$\scalemath{1}{\vec{0}}$}; 
\draw (0,-3)--(6,-3);
\node at (-.5,2.5) {\scalemath{1}{1}};
\node at (0.2,3.2) {$\scalemath{1}{\vec{0}}$};
\node at (0-.7,3.2) {$\yng(1)$};
\node at (-1.7,3.2) {$\yng(2)$};
\node at (-2.7,3.2) {$\yng(3)$};
\node at (-3.7,3.2) {$\yng(2)$};
\node at (-4.8,3.2) {$\yng(1)$};
\node at (-5.8,3.2) {$\scalemath{1}{\vec{0}}$}; 
\draw (0,-3)--(6,-3);
\draw (-6,3)--(0,3);

\node at (-3.5,-1) {$\sigma=321$};
\node at (3,2) {$\pi=654321$};
\end{tikzpicture}}
\scalemath{.7}{
\begin{tikzpicture}
\draw[step=1cm] (-5,2) grid (0,3);
\draw[step=1cm] (-4,1) grid (1,2);
\draw[step=1cm] (-3,0) grid (2,1);
\draw[step=1cm] (-2,-1) grid (3,0);
\draw[step=1cm] (-1,-2) grid (4,-1);
\draw[step=1cm] (0,-3) grid (5,-2);
\draw (0,-3)--(6,-3);
\draw (-6,3)--(0,3);

\draw[blue,ultra thick] (0,0)--(0,-3)--(3,-3)--(3,0)--cycle;
\draw[red,ultra thick] (0,0)--(0,3)--(-3,3)--(-3,0)--cycle;
\Yboxdim{2mm}
\node at (6.2,-2.8) {$\scalemath{1}{\vec{0}}$};
\node at (5.2,-2.8) {\yng(1)};
\node at (4.2,-2.7) {\yng(1,1)};
\node at (3.2,-2.65) {\yng(1,1,1)};
\node at (2.2,-2.7) {\yng(1,1)};
\node at (1.2,-2.8) {\yng(1)};
\node at (.2,-2.8) {$\scalemath{1}{\vec{0}}$};
\node at (2.5,-2.5) {\scalemath{1}{1}};
\node at (5.2,-1.8) {$\scalemath{1}{\vec{0}}$};
\node at (4.2,-1.8) {\yng(1)};
\node at (3.3,-1.8) {\yng(1,1)};
\node at (2.3,-1.65) {$\young(:~,:~,~)$};
\node at (1.25,-1.65) {$\young(:~,:\empty,~)$};
\node at (0.2,-1.65) {$\young(:\empty,:\empty,\empty)$};
\node at (-.8,-1.8) {$\scalemath{1}{\vec{0}}$};  
\node at (1.6,-1.5) {\scalemath{1}{1}};
\node at (4.2,-.8) {$\scalemath{1}{\vec{0}}$};
\node at (3.2,-.8) {\yng(1)};
\node at (2.3,-.65) {$\young(:~,:\empty,~)$};
\node at (1.3,-.65) {$\young(:~,~,~)$};
\node at (0.25,-.65) {$\young(:\empty,~,~)$};
\node at (-.8,-.65) {$\young(:\empty,:\empty,\empty)$};
\node at (-1.8,-.8) {$\scalemath{1}{\vec{0}}$};  
\node at (.6,-.5) {\scalemath{1}{1}};
\node at (3.2,.2) {$\scalemath{1}{\vec{0}}$};
\node at (2.2,.35) {$\young(:\empty,:\empty,~)$};
\node at (1.3,.35) {$\young(:\empty,~,~)$};
\node at (0.3,.4) {$\young(:\empty,~,~,~)$};
\node at (-.8,.35) {$\young(:\empty,~,~)$};
\node at (-1.8,.35) {$\young(:\empty,:\empty,\empty)$};
\node at (-2.8,.2) {$\scalemath{1}{\vec{0}}$};  
\node at (-2.5,2.5) {\scalemath{1}{1}};
%\node[circle,fill,inner sep=2pt] at (0,0) {};
\node at (2.2,1.2) {$\scalemath{1}{\vec{0}}$};
\node at (1.2,1.3) {$\young(:\empty,:\empty,~)$};
\node at (0.2,1.3) {$\young(:\empty,~,~)$};
\node at (-.78,1.3) {$\young(:~,~,~)$};
\node at (-1.8,1.3) {$\young(:~,:\empty,~)$};
\node at (-2.85,1.15) {$\yng(1)$};
\node at (-3.8,1.2) {$\scalemath{1}{\vec{0}}$};  
\node at (-1.4,1.5) {\scalemath{1}{1}};
\node at (1.2,2.2) {$\scalemath{1}{\vec{0}}$};
\node at (0.2,2.3) {$\young(:\empty,:\empty,~)$};
\node at (-.78,2.325) {$\young(:~,:\empty,~)$};
\node at (-1.78,2.325) {$\young(:~,:~,~)$};
\node at (-2.85,2.225) {$\yng(1,1)$};
\node at (-3.8,2.2) {$\yng(1)$};
\node at (-4.8,2.2) {$\scalemath{1}{\vec{0}}$}; 
\draw (0,-3)--(6,-3);
\node at (-.5,.5) {\scalemath{1}{1}};
\node at (0.2,3.2) {$\scalemath{1}{\vec{0}}$};
\node at (0-.8,3.15) {$\yng(1)$};
\node at (-1.8,3.25) {$\yng(1,1)$};
\node at (-2.8,3.35) {$\yng(1,1,1)$};
\node at (-3.8,3.25) {$\yng(1,1)$};
\node at (-4.8,3.15) {$\yng(1)$};
\node at (-5.8,3.2) {$\scalemath{1}{\vec{0}}$}; 
\draw (0,-3)--(6,-3);
\draw (-6,3)--(0,3);

\node at (-3.5,-1) {$\sigma=123$};
\node at (3,2) {$\pi=456123$};
\end{tikzpicture}
}

\scalemath{.7}{
\begin{tikzpicture}
\draw[step=1cm] (-5,2) grid (0,3);
\draw[step=1cm] (-4,1) grid (1,2);
\draw[step=1cm] (-3,0) grid (2,1);
\draw[step=1cm] (-2,-1) grid (3,0);
\draw[step=1cm] (-1,-2) grid (4,-1);
\draw[step=1cm] (0,-3) grid (5,-2);
\draw (0,-3)--(6,-3);
\draw (-6,3)--(0,3);

\draw[blue,ultra thick] (0,0)--(0,-3)--(3,-3)--(3,0)--cycle;
\draw[red,ultra thick] (0,0)--(0,3)--(-3,3)--(-3,0)--cycle;
\Yboxdim{2mm}
\node at (-5.8,3.2) {$\scalemath{1}{\vec{0}}$}; 
\node at (-4.8,3.2) {$\yng(1)$};
\node at (-3.7,3.2) {$\yng(2)$};
\node at (-2.7,3.25) {$\yng(2,1)$};
\node at (-1.7,3.2) {$\yng(2)$};
\node at (0-.7,3.2) {$\yng(1)$};
\node at (0.2,3.2) {$\scalemath{1}{\vec{0}}$};
\node at (-.5,2.5) {\scalemath{1}{1}};
\node at (-4.8,2.2) {$\scalemath{1}{\vec{0}}$}; 
\node at (-3.8,2.2) {$\yng(1)$};
\node at (-2.75,2.25) {$\yng(1,1)$};
\node at (-1.7,2.2) {$\yng(1)$};
\node at (-.7,2.2) {$\scalemath{1}{\vec{0}}$};
\node at (0.15,2.35) {$\young(:\empty,:\empty,~)$};
\node at (1.2,2.2) {$\scalemath{1}{\vec{0}}$};
\node at (-2.5,1.5) {\scalemath{1}{1}};
\node at (-3.8,1.2) {$\scalemath{1}{\vec{0}}$};  
\node at (-2.85,1.15) {$\yng(1)$};
\node at (-1.75,1.325) {$\young(:~,:\empty,~)$};
\node at (-.8,1.35) {$\young(:\empty,:\empty,~)$};
\node at (0.3,1.35) {$\young(::\empty,::\empty,~~)$};
\node at (1.2,1.35) {$\young(:\empty,:\empty,~)$};
\node at (2.2,1.2) {$\scalemath{1}{\vec{0}}$};
\node at (-1.5,.5) {\scalemath{1}{1}};
\node at (-2.8,.2) {$\scalemath{1}{\vec{0}}$};  
\node at (-1.8,.35) {$\young(:\empty,:\empty,\empty)$};
\node at (-.75,.35) {$\young(:\empty,~,~)$};
\node at (0.3,.35) {$\young(::\empty,:~,~~)$};
\node at (1.3,.35) {$\young(:\empty,~,~)$};
\node at (2.15,.35) {$\young(:\empty,:\empty,~)$};
\node at (3.2,.175) {$\scalemath{1}{\vec{0}}$};
\node at (1.5,-.5) {\scalemath{1}{1}};
\node at (-1.8,-.8) {$\scalemath{1}{\vec{0}}$};  
\node at (-.85,-.68) {$\young(:\empty,:\empty,\empty)$};
\node at (0.25,-.68) {$\young(::\empty,::\empty,~~)$};
\node at (1.2,-.68) {$\young(:\empty,:\empty,~)$};
\node at (2.25,-.68) {$\young(:~,:\empty,~)$};
\node at (3.2,-.825) {\yng(1)};
\node at (4.2,-.8) {$\scalemath{1}{\vec{0}}$};
\node at (2.5,-1.5) {\scalemath{1}{1}};
\node at (-.8,-1.8) {$\scalemath{1}{\vec{0}}$};  
\node at (0.2,-1.65) {$\young(:\empty,:\empty,~)$};
\node at (1.3,-1.8) {$\scalemath{1}{\vec{0}}$};
\node at (2.2,-1.8) {$\yng(1)$};
\node at (3.2,-1.75) {\yng(1,1)};
\node at (4.2,-1.825) {\yng(1)};
\node at (5.2,-1.8) {$\scalemath{1}{\vec{0}}$};
\node at (.5,-2.5) {\scalemath{1}{1}};
\node at (6.2,-2.8) {$\scalemath{1}{\vec{0}}$};
\node at (5.2,-2.8) {\yng(1)};
\node at (4.3,-2.8) {\yng(2)};
\node at (3.3,-2.775) {\yng(2,1)};
\node at (2.3,-2.8) {\yng(2)};
\node at (1.2,-2.8) {\yng(1)};
\node at (.2,-2.8) {$\scalemath{1}{\vec{0}}$};
\draw (0,-3)--(6,-3);
\draw (-6,3)--(0,3);

\node at (-3.5,-1) {$\sigma=231$};
\node at (3,2) {$\pi=645231$};
\end{tikzpicture}}
\end{minipage}
\begin{minipage}{.49\textwidth}
\scalemath{.7}{
\begin{tikzpicture}
\draw[step=1cm] (-5,2) grid (0,3);
\draw[step=1cm] (-4,1) grid (1,2);
\draw[step=1cm] (-3,0) grid (2,1);
\draw[step=1cm] (-2,-1) grid (3,0);
\draw[step=1cm] (-1,-2) grid (4,-1);
\draw[step=1cm] (0,-3) grid (5,-2);
\draw (0,-3)--(6,-3);
\draw (-6,3)--(0,3);

\draw[blue,ultra thick] (0,0)--(0,-3)--(3,-3)--(3,0)--cycle;
\draw[red,ultra thick] (0,0)--(0,3)--(-3,3)--(-3,0)--cycle;
\Yboxdim{2mm}
\node at (-5.8,3.2) {$\scalemath{1}{\vec{0}}$}; 
\node at (-4.8,3.2) {$\yng(1)$};
\node at (-3.7,3.225) {$\yng(1,1)$};
\node at (-2.7,3.225) {$\yng(2,1)$};
\node at (-1.7,3.225) {$\yng(1,1)$};
\node at (0-.7,3.2) {$\yng(1)$};
\node at (0.2,3.2) {$\scalemath{1}{\vec{0}}$};
\node at (-1.5,2.5) {\scalemath{1}{1}};
\node at (-4.8,2.2) {$\scalemath{1}{\vec{0}}$}; 
\node at (-3.85,2.15) {$\yng(1)$};
\node at (-2.75,2.15) {$\yng(2)$};
\node at (-1.85,2.15) {$\yng(1)$};
\node at (-.75,2.35) {$\young(:~,:\empty,~)$};
\node at (0.2,2.325) {$\young(:\empty,:\empty,~)$};
\node at (1.2,2.2) {$\scalemath{1}{\vec{0}}$};
\node at (-.4,1.5) {\scalemath{1}{1}};
\node at (-3.8,1.2) {$\scalemath{1}{\vec{0}}$};  
\node at (-2.8,1.15) {$\yng(1)$};
\node at (-1.7,1.2) {$\scalemath{1}{\vec{0}}$};
\node at (-.85,1.35) {$\young(:\empty,:\empty,~)$};
\node at (0.3,1.35) {$\young(:\empty,~,~)$};
\node at (1.2,1.35) {$\young(:\empty,:\empty,~)$};
\node at (2.2,1.2) {$\scalemath{1}{\vec{0}}$};
\node at (-2.5,.5) {\scalemath{1}{1}};
\node at (-2.8,.2) {$\scalemath{1}{\vec{0}}$};  
\node at (-1.85,.35) {$\young(:\empty,:\empty,\empty)$};
\node at (-.79,.35) {$\young(::\empty,::\empty,~~)$};
\node at (0.25,.35) {$\young(::\empty,:~,~~)$};
\node at (1.25,.35) {$\young(::\empty,::\empty,~~)$};
\node at (2.2,.35) {$\young(:\empty,:\empty,~)$};
\node at (3.2,.2) {$\scalemath{1}{\vec{0}}$};
\node at (2.5,-.5) {\scalemath{1}{1}};
\node at (-1.8,-.8) {$\scalemath{1}{\vec{0}}$};  
\node at (-.8,-.675) {$\young(:\empty,:\empty,\empty)$};
\node at (0.2,-.675) {$\young(:\empty,~,~)$};
\node at (1.2,-.675) {$\young(:\empty,:\empty,~)$};
\node at (2.3,-.8) {$\scalemath{1}{\vec{0}}$};
\node at (3.15,-.85) {\yng(1)};
\node at (4.2,-.8) {$\scalemath{1}{\vec{0}}$};
\node at (.5,-1.5) {\scalemath{1}{1}};
\node at (-.8,-1.8) {$\scalemath{1}{\vec{0}}$};  
\node at (0.2,-1.65) {$\young(:\empty,:\empty,~)$};
\node at (1.275,-1.65) {$\young(:~,:\empty,~)$};
\node at (2.2,-1.8) {$\yng(1)$};
\node at (3.3,-1.8) {\yng(2)};
\node at (4.2,-1.8) {\yng(1)};
\node at (5.2,-1.8) {$\scalemath{1}{\vec{0}}$};
\node at (1.5,-2.5) {\scalemath{1}{1}};
\node at (6.2,-2.8) {$\scalemath{1}{\vec{0}}$};
\node at (5.2,-2.8) {\yng(1)};
\node at (4.3,-2.75) {\yng(1,1)};
\node at (3.3,-2.75) {\yng(2,1)};
\node at (2.3,-2.75) {\yng(1,1)};
\node at (1.2,-2.8) {\yng(1)};
\node at (.2,-2.8) {$\scalemath{1}{\vec{0}}$};
\draw (0,-3)--(6,-3);
\draw (-6,3)--(0,3);

\node at (-3.5,-1) {$\sigma=312$};
\node at (3,2) {$\pi=564312$};
\end{tikzpicture}
}

\scalemath{.7}{
\begin{tikzpicture}
\draw[step=1cm] (-5,2) grid (0,3);
\draw[step=1cm] (-4,1) grid (1,2);
\draw[step=1cm] (-3,0) grid (2,1);
\draw[step=1cm] (-2,-1) grid (3,0);
\draw[step=1cm] (-1,-2) grid (4,-1);
\draw[step=1cm] (0,-3) grid (5,-2);
\draw (0,-3)--(6,-3);
\draw (-6,3)--(0,3);

\draw[blue,ultra thick] (0,0)--(0,-3)--(3,-3)--(3,0)--cycle;
\draw[red,ultra thick] (0,0)--(0,3)--(-3,3)--(-3,0)--cycle;
\Yboxdim{2mm}
\node at (-5.8,3.2) {$\scalemath{1}{\vec{0}}$}; 
\node at (-4.8,3.2) {$\yng(1)$};
\node at (-3.8,3.25) {$\yng(1,1)$};
\node at (-2.7,3.25) {$\yng(2,1)$};
\node at (-1.7,3.2) {$\yng(2)$};
\node at (0-.8,3.2) {$\yng(1)$};
\node at (0.2,3.2) {$\scalemath{1}{\vec{0}}$};
\node at (-2.5,2.5) {\scalemath{1}{1}};
\node at (-4.8,2.2) {$\scalemath{1}{\vec{0}}$}; 
\node at (-3.8,2.2) {$\yng(1)$};
\node at (-2.7,2.2) {$\yng(2)$};
\node at (-1.65,2.35) {$\young(:~~,:\empty,~)$};
\node at (-.75,2.35) {$\young(:~,:\empty,~)$};
\node at (0.15,2.35) {$\young(:\empty,:\empty,~)$};
\node at (1.2,2.2) {$\scalemath{1}{\vec{0}}$};
\node at (-.4,1.5) {\scalemath{1}{1}};
\node at (-3.8,1.2) {$\scalemath{1}{\vec{0}}$};  
\node at (-2.85,1.15) {$\yng(1)$};
\node at (-1.75,1.325) {$\young(:~,:\empty,~)$};
\node at (-.8,1.325) {$\young(:\empty,:\empty,~)$};
\node at (0.2,1.325) {$\young(:\empty,~,~)$};
\node at (1.2,1.325) {$\young(:\empty,:\empty,~)$};
\node at (2.2,1.2) {$\scalemath{1}{\vec{0}}$};
\node at (-1.5,.5) {\scalemath{1}{1}};
\node at (-2.8,.2) {$\scalemath{1}{\vec{0}}$};  
\node at (-1.8,.325) {$\young(:\empty,:\empty,\empty)$};
\node at (-.8,.325) {$\young(:\empty,~,~)$};
\node at (0.3,.325) {$\young(::\empty,:~,~~)$};
\node at (1.3,.325) {$\young(::\empty,::\empty,~~)$};
\node at (2.2,.325) {$\young(:\empty,:\empty,~)$};
\node at (3.2,.2) {$\scalemath{1}{\vec{0}}$};
\node at (.6,-.5) {\scalemath{1}{1}};
\node at (-1.8,-.8) {$\scalemath{1}{\vec{0}}$};  
\node at (-.8,-.675) {$\young(:\empty,:\empty,\empty)$};
\node at (0.275,-.675) {$\young(::\empty,::\empty,~~)$};
\node at (1.325,-.675) {$\young(::~,::\empty,~~)$};
\node at (2.275,-.675) {$\young(:~,:\empty,~)$};
\node at (3.2,-.675) {\yng(1)};
\node at (4.2,-.8) {$\scalemath{1}{\vec{0}}$};
\node at (2.5,-1.5) {\scalemath{1}{1}};
\node at (-.8,-1.8) {$\scalemath{1}{\vec{0}}$};  
\node at (0.175,-1.625) {$\young(:\empty,:\empty,~)$};
\node at (1.275,-1.675) {$\young(:~,:\empty,~)$};
\node at (2.2,-1.825) {$\yng(1)$};
\node at (3.2,-1.775) {\yng(1,1)};
\node at (4.2,-1.825) {\yng(1)};
\node at (5.2,-1.8) {$\scalemath{1}{\vec{0}}$};
\node at (1.5,-2.5) {\scalemath{1}{1}};
\node at (.2,-2.8) {$\scalemath{1}{\vec{0}}$};
\node at (1.2,-2.8) {\yng(1)};
\node at (2.3,-2.775) {\yng(1,1)};
\node at (3.3,-2.775) {\yng(2,1)};
\node at (4.3,-2.8) {\yng(2)};
\node at (5.2,-2.8) {\yng(1)};
\node at (6.2,-2.8) {$\scalemath{1}{\vec{0}}$};
\draw (0,-3)--(6,-3);
\draw (-6,3)--(0,3);

\node at (-3.5,-1) {$\sigma=132$};
\node at (3,2) {$\pi=465132$};
\end{tikzpicture}
}

\scalemath{.7}{
\begin{tikzpicture}
\draw[step=1cm] (-5,2) grid (0,3);
\draw[step=1cm] (-4,1) grid (1,2);
\draw[step=1cm] (-3,0) grid (2,1);
\draw[step=1cm] (-2,-1) grid (3,0);
\draw[step=1cm] (-1,-2) grid (4,-1);
\draw[step=1cm] (0,-3) grid (5,-2);
\draw (0,-3)--(6,-3);
\draw (-6,3)--(0,3);

\draw[blue,ultra thick] (0,0)--(0,-3)--(3,-3)--(3,0)--cycle;
\draw[red,ultra thick] (0,0)--(0,3)--(-3,3)--(-3,0)--cycle;
\Yboxdim{2mm}
\node at (-5.8,3.2) {$\scalemath{1}{\vec{0}}$}; 
\node at (-4.8,3.2) {$\yng(1)$};
\node at (-3.7,3.2) {$\yng(2)$};
\node at (-2.7,3.25) {$\yng(2,1)$};
\node at (-1.7,3.25) {$\yng(1,1)$};
\node at (0-.7,3.2) {$\yng(1)$};
\node at (0.2,3.2) {$\scalemath{1}{\vec{0}}$};
\node at (-1.5,2.5) {\scalemath{1}{1}};
\node at (-4.8,2.2) {$\scalemath{1}{\vec{0}}$}; 
\node at (-3.8,2.2) {$\yng(1)$};
\node at (-2.8,2.25) {$\yng(1,1)$};
\node at (-1.8,2.2) {$\yng(1)$};
\node at (-.725,2.35) {$\young(:~,:\empty,~)$};
\node at (0.175,2.35) {$\young(:\empty,:\empty,~)$};
\node at (1.2,2.2) {$\scalemath{1}{\vec{0}}$};
\node at (-2.4,1.5) {\scalemath{1}{1}};
\node at (-3.8,1.2) {$\scalemath{1}{\vec{0}}$};  
\node at (-2.8,1.2) {$\yng(1)$};
\node at (-1.725,1.35) {$\young(:~,:\empty,~)$};
\node at (-.675,1.35) {$\young(::~,::\empty,~~)$};
\node at (0.25,1.35) {$\young(::\empty,::\empty,~~)$};
\node at (1.2,1.35) {$\young(:\empty,:\empty,~)$};
\node at (2.2,1.2) {$\scalemath{1}{\vec{0}}$};
\node at (-.4,.5) {\scalemath{1}{1}};
\node at (-2.8,.2) {$\scalemath{1}{\vec{0}}$};  
\node at (-1.8,.35) {$\young(:\empty,:\empty,\empty)$};
\node at (-.75,.35) {$\young(::\empty,::\empty,~~)$};
\node at (0.3,.35) {$\young(::\empty,:~,~~)$};
\node at (1.2,.35) {$\young(:\empty,~,~)$};
\node at (2.2,.35) {$\young(:\empty,:\empty,~)$};
\node at (3.2,.2) {$\scalemath{1}{\vec{0}}$};
\node at (1.5,-.5) {\scalemath{1}{1}};
\node at (-1.8,-.8) {$\scalemath{1}{\vec{0}}$};  
\node at (-.8,-.65) {$\young(:\empty,:\empty,\empty)$};
\node at (0.2,-.65) {$\young(:\empty,~,~)$};
\node at (1.2,-.65) {$\young(:\empty,:\empty,~)$};
\node at (2.3,-.65) {$\young(:~,:\empty,~)$};
\node at (3.2,-.8) {\yng(1)};
\node at (4.2,-.8) {$\scalemath{1}{\vec{0}}$};
\node at (.5,-1.5) {\scalemath{1}{1}};
\node at (-.8,-1.8) {$\scalemath{1}{\vec{0}}$};  
\node at (0.2,-1.65) {$\young(:\empty,:\empty,~)$};
\node at (1.275,-1.65) {$\young(:~,:\empty,~)$};
\node at (2.325,-1.65) {$\young(:~~,:\empty,~)$};
\node at (3.275,-1.825) {\yng(2)};
\node at (4.175,-1.825) {\yng(1)};
\node at (5.2,-1.8) {$\scalemath{1}{\vec{0}}$};
\node at (2.5,-2.5) {\scalemath{1}{1}};
\node at (.2,-2.8) {$\scalemath{1}{\vec{0}}$};
\node at (1.2,-2.8) {\yng(1)};
\node at (2.3,-2.8) {\yng(2)};
\node at (3.3,-2.775) {\yng(2,1)};
\node at (4.3,-2.775) {\yng(1,1)};
\node at (5.2,-2.8) {\yng(1)};
\node at (6.2,-2.8) {$\scalemath{1}{\vec{0}}$};
\draw (0,-3)--(6,-3);
\draw (-6,3)--(0,3);

\node at (-3.5,-1) {$\sigma=213$};
\node at (3,2) {$\pi=546213$};
\end{tikzpicture}
}
\end{minipage}
\caption{The six affine growth diagrams of the Robinson--Schensted correspondence for $k=3$. \label{fig:RS}}
\end{figure}

\subsection{Robinson--Schensted--Knuth Generalizations \label{sec:RSK}}

We now drop the assumption that each $\lambda^i$ is equal to $\omega_1$ or $\omega_1^*$ and prove the analogue of Theorem \ref{thm:bijection} in this general setting. To this end, we briefly review how to extend Fomin growth diagrams to natural number matrices by applying Knuth's relabelling scheme \cite{Knu}, as shown in \cite{Rob} (see also \cite{Ful}). For a matrix $M$ of naturals construct a permutation matrix $M'$ as follows. Let $r_i$ (resp. $c_i$) denote the sum of the entries in the $i$th row (resp. $i$th column) of $M$ and for each row (resp. column) of $M$ create $r_i$ rows (resp. $c_i$ columns) in $M'$, so that $M'$ is an $n\times n$ matrix where $n$ is equal to the sum of all the entries of $M$. Each entry $M_{i,j}$ is replaced by $M_{i,j}$ many $1$'s in a diagonal such that for a pair of indices $(i,j)\not=(l,k)$, if $i\leq l$ and $j\leq k$, then all of the $1$ entries of $M'$ corresponding to $M_{i,j}$ occur strictly northwest of the $1$'s corresponding to $M_{l,k}$. 

Let $M'(i,j)$ be the result of applying this relabelling scheme to $M(i,j)$, the northwest-justified submatrix of $M$. Define a \emph{generalized Fomin diagram} to be the set of partitions $\alpha_{i,j}$, one for each vertex $(i,j)$ of the array containing the matrix $M$, where $\alpha_{i,j}=\lambda(M'(i,j)_{\searrow})$. Equivalently, the partitions $\alpha_{i,j}$ can be calculated by applying Fomin's growth rules to the partial permutation matrix $M'(i,j)$. The following lemma, stated without proof, shows that two adjacent partitions in the diagram differ by a vertical strip.

\begin{Lem}[{\cite[Thm.~4.1.4]{Rob}}]\label{lem:consecutive}
In a Fomin growth diagram, let $\mu^i,\mu^{i+1},\ldots, \mu^j$ be consecutive shapes growing along the bottom of the $n\times n$ grid and let the partial permutation in the $j-i$ columns above them be given by $w_i,w_{i+1},\ldots,w_j$. If $w_i<w_{i+1}<\ldots<w_j$, then $\mu^j$ and $\mu^i$ differ by a vertical strip and the boxes are added sequentially in $\mu^i,\mu^{i+1},\ldots,\mu^j$ from top to bottom. 
\end{Lem}

These generalized Fomin growth diagrams realize the full Robinson--Schensted--Knuth correspondence, a content preserving bijection between $\mathbb N$-matrices of size $n$ and pairs of same-shape semistandard Young tableaux on $n$ boxes. We will show that they appear as submatrices of affine growth diagrams with first row an oscillating tableau consisting of $k$ increasing steps followed by $n-k$ decreasing steps. More generally, as mentioned in \S\ref{sec:combinatorial background}, there is a bijection between size $n$ semistandard oscillating tableaux and symmetric $\mathbb N$-matrices such that $r_i=0$ or $c_i=0$ for all $i$. 

Let $\vec{a}=(a_1,\ldots,a_n)$ be a sequence of nonzero integers. Let $\mathcal A_{\vec{a}, m}$ be the set of $GL_m$ affine growth diagrams such that the first row is an oscillating tableau, and the type $\vec\lambda$ is defined by $\lambda^i=\omega_{a_i}$ if $a_i$ is positive and $\lambda^i=\omega_{\lvert a_i\rvert}^*$ if $a_i$ is negative. Note that the sum of the positive entries of $\vec{a}$ has to equal the sum of the negative entries for $\mathcal A_{\vec{a},m}$ to be nonempty. 

As before for an $n\times n$ matrix $M$, let $r_i$ denote the partial row sum $\sum_{j=i}^n M_{i,j}$ and $c_i$ the partial column sum $\sum_{j=1}^i M_{j,i}$. Let $\mathcal N_{\vec{a}}$ denote the set of $n\times n$ natural-number symmetric matrices such that $r_i=a_i$ and $c_i=0$ if $a_i$ is positive, and $r_i=0, c_i=\lvert a_i\rvert$ if $a_i$ is negative. Note that these conditions force zeros in the main diagonal. The matrices in $\mathcal N_{\vec{a}}$ can be thought of as the natural-number analogs of fixed-point-free involutions.

\begin{Def}
Let $\Phi_{\vec{a},m}:\mathcal A_{\vec{a},m}\rightarrow \mathcal N_{\vec{a}}$ denote the map producing the entries $n_{i,j}$, as previously defined in \S\ref{sec:filling}. Define an inverse map $\Psi_{\vec{a},m}:\mathcal N_{\vec{a}}\rightarrow \mathcal A_{\vec{a},m}$ for $m\geq \sum_{a_i>0}a_i$ as follows. Given a matrix $N$ in $\mathcal N_{\vec{a}}$ write the entries into a staircase diagram as in the previous section and extend periodically. Define $\gamma_{i,j}^+$ (resp. $\gamma_{i,j}^-$) by applying the relabelling scheme to $M(i,j)^+$ (resp. $M(i,j)^-$) to get a partial permutation matrix and taking the associated partition given by Greene's theorem. Set $\gamma_{i,j}=\gamma_{i,j}^++\gamma_{i,j}^-$.
\end{Def}

\begin{Thm}\label{thm:inverse map}
For fixed $\vec{a}$ and $m\geq \sum_{a_i>0}a_i$, we have $\Phi_{\vec{a},m}\circ\Psi_{\vec{a},m}(N)=N$. These maps are bijections.
\end{Thm}

\begin{proof}
The proof is an application of the relabeling scheme, followed by an application of Theorem \ref{thm:bijection}, and finally contracting the resulting affine growth diagram.  

Consider an empty staircase diagram of size $n$ and for $1\leq i\leq n$ and $i+1\leq j\leq i+n$ let $n_{i,j}=N_{\overline{i},\overline{j}}$ be the entries given by $N$ where $\overline{j}=j \mod n$. Apply the expansion scheme to get a larger staircase diagram, so that a row (resp. column) with sum $r$ (resp. $c$) in the old diagram is replaced by $r$ (resp. $c$) individual rows (resp. columns) in the new diagram. The resulting new diagram has $0,1$ entries such that each $n_{i,j}$ entry of the old diagram becomes a diagonal of $n_{i,j}$ many $1$'s. The $1$'s corresponding to entry $n_{i,j}$ appear strictly northwest of the $1$'s corresponding to an entry $n_{l,k}$ for $i\leq l$ and $j\leq k$. This process is reversible by coalescing consecutive rows and columns according to the content $\vec a$.

Applying Theorem \ref{thm:bijection} gives an affine growth diagram $\Gamma$ where each $\lambda^i$ is $\omega_1$ or $\omega_1^*$. Let $\gamma_{i,j}$ and $\gamma_{i,j+l}$ be weights labelling vertices of $\Gamma$ that correspond to adjacent vertices of the original width-$n$ staircase diagram. By Lemma \ref{lem:consecutive} these weights differ by a vertical strip. Construct a diagram $\Gamma'$ by forgetting all vertex labels in $\Gamma$ other than those corresponding to the original vertices and contracting. One checks that the resulting diagram $\Gamma'$ satisfies the local condition and is of type $\vec\lambda$ given by $\vec{a}$. We leave the details to the reader. This process is also reversible, since starting with a diagram from $\mathcal A_{\vec{a},m}$, if $\gamma_{i,j+1}-\gamma_{i,j}$ is a positive strip of length $l$, then this single step can be replaced by $l$ steps where $1$'s are added sequentially to $\gamma_{i,j}$ from top to bottom to the positions given by the positive strip. Similar reasoning applies if $\gamma_{i,j+1}-\gamma_{i,j}$ is a negative strip.
\end{proof}

For an oscillating tableau $\vec\mu$ let the \emph{content} be the sequence $\vec{a}=(a_1,\ldots,a_n)$ where $a_i$ is the signed size of the vertical strip $\mu^{i+1}-\mu^{i}$.

\begin{Cor}
There is a content preserving bijection between $\bigcup_{\vec{a}}\mathcal N_{\vec{a}}$ and length $n$ semistandard oscillating tableaux.
\end{Cor}

Finally, consider the content $\vec{a}$ where the first $k$ entries are positive and the last $n-k$ entries are negative. According to the construction $\Psi_{\vec{a},m}$ the first line of an affine growth diagram in $\mathcal A_{\vec{a},m}$ is an oscillating tableau that increases for $k$ steps and then decreases for $n-k$ steps, and can therefore be viewed as a pair of same-shape (row-strict) semistandard Young tableaux. 

\begin{Thm}\label{thm:generalized Fomin}
Suppose that the first $k$ integers of $\vec a=(a_1,\ldots,a_n)$ are positive and the last $n-k$ are negative. Suppose $m\geq \sum_{a_i>0}a_i$ and let $\vec\lambda$ be given according to $\vec a$. For an affine growth diagram in $\mathcal A_{\vec{a},m}$ the partitions $\gamma_{i,j}^+$ (resp. $\gamma_{i,j}^-$) for $1\leq i\leq k+1$ and $k+1\leq j\leq n+1$ form a generalized Fomin--Roby growth diagram growing from the southeast to the northwest (resp. from the northwest to the southeast). Likewise the $\gamma_{i,j}^+$ (resp. $\gamma_{i,j}^-$) for $k+1\leq i\leq n+1$, $n+1\leq j\leq n+1+k$ form a generalized Fomin--Roby growth diagram to the southeast (resp. northwest).
\end{Thm}

\section{Proof of Proposition \ref{prop:hive excavation}\label{main proof}}
This final section is devoted to the proof of Proposition \ref{prop:hive excavation}. For ease of notation let $\lambda=\gamma_{2,3}$, $\mu=\gamma_{1,4}$, $\nu=\gamma_{1,3}$, and $\rho=\gamma_{2,4}$.  Fix the orientation of the $4$-hive as in Figures \ref{fig:excavate} and \ref{fig:4hive example}, so that it is balanced on its bottom edge labelled by $\rho$ and is viewed from above. The northern face of the two closer faces will be called the top face A and the southern the top face B. The two farther faces will be called the left bottom face and the right bottom face. 

Recall that the successive differences of the hive values along the edges give dominant weights. The bottom edge is labelled by $\rho$ when read top to bottom. We will read the successive differences along all other edges from left to right given the fixed orientation of the hive. Let N, E, S, W denote the compass directions. Suppose that the NW, NE, and S edges of face A are labelled respectively by $\omega_r, \lambda,$ and $\nu$. Let the SW, SE, and N edges of face B be labelled respectively by $\mu, \omega_s,$ and $\nu$. We will verify that the local condition holds for this setup. There are three other cases, involving dual fundamental weights in place of $\omega_r$, $\omega_s$, but these are similar.

Consider the face-A $3$-hive with boundary weights $\omega_r, \lambda$, and $\nu$. The NW edge has differences $\omega_r$ when read from SW to NE, so the integer labels along this edge increase by one for $r$ steps and are constant thereafter. A lemma in \cite{KTW} gives an easy description of the differences along the strip one step in from this NW external edge. The proof is not difficult and is omitted.
\begin{Lem}[{\cite[Lem.~2]{KTW}}]\label{lem:strips}
If the NW external edge of a $3$-hive has differences $\omega_r$, then the SW-to-NE-oriented strip one step in from the NW edge has differences $\omega_r$ or $\omega_{r-1}$. Furthermore, the first case occurs if and only if $\nu_1=\lambda_1$ and the second case if and only if $\nu_1=\lambda_1+1$.
\end{Lem}

Let the second strip of face A be $\omega_{r'}$ where $r'$ must be $r$ or $r-1$. Since this strip can be viewed as the NW external edge of a smaller $3$-hive contained in face A, applying the lemma again implies that the third strip must have differences $\omega_{r'}$ or $\omega_{r'-1}$. This pattern continues so that there are $r$ decrements from the first strip to the $(n+1)$st strip, which is just a single lattice point. A decrement occurs between strip $k$ and $k+1$ if and only if $\nu_k=\lambda_k+1$. Figure \ref{top break paths} shows the top of a $4$-hive for $n=5$, $k=3$, $\lambda=(3,3,1,1,0)$, and $\nu=(4,3,2,2,0)$. The blue path indicates the lattice points after which the strips are constant. We call this the \emph{break path}. The break path determines the difference $\nu - \lambda$.

Likewise, face B of the $4$-hive has $\omega_s$ along its SE edge with $\omega_s$ being read from the SW to the NE. The strip one step in from this boundary edge is either $\omega_s$ or $\omega_{s-1}$, and so forth with $s$ decrements between the first strip and the $n+1$st strip. The break path of face B connects the SE external edge to the western lattice point. In order for the numbering of the strips to coincide with that of face A, we can view the break path of face B as starting at the western lattice point, thought of as labelled by $\omega_0$, so that there are $s$ increments before reaching $\omega_s$ on the SE edge. As with face A, there is an increase between strip $k$ and $k+1$ if and only if $\nu_k=\mu_k+1$, and no increase if and only if $\nu_k=\mu_k$. Thus, $\nu-\mu$ is determined by the face B break path. Figure \ref{top break paths} shows the top view of a $4$-hive with face B for $s=2$, $\mu=(4,2,2,1,0)$. 

\begin{figure}[t!]
\centering
\minipage[b]{0.5\textwidth}
\centering
\scalemath{.4}{
\begin{tikzpicture}
  \begin{ternaryaxis}[ylabel=$\scalemath{2}{\omega_3}$,xlabel=$\scalemath{2}{\lambda}$,zlabel=$\scalemath{2}{\nu}$,
    clip=false,xticklabels={},yticklabels={},zticklabels={},xmin=0,
    ymin=0,
    zmin=0,
    xmax=10,
    ymax=10,
    zmax=10,
    xtick={2,4,6,8},
    ytick={2,4,6,8},
    ztick={2,4,6,8},
    major tick length=0,
    %axis .style={opacity=.1}
    %%opacity=0.8,
    %%clip=false
]    
  	\node at (axis cs:0,24,0) {$\scalemath{1.75}{0}$};
	\node at (axis cs:2,22,0) {$\scalemath{1.75}{1}$};
	\node at (axis cs:4,20,0) {$\scalemath{1.75}{2}$};
	\node at (axis cs:6,18,0) {$\scalemath{1.75}{3}$};
	\node at (axis cs:8,16,0) {$\scalemath{1.75}{3}$};
	\node at (axis cs:10,14,0) {$\scalemath{1.75}{3}$};
	
	\node at (axis cs:0,22,2) {$\scalemath{1.75}{4}$};
	\node at (axis cs:2,20,2) {$\scalemath{1.75}{5}$};
	\node at (axis cs:4,18,2) {$\scalemath{1.75}{6}$};
	\node at (axis cs:6,16,2) {$\scalemath{1.75}{6}$};
	\node at (axis cs:8,14,2) {$\scalemath{1.75}{6}$};
	
	\node at (axis cs:0,20,4) {$\scalemath{1.75}{7}$};
	\node at (axis cs:2,18,4) {$\scalemath{1.75}{8}$};
	\node at (axis cs:4,16,4) {$\scalemath{1.75}{9}$};
	\node at (axis cs:6,14,4) {$\scalemath{1.75}{9}$};
	
	\node at (axis cs:0,18,6) {$\scalemath{1.75}{9}$};
	\node at (axis cs:2,16,6) {$\scalemath{1.75}{10}$};
	\node at (axis cs:4,14,6) {$\scalemath{1.75}{10}$};
	
	\node at (axis cs:0,16,8) {$\scalemath{1.75}{11}$};
	\node at (axis cs:2,14,8) {$\scalemath{1.75}{11}$};
	
	\node at (axis cs:0,14,10) {$\scalemath{1.75}{11}$};
	
	\addplot3[color=blue,ultra thick] coordinates {(6,18,0)(4,18,2)(4,16,4)(2,16,6)(0,16,8)(0,14,10)};

  \end{ternaryaxis}
\end{tikzpicture}}

\vspace{-5.2mm}

\scalemath{.4}{
\begin{tikzpicture}
  \begin{ternaryaxis}[ylabel=$\scalemath{2}{\mu}$,xlabel=$\scalemath{2}{\omega_2}$,xlabel style={xshift=0.5cm},
    clip=false,xticklabels={},yticklabels={},zticklabels={},xmin=0,
    ymin=0,
    zmin=0,
    xmax=10,
    ymax=10,
    zmax=10,
    yscale=-1,
    xtick={2,4,6,8},
    ytick={2,4,6,8},
    ztick={2,4,6,8},
    major tick length=0,
    %axis .style={opacity=.1}
    %%opacity=0.8,
    %%clip=false
]    
  	\node at (axis cs:0,24,0) {$\scalemath{1.75}{0}$};
	\node at (axis cs:2,22,0) {$\scalemath{1.75}{4}$};
	\node at (axis cs:4,20,0) {$\scalemath{1.75}{6}$};
	\node at (axis cs:6,18,0) {$\scalemath{1.75}{8}$};
	\node at (axis cs:8,16,0) {$\scalemath{1.75}{9}$};
	\node at (axis cs:10,14,0) {$\scalemath{1.75}{9}$};
	
	\node at (axis cs:0,22,2) {$\scalemath{1.75}{4}$};
	\node at (axis cs:2,20,2) {$\scalemath{1.75}{7}$};
	\node at (axis cs:4,18,2) {$\scalemath{1.75}{9}$};
	\node at (axis cs:6,16,2) {$\scalemath{1.75}{10}$};
	\node at (axis cs:8,14,2) {$\scalemath{1.75}{10}$};
	
	\node at (axis cs:0,20,4) {$\scalemath{1.75}{7}$};
	\node at (axis cs:2,18,4) {$\scalemath{1.75}{9}$};
	\node at (axis cs:4,16,4) {$\scalemath{1.75}{11}$};
	\node at (axis cs:6,14,4) {$\scalemath{1.75}{11}$};
	
	\node at (axis cs:0,18,6) {$\scalemath{1.75}{9}$};
	\node at (axis cs:2,16,6) {$\scalemath{1.75}{11}$};
	\node at (axis cs:4,14,6) {$\scalemath{1.75}{11}$};
	
	\node at (axis cs:0,16,8) {$\scalemath{1.75}{11}$};
	\node at (axis cs:2,14,8) {$\scalemath{1.75}{11}$};
	
	\node at (axis cs:0,14,10) {$\scalemath{1.75}{11}$};
	
	\addplot3[color=red,ultra thick] coordinates {(0,10,0)(2,8,0)(2,6,2)(4,4,2)(4,2,4)(6,0,4)};

  \end{ternaryaxis}
\end{tikzpicture}}
\subcaption{top faces \label{top break paths}}
\endminipage
\minipage[b]{0.5\textwidth}
\centering
\scalemath{.4}{
\begin{tikzpicture}
  \begin{ternaryaxis}[rotate=-30,ylabel=$\scalemath{2}{\omega_3}$,xlabel=$\scalemath{2}{\rho}$,xlabel style={xshift=-.4cm,yshift=-.2},zlabel=$\scalemath{2}{\mu}$,
    clip=false,xticklabels={},yticklabels={},zticklabels={},xmin=0,
    ymin=0,
    zmin=0,
    xmax=10,
    ymax=10,
    zmax=10,
    xtick={2,4,6,8},
    ytick={2,4,6,8},
    ztick={2,4,6,8},
    major tick length=0,
    %axis .style={opacity=.1}
    %%opacity=0.8,
    %%clip=false
]    
  	\node at (axis cs:0,24,0) {$\scalemath{1.75}{0}$};
	\node at (axis cs:2,22,0) {$\scalemath{1.75}{1}$};
	\node at (axis cs:4,20,0) {$\scalemath{1.75}{2}$};
	\node at (axis cs:6,18,0) {$\scalemath{1.75}{3}$};
	\node at (axis cs:8,16,0) {$\scalemath{1.75}{3}$};
	\node at (axis cs:10,14,0) {$\scalemath{1.75}{3}$};
	
	\node at (axis cs:0,22,2) {$\scalemath{1.75}{4}$};
	\node at (axis cs:2,20,2) {$\scalemath{1.75}{5}$};
	\node at (axis cs:4,18,2) {$\scalemath{1.75}{6}$};
	\node at (axis cs:6,16,2) {$\scalemath{1.75}{6}$};
	\node at (axis cs:8,14,2) {$\scalemath{1.75}{6}$};
	
	\node at (axis cs:0,20,4) {$\scalemath{1.75}{6}$};
	\node at (axis cs:2,18,4) {$\scalemath{1.75}{7}$};
	\node at (axis cs:4,16,4) {$\scalemath{1.75}{8}$};
	\node at (axis cs:6,14,4) {$\scalemath{1.75}{8}$};
	
	\node at (axis cs:0,18,6) {$\scalemath{1.75}{8}$};
	\node at (axis cs:2,16,6) {$\scalemath{1.75}{9}$};
	\node at (axis cs:4,14,6) {$\scalemath{1.75}{9}$};
	
	\node at (axis cs:0,16,8) {$\scalemath{1.75}{9}$};
	\node at (axis cs:2,14,8) {$\scalemath{1.75}{9}$};
	
	\node at (axis cs:0,14,10) {$\scalemath{1.75}{9}$};
	
	\addplot3[color=blue,ultra thick] coordinates {(6,18,0)(4,18,2)(4,16,4)(2,16,6)(0,16,8)(0,14,10)};

  \end{ternaryaxis}
\end{tikzpicture}}
\hspace{-6.3mm}
\scalemath{.4}{
\begin{tikzpicture}
  \begin{ternaryaxis}[rotate=-30,xlabel=$\scalemath{2}{\omega_2}$,xlabel style={xshift=.5cm,yshift=-.5cm},zlabel=$\scalemath{2}{\lambda}$,zlabel style={yshift=1cm,xshift=.5},
    clip=false,xticklabels={},yticklabels={},zticklabels={},xmin=0,
    ymin=0,
    zmin=0,
    xmax=10,
    ymax=10,
    zmax=10,
    xtick={2,4,6,8},
    ytick={2,4,6,8},
    ztick={2,4,6,8},
    yscale=-1,
    major tick length=0,
    %axis .style={opacity=.1}
    %%opacity=0.8,
    %%clip=false
]    
  	\node at (axis cs:0,24,0) {$\scalemath{1.75}{3}$};
	\node at (axis cs:2,22,0) {$\scalemath{1.75}{6}$};
	\node at (axis cs:4,20,0) {$\scalemath{1.75}{8}$};
	\node at (axis cs:6,18,0) {$\scalemath{1.75}{9}$};
	\node at (axis cs:8,16,0) {$\scalemath{1.75}{9}$};
	\node at (axis cs:10,14,0) {$\scalemath{1.75}{9}$};
	
	\node at (axis cs:0,22,2) {$\scalemath{1.75}{6}$};
	\node at (axis cs:2,20,2) {$\scalemath{1.75}{9}$};
	\node at (axis cs:4,18,2) {$\scalemath{1.75}{10}$};
	\node at (axis cs:6,16,2) {$\scalemath{1.75}{10}$};
	\node at (axis cs:8,14,2) {$\scalemath{1.75}{10}$};
	
	\node at (axis cs:0,20,4) {$\scalemath{1.75}{9}$};
	\node at (axis cs:2,18,4) {$\scalemath{1.75}{10}$};
	\node at (axis cs:4,16,4) {$\scalemath{1.75}{11}$};
	\node at (axis cs:6,14,4) {$\scalemath{1.75}{11}$};
	
	\node at (axis cs:0,18,6) {$\scalemath{1.75}{10}$};
	\node at (axis cs:2,16,6) {$\scalemath{1.75}{11}$};
	\node at (axis cs:4,14,6) {$\scalemath{1.75}{11}$};
	
	\node at (axis cs:0,16,8) {$\scalemath{1.75}{11}$};
	\node at (axis cs:2,14,8) {$\scalemath{1.75}{11}$};
	
	\node at (axis cs:0,14,10) {$\scalemath{1.75}{11}$};
	
	\addplot3[color=red,ultra thick] coordinates {(0,24,0)(2,22,0)(2,20,2)(4,18,2)(4,16,4)(6,14,4)};

  \end{ternaryaxis}
\end{tikzpicture}}
\subcaption{bottom faces\label{bottom break paths}}
\endminipage
\caption{The top and bottom of a $4$-hive with break paths indicated. (The top and bottom of the tetrahedron have been laid flat.) \label{fig:4hive example}}
\end{figure}

The left bottom face is labelled by $\omega_r, \mu,$ and $\rho$ along the NW, SW, and E edges respectively. As with the top faces, the NW external edge is labelled by $\omega_r$, and the successive strips are labelled by the same fundamental weight or one smaller giving a break path. The final strip is the southern lattice point. The decreases in $r$ together with $\mu$, determine $\rho$. More specifically, $\rho_k=\mu_k$ if there is no decrease between the $k$ and $k+1$ strip and $\rho_k=\mu_k-1$ if there is a decrease. 

For fixed $\mu$ the $i$ and $i+1$ strips of the left bottom face determine $\rho_i$. We will use the octahedron recurrence to excavate individual lattice points to reveal the decrease steps of the break path, and hence determine $\rho$. The lattice points of the second strip of the bottom left face lie directly beneath the lattice points of the second strip of face A. To reveal the second strip of the bottom face, and hence determine $\rho_1$, it suffices to ``shave off'' face A by a depth of $1$ by applications of the octahedron recurrence. This leaves a smaller tetrahedron, whose main horizontal edge visible from the top is one row down on face B from the original main horizontal edge. 

Figure $\ref{top shaving}$ demonstrates this first shave step for our previous example with the new tetrahedron indicated. At this point the first two strips of the bottom left face are known, so in particular $\rho_1=3$ can be seen either by directly reading off the difference $6-3$, or by noting that there is a decrease from $\omega_3$ to $\omega_2$ between the the first and second strips and that $\mu_1=4$.

\begin{figure}
\centering
\minipage{0.5\textwidth}
\centering
\scalemath{.4}{
\begin{tikzpicture}
  \begin{ternaryaxis}[
  separate axis lines,
    z axis line style= { draw opacity=.1 },
  ylabel=$\scalemath{2}{\omega_3}$,xlabel=$\scalemath{2}{\lambda}$,
    clip=false,xticklabels={},yticklabels={},zticklabels={},xmin=0,
    ymin=0,
    zmin=0,
    xmax=10,
    ymax=10,
    zmax=10,
    xtick={2,4,6,8},
    ytick={2,4,6,8},
    ztick={2,4,6,8},
    major tick length=0,
    %axis .style={opacity=.1}
    %%opacity=0.8,
    %%clip=false
]    
  	\node at (axis cs:0,24,0) {$\scalemath{1.75}{0}$};
	\node at (axis cs:2,22,0) {$\scalemath{1.75}{1}$};
	\node at (axis cs:4,20,0) {$\scalemath{1.75}{2}$};
	\node at (axis cs:6,18,0) {$\scalemath{1.75}{3}$};
	\node at (axis cs:8,16,0) {$\scalemath{1.75}{3}$};
	\node at (axis cs:10,14,0) {$\scalemath{1.75}{3}$};
	
	\node[color=gray] at (axis cs:0,22,2) {$\scalemath{1.75}{5}$};
	\node at (axis cs:2,20,2) {$\scalemath{1.75}{5}$};
	\node at (axis cs:4,18,2) {$\scalemath{1.75}{6}$};
	\node at (axis cs:6,16,2) {$\scalemath{1.75}{6}$};
	\node at (axis cs:8,14,2) {$\scalemath{1.75}{6}$};
	
	\node[color=gray] at (axis cs:0,20,4) {$\scalemath{1.75}{8}$};
	\node at (axis cs:2,18,4) {$\scalemath{1.75}{8}$};
	\node at (axis cs:4,16,4) {$\scalemath{1.75}{9}$};
	\node at (axis cs:6,14,4) {$\scalemath{1.75}{9}$};
	
	\node[color=gray] at (axis cs:0,18,6) {$\scalemath{1.75}{10}$};
	\node at (axis cs:2,16,6) {$\scalemath{1.75}{10}$};
	\node at (axis cs:4,14,6) {$\scalemath{1.75}{10}$};
	
	\node[color=gray] at (axis cs:0,16,8) {$\scalemath{1.75}{11}$};
	\node at (axis cs:2,14,8) {$\scalemath{1.75}{11}$};
	
	\node at (axis cs:0,14,10) {$\scalemath{1.75}{11}$};
	
	\addplot3[color=blue,ultra thick] coordinates {(6,18,0)(4,18,2)(4,16,4)(2,16,6)(0,16,8)(0,14,10)};
	\addplot3[color=black,thick] coordinates {(2,8,0)(2,0,8)};

  \end{ternaryaxis}
\end{tikzpicture}}

\vspace{-3.15mm}

\scalemath{.4}{
\begin{tikzpicture}
  \begin{ternaryaxis}[
  separate axis lines,
    z axis line style= { draw opacity=.1 },
  ylabel=$\scalemath{2}{\mu}$,xlabel=$\scalemath{2}{\omega_2}$,xlabel style={xshift=0.5cm},
    clip=false,xticklabels={},yticklabels={},zticklabels={},xmin=0,
    ymin=0,
    zmin=0,
    xmax=10,
    ymax=10,
    zmax=10,
    yscale=-1,
    xtick={2,4,6,8},
    ytick={2,4,6,8},
    ztick={2,4,6,8},
    major tick length=0,
    %axis .style={opacity=.1}
    %%opacity=0.8,
    %%clip=false
]    
  	\node at (axis cs:0,24,0) {$\scalemath{1.75}{0}$};
	\node at (axis cs:2,22,0) {$\scalemath{1.75}{4}$};
	\node at (axis cs:4,20,0) {$\scalemath{1.75}{6}$};
	\node at (axis cs:6,18,0) {$\scalemath{1.75}{8}$};
	\node at (axis cs:8,16,0) {$\scalemath{1.75}{9}$};
	\node at (axis cs:10,14,0) {$\scalemath{1.75}{9}$};
	
	\node at (axis cs:0,22,2) {$\scalemath{1.75}{\mathbf{5}}$};
	\node at (axis cs:2,20,2) {$\scalemath{1.75}{7}$};
	\node at (axis cs:4,18,2) {$\scalemath{1.75}{9}$};
	\node at (axis cs:6,16,2) {$\scalemath{1.75}{10}$};
	\node at (axis cs:8,14,2) {$\scalemath{1.75}{10}$};
	
	\node at (axis cs:0,20,4) {$\scalemath{1.75}{\mathbf{8}}$};
	\node at (axis cs:2,18,4) {$\scalemath{1.75}{9}$};
	\node at (axis cs:4,16,4) {$\scalemath{1.75}{11}$};
	\node at (axis cs:6,14,4) {$\scalemath{1.75}{11}$};
	
	\node at (axis cs:0,18,6) {$\scalemath{1.75}{\mathbf{10}}$};
	\node at (axis cs:2,16,6) {$\scalemath{1.75}{11}$};
	\node at (axis cs:4,14,6) {$\scalemath{1.75}{11}$};
	
	\node at (axis cs:0,16,8) {$\scalemath{1.75}{\mathbf{11}}$};
	\node at (axis cs:2,14,8) {$\scalemath{1.75}{11}$};
	
	\node at (axis cs:0,14,10) {$\scalemath{1.75}{11}$};
	
	\addplot3[color=red,ultra thick] coordinates {(0,10,0)(2,8,0)(2,6,2)(4,4,2)(4,2,4)(6,0,4)};
        \addplot3[color=black,thick] coordinates {(2,8,0)(2,0,8)};

  \end{ternaryaxis}
\end{tikzpicture}}
\subcaption{Main horizontal excavated, new values in bold\label{horizontal excavation}}
\endminipage
\minipage{0.5\textwidth}
\centering
\scalemath{.4}{
\begin{tikzpicture}
  \begin{ternaryaxis}[
    separate axis lines,
    z axis line style= { draw opacity=.1 },
    ylabel=$\scalemath{2}{\omega_3}$,xlabel=$\scalemath{2}{\lambda}$,
    clip=false,xticklabels={},yticklabels={},zticklabels={},xmin=0,
    ymin=0,
    zmin=0,
    xmax=10,
    ymax=10,
    zmax=10,
    xtick={2,4,6,8},
    ytick={2,4,6,8},
    ztick={2,4,6,8},
    major tick length=0,
    %axis .style={opacity=.1}
    %%opacity=0.8,
    %%clip=false
]    

	\fill[gray,opacity=.5] (axis cs:0,24,0)--(axis cs:10,14,0)--(axis cs:8,14,2)--(axis cs:0,22,2)--cycle;
	\fill[gray,opacity=.5] (axis cs:0,14,10)--(axis cs:8,14,2)--(axis cs:6,16,2)--(axis cs:0,16,8)--cycle;

  	\node[color=gray] at (axis cs:0,24,0) {$\scalemath{1.75}{0}$};
	\node[color=gray] at (axis cs:2,22,0) {$\scalemath{1.75}{1}$};
	\node[color=gray] at (axis cs:4,20,0) {$\scalemath{1.75}{2}$};
	\node[color=gray] at (axis cs:6,18,0) {$\scalemath{1.75}{3}$};
	\node[color=gray] at (axis cs:8,16,0) {$\scalemath{1.75}{3}$};
	\node[color=gray] at (axis cs:10,14,0) {$\scalemath{1.75}{3}$};
	
	\node at (axis cs:0,22,2) {$\scalemath{1.75}{5}$};
	\node at (axis cs:2,20,2) {$\scalemath{1.75}{6}$};
	\node at (axis cs:4,18,2) {$\scalemath{1.75}{6}$};
	\node at (axis cs:6,16,2) {$\scalemath{1.75}{6}$};
	\node[color=gray] at (axis cs:8,14,2) {$\scalemath{1.75}{6}$};
	
	\node at (axis cs:0,20,4) {$\scalemath{1.75}{8}$};
	\node at (axis cs:2,18,4) {$\scalemath{1.75}{9}$};
	\node at (axis cs:4,16,4) {$\scalemath{1.75}{9}$};
	\node[color=gray] at (axis cs:6,14,4) {$\scalemath{1.75}{9}$};
	
	\node at (axis cs:0,18,6) {$\scalemath{1.75}{10}$};
	\node at (axis cs:2,16,6) {$\scalemath{1.75}{10}$};
	\node[color=gray] at (axis cs:4,14,6) {$\scalemath{1.75}{10}$};
	
	\node at (axis cs:0,16,8) {$\scalemath{1.75}{11}$};
	\node[color=gray] at (axis cs:2,14,8) {$\scalemath{1.75}{11}$};
	
	\node[color=gray] at (axis cs:0,14,10) {$\scalemath{1.75}{11}$};
	
	\addplot3[color=blue,ultra thick] coordinates {(2,6,2)(2,4,4)(0,4,6)};
	
	\addplot3[color=black, thick] coordinates {(6,2,2)(0,8,2)};
	\addplot3[color=black, thick] coordinates {(6,2,2)(0,2,8)};
	
  \end{ternaryaxis}
\end{tikzpicture}}

\vspace{-3.15mm}

\scalemath{.4}{
\begin{tikzpicture}
  \begin{ternaryaxis}[
    separate axis lines,
    z axis line style= { draw opacity=.1 },
    ylabel=$\scalemath{2}{\mu}$,xlabel=$\scalemath{2}{\omega_2}$,xlabel style={xshift=0.5cm},
    clip=false,xticklabels={},yticklabels={},zticklabels={},xmin=0,
    ymin=0,
    zmin=0,
    xmax=10,
    ymax=10,
    zmax=10,
    yscale=-1,
    xtick={2,4,6,8},
    ytick={2,4,6,8},
    ztick={2,4,6,8},
    major tick length=0,
    %axis .style={opacity=.1}
    %%opacity=0.2,
    %%clip=false
]    
  	\node[color=gray] at (axis cs:0,24,0) {$\scalemath{1.75}{0}$};
	\node at (axis cs:2,22,0) {$\scalemath{1.75}{4}$};
	\node at (axis cs:4,20,0) {$\scalemath{1.75}{6}$};
	\node at (axis cs:6,18,0) {$\scalemath{1.75}{8}$};
	\node at (axis cs:8,16,0) {$\scalemath{1.75}{9}$};
	\node at (axis cs:10,14,0) {$\scalemath{1.75}{9}$};
	
	\node at (axis cs:0,22,2) {$\scalemath{1.75}{5}$};
	\node at (axis cs:2,20,2) {$\scalemath{1.75}{7}$};
	\node at (axis cs:4,18,2) {$\scalemath{1.75}{9}$};
	\node at (axis cs:6,16,2) {$\scalemath{1.75}{10}$};
	\node at (axis cs:8,14,2) {$\scalemath{1.75}{10}$};
	
	\node at (axis cs:0,20,4) {$\scalemath{1.75}{8}$};
	\node at (axis cs:2,18,4) {$\scalemath{1.75}{9}$};
	\node at (axis cs:4,16,4) {$\scalemath{1.75}{11}$};
	\node at (axis cs:6,14,4) {$\scalemath{1.75}{11}$};
	
	\node at (axis cs:0,18,6) {$\scalemath{1.75}{10}$};
	\node at (axis cs:2,16,6) {$\scalemath{1.75}{11}$};
	\node at (axis cs:4,14,6) {$\scalemath{1.75}{11}$};
	
	\node at (axis cs:0,16,8) {$\scalemath{1.75}{11}$};
	\node at (axis cs:2,14,8) {$\scalemath{1.75}{11}$};
	
	\node[color=gray] at (axis cs:0,14,10) {$\scalemath{1.75}{11}$};
	
	\addplot3[color=red,ultra thick] coordinates {(2,8,0)(2,6,2)(4,4,2)(4,2,4)(6,0,4)};
	\addplot3[color=blue,ultra thick] coordinates {(0,4,6)(2,2,6)(2,0,8)};

	\addplot3[color=black] coordinates {(2,8,0)(2,0,8)};
	\addplot3[color=black, thick] coordinates {(0,8,2)(2,8,0)};
	\addplot3[color=black, thick] coordinates {(0,2,8)(2,0,8)};
	
	\fill[gray, opacity=.5] (axis cs:0,22,2)--(axis cs:0,24,0)--(axis cs:2,22,0)--cycle;
        \fill[gray, opacity=.5] (axis cs:0,14,10)--(axis cs:2,14,8)--(axis cs:0,16,8)--cycle;
	
  \end{ternaryaxis}
\end{tikzpicture}}
\subcaption{One shave step, old values in gray \label{top shaving}}
\endminipage
\caption{Excavating a $4$-hive}
\end{figure}

The newly exposed, smaller tetrahedron has a new top face B, whose values were the already exposed values of the original top face B, and a new top face A, whose values were previously hidden. To get $\rho_2$, or equivalently the third strip of the bottom left face, the top face A of this new tetrahedron is shaved off and so forth. We call such a step in our process a \emph{shave step} to differentiate from a single application of the octahedron recurrence at a specific lattice point. We already observed that every strip of the original face A is an $\omega_t$ for some $t$. The same reasoning applies to the top face A of the new tetrahedron, and induction gives the following lemma. 

\begin{Lem}
Every strip of a newly exposed top face A has differences $\omega_t$ for some $t$.
\end{Lem}
\begin{proof}
The NW external edge of each new top face A is a strip of the bottom left face, which we know by Lemma \ref{lem:strips} has differences $\omega_t$ for some $t$. Apply the lemma repeatedly to conclude the same for every strip.
\end{proof}

To get each $\rho_k$ we must understand which $\omega_t$ appear after a single shave step, or equivalently, how the break paths change after a shave step. Since the new top face B is obtained simply by ignoring the previous main horizontal, the break path contracts along the slanted step from the main horizontal to the horizontal one row down. If there is no such slanted step, then the break path must itself be the main horizontal, in which case the new break path is the new main horizontal. 

Analyzing the change in the break path of top face A requires looking at a number of cases. The goal is the following technical lemma. We say that the break path of face A has an \emph{out elbow} at strip $k$ if there's a slanted step into strip $k$ followed by a horizontal step. Similarly, we say that face B has an \emph{out elbow} at strip $k$ if there's a horizontal step into strip $k$ followed by a slanted step.

\begin{Lem}{\label{lem:shave step}}
Let $\omega_{t_k}$ denote the weight along the $k$th strip of face A where the strips are numbered $1$ to $n+1$ from left to right. For $k$ between $2$ and $n+1$ let $\omega_{t_k'}$ be the weight along the $k$th strip of the newly exposed face A after a single shave step. Then for any $k$ either $t_k'=t_k$ or $t_k'=t_k+1$. 

There is at most a single interval $[a,b]$, possibly empty, of indices $k$ for which $t_k'=t_k+1$. In this case, it is necessary that the break path of face A has an out elbow at strip $b$ and that the break path of face B has an out elbow on the main horizontal at strip $b$. Furthermore, $\mu_k=\mu_{k-1}$ for all $k\in [a,b]$. In particular, if the first step of the face B break path is a slanted step, or if $\mu_k>\mu_{k+1}$ for all $k$, then $t_k'=t_k$ for all $k$.
\end{Lem}

\begin{proof}
Recall the octahedron recurrence (\ref{oct rec}), pictured in Figure \ref{fig:octahedron}. Excavating a single position depends on the value of the position itself and the four values that are directly NE, SE, SW, NW of the position in question. However, in analyzing the cases arising, it will be easier to keep track of two extra lattice points. The cases will depend on values of the following relative positions where we are excavating the middle position of the middle row. The values of these relative positions will be called the \emph{local data}.
\begin{align*}
\begin{array}{ccccccc}
&&\square&&\square&&\\
&\square&&\square&&\square&\\
&&\square&&\square&&\\
\end{array}
\end{align*}
A single shave step consists of inductive steps of excavating rows (first the main horizontal, then the one above, and so forth). The base case of this induction is the excavation of the main horizontal of the original $4$-hive and is pictured in Figure \ref{horizontal excavation}. 
\subsection{Excavating Below the Break Path}
%%%%%%%%Begin case A1 %%%%%%%%%

\begin{figure}[t!]
\centering
\minipage{0.32\textwidth}
\centering
\scalemath{.4}{
\begin{tikzpicture}
  \begin{ternaryaxis}[
    xmin=0,
    ymin=0,
    zmin=0,
    xmax=15,
    ymax=15,
    zmax=15,
    xtick={0, ..., 15},
    ytick={0, ..., 15},
    ztick={0, ..., 15},
    xticklabels={},
    yticklabels={},
    zticklabels={},
    major tick length=0,
        minor tick length=0,
            every axis grid/.style={gray},
    minor tick num=1,
    ]
 
   \addplot3[color=blue,ultra thick] coordinates {(0,0,15)(0,3,12)(1,3,11)(2,3,10)(3,3,9)(4,3,8)(5,3,7)(5,4,6)(5,5,5)(5,6,4)(5,7,3)(6,7,2)(7,7,1)(8,7,0)};

  \end{ternaryaxis}
\end{tikzpicture}}

\vspace{-3.5mm}

\scalemath{.4}{
\begin{tikzpicture}
  \begin{ternaryaxis}[
  clip=false,
    xmin=0,
    ymin=0,
    zmin=0,
    xmax=15,
    ymax=15,
    zmax=15,
    yscale=-1,
    xtick={0, ..., 15},
    ytick={0, ..., 15},
    ztick={0, ..., 15},
    xticklabels={},
    yticklabels={},
    zticklabels={},
    major tick length=0,
        minor tick length=0,
            every axis grid/.style={gray},
    minor tick num=1,
    ]
            \node[label={180:{}},circle,fill,inner sep=2pt] at (axis cs:0,6,9) {};
            \node[label={180:{}},circle,fill,inner sep=2pt] at (axis cs:0,7,8) {};
            \node[label={180:{}},circle,fill,inner sep=2pt] at (axis cs:0,8,7) {};
            
            \node[label={180:{}},circle,fill,inner sep=2pt] at (axis cs:1,6,8) {};
            \node[label={180:{}},circle,fill,inner sep=2pt] at (axis cs:1,7,7) {};
            
            \node[label={180:{}},circle,fill,inner sep=2pt] at (axis cs:-1,7,9) {};
            \node[label={180:{}},circle,fill,inner sep=2pt] at (axis cs:-1,8,8) {};
 
   \addplot3[color=red,ultra thick] coordinates {(0,15,0)(0,11,4)(11,0,4)};
   
  \end{ternaryaxis}
\end{tikzpicture}}
    \subcaption{Case A(i)}
\endminipage\hfill
\minipage{0.32\textwidth}
\centering
\scalemath{.4}{
\begin{tikzpicture}
  \begin{ternaryaxis}[
    xmin=0,
    ymin=0,
    zmin=0,
    xmax=15,
    ymax=15,
    zmax=15,
    xtick={0, ..., 15},
    ytick={0, ..., 15},
    ztick={0, ..., 15},
    xticklabels={},
    yticklabels={},
    zticklabels={},
    major tick length=0,
        minor tick length=0,
            every axis grid/.style={gray},
    minor tick num=1,
    ]
 
   \addplot3[color=blue,ultra thick] coordinates {(0,0,15)(0,3,12)(1,3,11)(2,3,10)(3,3,9)(4,3,8)(5,3,7)(5,4,6)(5,5,5)(5,6,4)(5,7,3)(6,7,2)(7,7,1)(8,7,0)};

  \end{ternaryaxis}
\end{tikzpicture}}

\vspace{-3.5mm}

\scalemath{.4}{
\begin{tikzpicture}
  \begin{ternaryaxis}[
  clip=false,
    xmin=0,
    ymin=0,
    zmin=0,
    xmax=15,
    ymax=15,
    zmax=15,
    yscale=-1,
    xtick={0, ..., 15},
    ytick={0, ..., 15},
    ztick={0, ..., 15},
    xticklabels={},
    yticklabels={},
    zticklabels={},
    major tick length=0,
        minor tick length=0,
            every axis grid/.style={gray},
    minor tick num=1,
    ]
               \node[label={180:{}},circle,fill,inner sep=2pt] at (axis cs:0,6,9) {};
            \node[label={180:{}},circle,fill,inner sep=2pt] at (axis cs:0,7,8) {};
            \node[label={180:{}},circle,fill,inner sep=2pt] at (axis cs:0,8,7) {};
            
            \node[label={180:{}},circle,fill,inner sep=2pt] at (axis cs:1,6,8) {};
            \node[label={180:{}},circle,fill,inner sep=2pt] at (axis cs:1,7,7) {};
            
            \node[label={180:{}},circle,fill,inner sep=2pt] at (axis cs:-1,7,9) {};
            \node[label={180:{}},circle,fill,inner sep=2pt] at (axis cs:-1,8,8) {};
 
   \addplot3[color=red,ultra thick] coordinates {
 (0,15,0)(0,4,11)(4,0,11)
};
  \end{ternaryaxis}
\end{tikzpicture}}
\subcaption{Case A(ii)}
\endminipage\hfill
\minipage{0.32\textwidth}
\centering
\scalemath{.4}{
\begin{tikzpicture}
  \begin{ternaryaxis}[
    xmin=0,
    ymin=0,
    zmin=0,
    xmax=15,
    ymax=15,
    zmax=15,
    xtick={0, ..., 15},
    ytick={0, ..., 15},
    ztick={0, ..., 15},
    xticklabels={},
    yticklabels={},
    zticklabels={},
    major tick length=0,
        minor tick length=0,
            every axis grid/.style={gray},
    minor tick num=1,
    ]
 
   \addplot3[color=blue,ultra thick] coordinates {(0,0,15)(0,3,12)(1,3,11)(2,3,10)(3,3,9)(4,3,8)(5,3,7)(5,4,6)(5,5,5)(5,6,4)(5,7,3)(6,7,2)(7,7,1)(8,7,0)};

  \end{ternaryaxis}
\end{tikzpicture}}

\vspace{-3.5mm}

\scalemath{.4}{
\begin{tikzpicture}
  \begin{ternaryaxis}[
  clip=false,
    xmin=0,
    ymin=0,
    zmin=0,
    xmax=15,
    ymax=15,
    zmax=15,
    yscale=-1,
    xtick={0, ..., 15},
    ytick={0, ..., 15},
    ztick={0, ..., 15},
    xticklabels={},
    yticklabels={},
    zticklabels={},
    major tick length=0,
        minor tick length=0,
            every axis grid/.style={gray},
    minor tick num=1,
    ]
            \node[label={180:{}},circle,fill,inner sep=2pt] at (axis cs:0,6,9) {};
            \node[label={180:{}},circle,fill,inner sep=2pt] at (axis cs:0,7,8) {};
            \node[label={180:{}},circle,fill,inner sep=2pt] at (axis cs:0,8,7) {};
            
            \node[label={180:{}},circle,fill,inner sep=2pt] at (axis cs:1,6,8) {};
            \node[label={180:{}},circle,fill,inner sep=2pt] at (axis cs:1,7,7) {};
            
            \node[label={180:{}},circle,fill,inner sep=2pt] at (axis cs:-1,7,9) {};
            \node[label={180:{}},circle,fill,inner sep=2pt] at (axis cs:-1,8,8) {};
 
   \addplot3[color=red,ultra thick] coordinates {
   (0,15,0)(0,7,8)(7,0,8)

};
  \end{ternaryaxis}
\end{tikzpicture}}
\subcaption{Case A(iii)}
\endminipage
\caption{Case A (main horizontal)\label{case A1}}
\end{figure}

%%%%%%% End case A1 %%%%%%%

%%%%%%%%Begin case A2 %%%%%%%%%

\begin{figure}[t!]
\centering
\minipage{0.32\textwidth}
\centering
\scalemath{.4}{
\begin{tikzpicture}
  \begin{ternaryaxis}[
    xmin=0,
    ymin=0,
    zmin=0,
    xmax=15,
    ymax=15,
    zmax=15,
    xtick={0, ..., 15},
    ytick={0, ..., 15},
    ztick={0, ..., 15},
    xticklabels={},
    yticklabels={},
    zticklabels={},
    major tick length=0,
        minor tick length=0,
            every axis grid/.style={gray},
    minor tick num=1,
    ]
 
   \addplot3[color=blue,ultra thick] coordinates {(0,0,15)(0,3,12)(1,3,11)(2,3,10)(3,3,9)(4,3,8)(5,3,7)(5,4,6)(5,5,5)(5,6,4)(5,7,3)(6,7,2)(7,7,1)(8,7,0)};

	   \node[label={180:{}},circle,fill,inner sep=2pt] at (axis cs:3,5,7) {};
            \node[label={180:{}},circle,fill,inner sep=2pt] at (axis cs:3,4,8) {};
            
            \node[label={180:{}},circle,fill,inner sep=2pt] at (axis cs:2,4,9) {};
            \node[label={180:{}},circle,fill,inner sep=2pt] at (axis cs:2,5,8) {};
            \node[label={180:{}},circle,fill,inner sep=2pt] at (axis cs:2,6,7) {};
            
             \node[label={180:{}},circle,fill,inner sep=2pt] at (axis cs:1,5,9) {};
            \node[label={180:{}},circle,fill,inner sep=2pt] at (axis cs:1,6,8) {};

  \end{ternaryaxis}
\end{tikzpicture}}

\vspace{-2.3mm}

\scalemath{.4}{
\begin{tikzpicture}
  \begin{ternaryaxis}[
  clip=false,
    xmin=0,
    ymin=0,
    zmin=0,
    xmax=15,
    ymax=15,
    zmax=15,
    yscale=-1,
    xtick={0, ..., 15},
    ytick={0, ..., 15},
    ztick={0, ..., 15},
    xticklabels={},
    yticklabels={},
    zticklabels={},
    major tick length=0,
        minor tick length=0,
            every axis grid/.style={gray},
    minor tick num=1,
    ]
 
   \addplot3[color=red,ultra thick] coordinates {(0,15,0)(0,11,4)(11,0,4)};
   
  \end{ternaryaxis}
\end{tikzpicture}}
    \subcaption{Case A(i)}
\endminipage\hfill
\minipage{0.32\textwidth}
\centering
\scalemath{.4}{
\begin{tikzpicture}
  \begin{ternaryaxis}[
    xmin=0,
    ymin=0,
    zmin=0,
    xmax=15,
    ymax=15,
    zmax=15,
    xtick={0, ..., 15},
    ytick={0, ..., 15},
    ztick={0, ..., 15},
    xticklabels={},
    yticklabels={},
    zticklabels={},
    major tick length=0,
        minor tick length=0,
            every axis grid/.style={gray},
    minor tick num=1,
    ]
 
   \addplot3[color=blue,ultra thick] coordinates {(0,0,15)(0,3,12)(1,3,11)(2,3,10)(3,3,9)(4,3,8)(5,3,7)(5,4,6)(5,5,5)(5,6,4)(5,7,3)(6,7,2)(7,7,1)(8,7,0)};

	\node[label={180:{}},circle,fill,inner sep=2pt] at (axis cs:3,5,7) {};
            \node[label={180:{}},circle,fill,inner sep=2pt] at (axis cs:3,4,8) {};
            
            \node[label={180:{}},circle,fill,inner sep=2pt] at (axis cs:2,4,9) {};
            \node[label={180:{}},circle,fill,inner sep=2pt] at (axis cs:2,5,8) {};
            \node[label={180:{}},circle,fill,inner sep=2pt] at (axis cs:2,6,7) {};
            
             \node[label={180:{}},circle,fill,inner sep=2pt] at (axis cs:1,5,9) {};
            \node[label={180:{}},circle,fill,inner sep=2pt] at (axis cs:1,6,8) {};

  \end{ternaryaxis}
\end{tikzpicture}}

\vspace{-2.3mm}

\scalemath{.4}{
\begin{tikzpicture}
  \begin{ternaryaxis}[
  clip=false,
    xmin=0,
    ymin=0,
    zmin=0,
    xmax=15,
    ymax=15,
    zmax=15,
    yscale=-1,
    xtick={0, ..., 15},
    ytick={0, ..., 15},
    ztick={0, ..., 15},
    xticklabels={},
    yticklabels={},
    zticklabels={},
    major tick length=0,
        minor tick length=0,
            every axis grid/.style={gray},
    minor tick num=1,
    ]
 
   \addplot3[color=red,ultra thick] coordinates {
 (0,15,0)(0,4,11)(4,0,11)
};
  \end{ternaryaxis}
\end{tikzpicture}}
\subcaption{Case A(ii)}
\endminipage\hfill
\minipage{0.32\textwidth}
\centering
\scalemath{.4}{
\begin{tikzpicture}
  \begin{ternaryaxis}[
    xmin=0,
    ymin=0,
    zmin=0,
    xmax=15,
    ymax=15,
    zmax=15,
    xtick={0, ..., 15},
    ytick={0, ..., 15},
    ztick={0, ..., 15},
    xticklabels={},
    yticklabels={},
    zticklabels={},
    major tick length=0,
        minor tick length=0,
            every axis grid/.style={gray},
    minor tick num=1,
    ]
 
   \addplot3[color=blue,ultra thick] coordinates {(0,0,15)(0,3,12)(1,3,11)(2,3,10)(3,3,9)(4,3,8)(5,3,7)(5,4,6)(5,5,5)(5,6,4)(5,7,3)(6,7,2)(7,7,1)(8,7,0)};
   
            \node[label={180:{}},circle,fill,inner sep=2pt] at (axis cs:3,5,7) {};
            \node[label={180:{}},circle,fill,inner sep=2pt] at (axis cs:3,4,8) {};
            
            \node[label={180:{}},circle,fill,inner sep=2pt] at (axis cs:2,4,9) {};
            \node[label={180:{}},circle,fill,inner sep=2pt] at (axis cs:2,5,8) {};
            \node[label={180:{}},circle,fill,inner sep=2pt] at (axis cs:2,6,7) {};
            
             \node[label={180:{}},circle,fill,inner sep=2pt] at (axis cs:1,5,9) {};
            \node[label={180:{}},circle,fill,inner sep=2pt] at (axis cs:1,6,8) {};

  \end{ternaryaxis}
\end{tikzpicture}}

\vspace{-2.3mm}

\scalemath{.4}{
\begin{tikzpicture}
  \begin{ternaryaxis}[
  clip=false,
    xmin=0,
    ymin=0,
    zmin=0,
    xmax=15,
    ymax=15,
    zmax=15,
    yscale=-1,
    xtick={0, ..., 15},
    ytick={0, ..., 15},
    ztick={0, ..., 15},
    xticklabels={},
    yticklabels={},
    zticklabels={},
    major tick length=0,
        minor tick length=0,
            every axis grid/.style={gray},
    minor tick num=1,
    ]

   \addplot3[color=red,ultra thick] coordinates {
   (0,15,0)(0,7,8)(7,0,8)

};
  \end{ternaryaxis}
\end{tikzpicture}}
\subcaption{Case A(iii)}
\endminipage
\caption{Case A (inductive step)\label{case A2}}
\end{figure}

%%%%%%% End case A2 %%%%%%%

We will first consider excavating points strictly below the break path of face A. Even more specifically, begin with a lattice point on the main horizontal. There are three subcases depending on the break path of face B, as shown in Figure \ref{case A1}. If the point is on the $k+1$ strip, then the local data is the following. 
\begin{align*}
\begin{array}{ccccccc}
&&\sum_1^{k-1}\nu_i+1&&\sum_1^{k}\nu_i+1&&\\
&\sum_1^{k-1}\nu_i&&\sum_1^{k}\nu_i&&\sum_1^{k+1}\nu_i&\\
&&\sum_1^{k}\nu_i-x&& \sum_1^{k+1}\nu_i-y&&
\end{array}
\end{align*} 
Since we are excavating a point on the main horizontal, the two bottom lattice points are on face B, and the $x$, $y$ variables are $0$ or $1$ depending on the break path of face B. Only the values of $x$ and $y$ will affect excavation at this point, so the possible cases are $x=0=y$, or $x=1=y$, or $x=1$ and $y=0$, depending, respectively, on whether the break path of face B does not touch the excavation point, touches all three middle points, or has an out elbow at the excavation point, as shown in Figure \ref{case A1}.

Let $z$ be the value of the lattice point that the current excavation will reveal. The value of $z$ must be either $\sum_1^{k}\nu_i-x$ or $\sum_1^{k}\nu_i-x+1$ because we know that every strip of the new face will be an $\omega_i$ for some $i$, and the value of the first lattice point of this strip is $\sum_1^{k}\nu_i-x$. In each case we show that $z$ is actually $\sum_1^{k}\nu_i -x+1$, hence there is a difference of $1$ between the first and second positions of the revealed strip. Cases (i) and (ii) below are implied by $\nu_{k+1}\leq \nu_k$ because $\nu$ is dominant.
\begin{enumerate}[itemindent=.6cm,labelwidth=\itemindent,align=left,label=(\roman*)]
\item
($x=0,y=0$) \hspace{5mm}$z=\max\left(\sum_1^{k}\nu_i+1,\sum_1^{k-1}\nu_i+1+\nu_{k+1}\right)=\sum_1^{k}\nu_i+1$
\item
($x=1,y=1$) \hspace{5mm}$z=\max\left(\sum_1^{k}\nu_i,\sum_1^{k-1}\nu_i+\nu_{k+1}\right)=\sum_1^{k}\nu_i$
\item
($x=1,y=0$) \hspace{5mm}$z=\max\left(\sum_1^{k}\nu_i,\sum_1^{k-1}\nu_i+\nu_{k+1}+1\right)=\sum_1^{k}\nu_i$
\end{enumerate}
In case (iii), the second input to the maximum function can be $z=\sum_1^{k}\nu_i+1$ only if $\nu_{k+1}=\nu_{k}$, but this is impossible because the newly excavated line has to be an $\omega_t$ for some $t$, yet this value together with $x=1$ gives a difference of $2$. At this point, the first two lattice points of the new strip are revealed, the first of which is labelled by $\sum_1^{k}\nu_i-x$. In each of the three cases, the difference between the values at these two lattice points is $1$. 

Excavating the main horizontal leaves a trench, where the newly visible lattice points now play the role of the bottom row of the local data. We call these the \emph{trench values}. In the case just analyzed, the trench values were the values on face B. In all three cases above, the \emph{trench difference} $x$ between the trench and the excavation point is preserved, an observation we will need for the inductive step. 

Now consider excavating a point that is on the $k+1$ strip and on the $c$th horizontal line above the main horizontal, but still strictly below the break path of face A. Assume that the two relevant lattice points below this location have already been excavated. The local data looks like the following.
\begin{align*}
\begin{array}{ccccccc}
&&\sum_1^{k-1}\nu_i+c+1&&\sum_1^{k}\nu_i+c+1&&\\
&\sum_1^{k-1}\nu_i+c&&\sum_1^{k}\nu_i+c&&\sum_1^{k+1}\nu_i+c&\\
&&\sum_1^{k}\nu_i+c-x&& \sum_1^{k+1}\nu_i+c-y&&
\end{array}
\end{align*} 
By induction the exposed portion of the new $k+1$ strip increases by one at every step thus far, and the trench differences of $x$ and $y$ in the local data are the same as the trench differences given by the break path of face B. In particular, the same three subcases may arise, as indicated in Figure \ref{case A2}. Since the constant $c$ does not affect the octahedron recurrence, the calculation of $z$ is as before, so we may conclude in this case that both the strip difference and the trench difference propagate. In particular, the trench differences given by the break path of face B propagate all the way up to excavating the points directly on the A break path.

\subsection{Excavating on the Break Path}
Excavating a point on the break path of A is more interesting. Consider case B of Figure \ref{on path}. Since the two positions of the bottom row of the local data were excavated by instances of case A, the trench differences $x,y$ take on the same values originally given by the break path of face B. In particular, the same three possible subcases can arise. The local data is as follows where the constant $c$ is omitted. Note that there are now equalities beyond this point along the $k+1$ strip.
\begin{align*}
\begin{array}{ccccccc}
&&\sum_1^{k-1}\nu_i&&\sum_1^{k}\nu_i&&\\
&\sum_1^{k-1}\nu_i&&\sum_1^{k}\nu_i&&\sum_1^{k+1}\nu_i&\\
&&\sum_1^{k}\nu_i-x&& \sum_1^{k+1}\nu_i-y&&
\end{array}
\end{align*} 
Note that $\nu_{k+1}\leq \nu_k$ again implies the calculations in (i) and (ii) below.
\begin{enumerate}[itemindent=.6cm,labelwidth=\itemindent,align=left,label=(\roman*)]
\item
($x=0,y=0$) \hspace{5mm} $z=\max\left(\sum_1^{k}\nu_i,\sum_1^{k-1}\nu_i+\nu_{k+1}\right)=\sum_1^{k}\nu_i$
\item
($x=1,y=1$) \hspace{5mm} $z=\max\left(\sum_1^{k}\nu_i-1,\sum_1^{k-1}\nu_i+\nu_{k+1}-1\right)=\sum_1^{k}\nu_i-1$
\item
($x=1,y=0$) \hspace{5mm} $z=\max\left(\sum_1^{k}\nu_i-1,\sum_1^{k-1}\nu_i+\nu_{k+1}\right)=\sum_1^{k}\nu_i-1$
\end{enumerate}
In case (iii), the second input to the maximum function is $z=\sum_1^{k} \nu_i$ if $\nu_{k+1}=\nu_{k}$, but we show this is impossible. The values $x=1$ and $y=0$ are given by the break path of B, hence imply that $\nu_k=\mu_k+1$ and $\nu_{k+1}=\mu_{k+1}$. Together with $\nu_{k+1}=\nu_{k}$ this would imply, $\mu_{k+1}=\nu_{k+1}=\nu_{k}=\mu_{k}+1$, contradicting $\mu$ being dominant.

In each case (i), (ii), and (iii), the strip equality and the trench difference propagate to the strip below. At this point of excavating a strip that falls under case B, the newly revealed strip has $c$ increasing steps (which we knew from case A), followed by a constant step.

Now consider case C, the interesting case where the excavation point is at an out elbow in the break path of A. In this case, the break path of A implies that $\nu_k=\lambda_k+1$ and $\nu_{k+1}=\lambda_{k+1}$. Note that $\nu_{k+1}\leq\nu_{k}-1$ because $\nu_{k+1}=\nu_{k}$ would imply $\lambda_{k}<\lambda_{k+1}$, contradicting that $\lambda$ is dominant. The local data is as follows. As before, the trench differences $x$ and $y$ are determined by instances of case A, hence were propagated from the break path of B. The common constant $c$ is again omitted. 
\begin{align*}
\begin{array}{ccccccc}
&&\sum_1^{k-1}\nu_i+1&&\sum_1^{k}\nu_i&&\\
&\sum_1^{k-1}\nu_i&&\sum_1^{k}\nu_i&&\sum_1^{k+1}\nu_i&\\
&&\sum_1^{k}\nu_i-x&& \sum_1^{k+1}\nu_i-y&&
\end{array}
\end{align*} 
The calculations in (i) and (ii) below follow from $\nu_{k+1}\leq\nu_{k}-1$.
\begin{enumerate}[itemindent=.6cm,labelwidth=\itemindent,align=left,label=(\roman*)]
\item
($x=0,y=0$) \hspace{5mm} $z=\max\left(\sum_1^{k}\nu_i,\sum_1^{k-1}\nu_i+\nu_{k+1}+1\right)=\sum_1^{k}\nu_i$
\item
($x=1,y=1$) \hspace{5mm} $z=\max\left(\sum_1^{k}\nu_i-1,\sum_1^{k-1}\nu_i+\nu_{k+1}\right)=\sum_1^{k}\nu_i-1$
\item
($x=1,y=0$) \hspace{5mm} $z=\max\left(\sum_1^{k}\nu_i-1,\sum_1^{k-1}\nu_i+\nu_{k+1}+1\right)$%=\sum_1^{k}\nu_i-1
\end{enumerate}
Unlike previous cases, there are two possibilities in case (iii). We have $z=\sum_1^{k}\nu_i-1$ if $\nu_{k+1}<\nu_{k}-1$ and $z=\sum_1^{k}\nu_i$ if $\nu_{k+1}=\nu_{k}-1$. The second case produces a difference of $1$ between the current two positions of the revealed strip, even though there is a difference of $0$ between the two positions in the strip above. Also in this special case, the trench difference $x=1$ turns into $x=0$. We will call this the C$(iii)$ special case.

Note that $x=1$ and $y=0$ corresponds to the break path of face $B$ also having an out elbow along the same strip, since the values of $x$ and $y$ propagated up from the break path of face B. Furthermore, if $x=1$, $y=0$, then $\nu_{k+1}=\nu_{k}-1$ if and only if $\mu_k=\mu_{k+1}$. This is because $x=1$ implies $\nu_{k}=\mu_{k}+1$ and $y=0$ implies $\nu_{k+1}=\mu_{k+1}$. Hence, excavation of a case C strip has thus far revealed $c$ increasing steps, followed by possibly another increasing step if and only if the B break path has an out elbow at the same strip and $\mu_{k}=\mu_{k+1}$.

%%%%%%%% Begin "on break path" cases %%%%%%%%%

\begin{figure}
\centering
\scalemath{.4}{
\begin{tikzpicture}
  \begin{ternaryaxis}[
  title={\Huge Case B},
  clip=false,
    xmin=0,
    ymin=0,
    zmin=0,
    xmax=15,
    ymax=15,
    zmax=15,
    xtick={0, ..., 15},
    ytick={0, ..., 15},
    ztick={0, ..., 15},
    xticklabels={},
    yticklabels={},
    zticklabels={},
    major tick length=0,
        minor tick length=0,
            every axis grid/.style={gray},
    minor tick num=1,
    ]
 
   \addplot3[color=blue,very thick] coordinates {
(0,0,15)
(0,3,12)
(1,3,11)
(2,3,10)
(3,3,9)
(4,3,8)
(5,3,7)
(5,4,6)
(5,5,5)
(5,6,4)
(5,7,3)
(6,7,2)
(7,7,1)
(8,7,0)

};

            \node[label={180:{}},circle,fill,inner sep=2pt] at (axis cs:6,4,5) {};
            \node[label={180:{}},circle,fill,inner sep=2pt] at (axis cs:6,5,4) {};
            
            \node[label={180:{}},circle,fill,inner sep=2pt] at (axis cs:4,5,6) {};
            \node[label={180:{}},circle,fill,inner sep=2pt] at (axis cs:4,6,5) {};
            
            \node[label={180:{}},circle,fill,inner sep=2pt,color=black] at (axis cs:5,6,4) {};
            \node[label={180:{}},circle,fill,inner sep=2pt,color=black] at (axis cs:5,5,5) {};
            \node[label={180:{}},circle,fill,inner sep=2pt,color=black] at (axis cs:5,4,6) {};

  \end{ternaryaxis}
\end{tikzpicture}
}
\scalemath{.4}{
\begin{tikzpicture}
  \begin{ternaryaxis}[
  title={\Huge Case C},
  clip=false,
    xmin=0,
    ymin=0,
    zmin=0,
    xmax=15,
    ymax=15,
    zmax=15,
    xtick={0, ..., 15},
    ytick={0, ..., 15},
    ztick={0, ..., 15},
    xticklabels={},
    yticklabels={},
    zticklabels={},
    major tick length=0,
        minor tick length=0,
            every axis grid/.style={gray},
    minor tick num=1,
    ]
 
   \addplot3[color=blue,very thick] coordinates {
   (0,0,15)
(0,3,12)
(1,3,11)
(2,3,10)
(3,3,9)
(4,3,8)
(5,3,7)
(5,4,6)
(5,5,5)
(5,6,4)
(5,7,3)
(6,7,2)
(7,7,1)
(8,7,0)
};

            \node[label={180:{}},circle,fill,inner sep=2pt] at (axis cs:6,6,3) {};
            \node[label={180:{}},circle,fill,inner sep=2pt] at (axis cs:6,7,2) {};
            
            \node[label={180:{}},circle,fill,inner sep=2pt] at (axis cs:4,7,4) {};
            \node[label={180:{}},circle,fill,inner sep=2pt] at (axis cs:4,8,3) {};
            
            \node[label={180:{}},circle,fill,inner sep=2pt,color=black] at (axis cs:5,8,2) {};
            \node[label={180:{}},circle,fill,inner sep=2pt,color=black] at (axis cs:5,7,3) {};
            \node[label={180:{}},circle,fill,inner sep=2pt,color=black] at (axis cs:5,6,4) {};

  \end{ternaryaxis}
\end{tikzpicture}
}
\scalemath{.4}{
\begin{tikzpicture}
  \begin{ternaryaxis}[
  title={\Huge Case D},
  clip=false,
    xmin=0,
    ymin=0,
    zmin=0,
    xmax=15,
    ymax=15,
    zmax=15,
    xtick={0, ..., 15},
    ytick={0, ..., 15},
    ztick={0, ..., 15},
    xticklabels={},
    yticklabels={},
    zticklabels={},
    major tick length=0,
        minor tick length=0,
            every axis grid/.style={gray},
    minor tick num=1,
    ]
 
   \addplot3[color=blue, very thick] coordinates {
   (0,0,15)
(0,3,12)
(1,3,11)
(2,3,10)
(3,3,9)
(4,3,8)
(5,3,7)
(5,4,6)
(5,5,5)
(5,6,4)
(5,7,3)
(6,7,2)
(7,7,1)
(8,7,0)
};

            \node[label={180:{}},circle,fill,inner sep=2pt] at (axis cs:4,2,9) {};
            \node[label={180:{}},circle,fill,inner sep=2pt] at (axis cs:4,3,8) {};
            
            \node[label={180:{}},circle,fill,inner sep=2pt] at (axis cs:2,3,10) {};
            \node[label={180:{}},circle,fill,inner sep=2pt] at (axis cs:2,4,9) {};
            
            \node[label={180:{}},circle,fill,inner sep=2pt,color=black] at (axis cs:3,4,8) {};
            \node[label={180:{}},circle,fill,inner sep=2pt,color=black] at (axis cs:3,3,9) {};
            \node[label={180:{}},circle,fill,inner sep=2pt,color=black] at (axis cs:3,2,10) {};

  \end{ternaryaxis}
\end{tikzpicture}
}
\scalemath{.4}{
\begin{tikzpicture}
  \begin{ternaryaxis}[
  title={\Huge Case E},
  clip=false,
    xmin=0,
    ymin=0,
    zmin=0,
    xmax=15,
    ymax=15,
    zmax=15,
    xtick={0, ..., 15},
    ytick={0, ..., 15},
    ztick={0, ..., 15},
    xticklabels={},
    yticklabels={},
    zticklabels={},
    major tick length=0,
        minor tick length=0,
            every axis grid/.style={gray},
    minor tick num=1,
    ]
 
   \addplot3[color=blue, very thick] coordinates {
   (0,0,15)
(0,3,12)
(1,3,11)
(2,3,10)
(3,3,9)
(4,3,8)
(5,3,7)
(5,4,6)
(5,5,5)
(5,6,4)
(5,7,3)
(6,7,2)
(7,7,1)
(8,7,0)
};

            \node[label={180:{}},circle,fill,inner sep=2pt] at (axis cs:6,2,7) {};
            \node[label={180:{}},circle,fill,inner sep=2pt] at (axis cs:6,3,6) {};
            
            \node[label={180:{}},circle,fill,inner sep=2pt] at (axis cs:4,3,8) {};
            \node[label={180:{}},circle,fill,inner sep=2pt] at (axis cs:4,4,7) {};
            
            \node[label={180:{}},circle,fill,inner sep=2pt,color=black] at (axis cs:5,4,6) {};
            \node[label={180:{}},circle,fill,inner sep=2pt,color=black] at (axis cs:5,3,7) {};
            \node[label={180:{}},circle,fill,inner sep=2pt,color=black] at (axis cs:5,2,8) {};

  \end{ternaryaxis}
\end{tikzpicture}
}
\caption{All of the ``on break path'' cases \label{on path}}
\end{figure}

%%%%%%% End "on break path" cases %%%%%%%%%

Now consider case D, the other interesting case with local data as follows. 
\begin{align*}
\begin{array}{ccccccc}
&&\sum_1^{k-1}\nu_i+2&&\sum_1^{k}\nu_i+1&&\\
&\sum_1^{k-1}\nu_i+1&&\sum_1^{k}\nu_i+1&&\sum_1^{k+1}\nu_i&\\
&&\sum_1^{k}\nu_i+1-x&& \sum_1^{k+1}\nu_i-y&&
\end{array}
\end{align*} 
As before, the position containing $x$ was excavated by an instance of case A, so this trench value propagated from the original trench value given by break path B. The position containing $y$ however could have been excavated by an instance of case C, or by another instance of case D. If it was excavated by an instance of case C, then it is possible that the special case of C changed a trench difference of a $1$ to a $0$. Either way, we have the same three subcases for $x,y$ values. 
\begin{enumerate}[itemindent=.6cm,labelwidth=\itemindent,align=left,label=(\roman*)]
\item
($x=0,y=0$) \hspace{5mm} $z=\max\left(\sum_1^{k}\nu_i+1,\sum_1^{k-1}\nu_i+\nu_{k+1}+1\right)=\sum_1^{k}\nu_i+1$
\item
($x=1,y=1$) \hspace{5mm} $z=\max\left(\sum_1^{k}\nu_i,\sum_1^{k-1}\nu_i+\nu_{k+1}\right)=\sum_1^{k}\nu_i$
\item
($x=1,y=0$) \hspace{5mm} $z=\max\left(\sum_1^{k}\nu_i,\sum_1^{k-1}\nu_i+\nu_{k+1}+1\right)$%=\sum_1^{k}\nu_i-1
\end{enumerate}
In case (iii), $z=\sum_1^{k}\nu_i$ if $\nu_{k+1}<\nu_{k}$, and $z=\sum_1^{k}\nu_i+1$ if $\nu_{k+1}=\nu_{k}$. In the first case, the strip difference and trench difference propagate. The second case is more interesting. It corresponds to an increase in the current position of the revealed strip, even though there was no such increase in the original strip. Also in this second case the trench difference $x=1$ turns into $x=0$. 

Let's consider when this special case of $\nu_{k+1}=\nu_{k}$ can actually happen. One way to end up with $x=1$ and $y=0$ is if the break path of the top face B starts by giving $x=1$ and $y=0$ at the very first step. Then $\nu_{k+1}=\nu_{k}$ forces $\mu_{k}=\nu_{k}-1<\nu_{k}=\nu_{k+1}=\mu_{k+1}$ which is a contradiction. The only other way to end up with $x=1$ and $y=0$ is if the break path of top face B starts with $x=1$ and $y=1$ on the $k+1$ and $k+2$ strips, and the excavation of the position to the SE changes $y=1$ into $y=0$, as in the special case of C(iii). In this case, $\nu_{k}=\nu_{k+1}$ if and only if $\mu_{k}=\mu_{k+1}$. 

If this special case of D(iii) occurs then the strip difference of $0$ becomes a difference of $1$, and the trench difference of $x=1$ turns into a trench difference of $x=0$. By induction, an instance of case D can only have the above three possible values for $x$ and $y$. In particular, the special case of D(iii) can trigger another instance of the special case of D(iii) directly to the NW. Such a string of special case D(iii)'s can only occur if the special case C(iii) occurs at the out elbow of the same NW-SE strip.

Finally, for case E the position involving $y$ could have been excavated by an instance of case C or D, in which case we know that $y=0$ or $y=1$. We claim there are no interesting subcases, meaning that the strip difference and trench difference propagate. We leave the routine analysis to the reader.

This concludes the cases of excavating a point on the break path. At this point we have shown that strip differences and trench differences always propagate when below the break path, and also when on the break path, except for the special cases of C(iii) and D(iii). The special cases of C(iii) and D(iii) can occur at strip $k+1$ only if $\mu_{k}=\mu_{k+1}$. In both special cases, the strip difference of $0$ flips to a $1$, while the trench difference of $1$ flips to a $0$.

We claim that there are no new interesting cases when excavating a point above the break path, i.e. the strip difference and trench difference propagate. We leave this analysis to the reader.
\end{proof}

The only assumptions that we have made on the original $4$-hive is that the NW and SE edges are fundamental weights and that $\lambda,\mu,\nu$ are dominant weights. After one shave step we see that these assumptions still hold, so the same applies to the smaller $4$-hive. Note though that the break paths change accordingly.

\subsection{Completing the Proof}
Recall that $\nu_k=\lambda_k+1$ if and only if the $k$th strip of face A is labeled by $\omega_t$ and the $k+1$st strip by $\omega_{t-1}$ for some $t$. Likewise, $\rho_k=\mu_k-1$ if and only if the $k$th strip of the bottom left face is labeled by $\omega_t$ and the $(k+1)$st by $\omega_{t-1}$ for some $t$. This information is revealed after the $k$th shave step. 

Thus, if $\mu_{k}>\mu_{k+1}$ for all $k$, then by the previous proposition each shave step simply copies the $\omega_t$ from the top face to the next. This means that $1$ is added to $\lambda_k$ to get $\nu_k$ if and only if $1$ is subtracted from $\mu_k$ to get $\rho_k$. This is precisely the local condition (\ref{local condition}), $\rho=\mu-(\nu-\lambda)$, without any sorting. 

On the other hand, suppose that $\mu_{k}=\mu_{k+1}$. If $\nu_k=\lambda_k+1$ and $\nu_{k+1}=\lambda_{k+1}$, then setting $\rho_k=\mu_k-1$ and $\rho_{k+1}=\mu_{k+1}$ would result in $\rho_k=\mu_k-1<\mu_{k+1}=\rho_{k+1}$, i.e. not a dominant weight. Instead we will show that sorting occurs, meaning that that $\rho_i=\mu_i$ for all indices $i\geq k$, up until the smallest index $j$ for which $\rho_j=\mu_j-1$ gives a dominant weight. We will show that the special cases of C(iii) and D(iii) in the proof of the previous proposition facilitate the sorting. 

\begin{proof}[Proof of Proposition \ref{prop:hive excavation}]
We would like to show that subtracting $\nu-\lambda$ from $\mu$ and sorting gives $\rho$. For a dominant weight $\mu$ let $i_l<j_l$ be the indices denoting the stretches of equalities of the entries of $\mu$,
\begin{align*}
\mu_1&>\mu_2
>\cdots
>\mu_{i_1}=\mu_{i_1+1}=\cdots=\mu_{j_1}
>\mu_{j_1+1}
>\cdots
>\mu_{i_2}=\mu_{i_2+1}=\cdots=\mu_{j_2}
>\cdots
>\mu_n.
\end{align*}
Note it is possible that $j_l+1=i_{l+1}$. Recall that the SW-to-NE-oriented strips of the $4$-hive are indexed $1,\ldots, n+1$ left to right where the last strip is a single lattice point. The parts of the dominant weights are indexed $1,\ldots, n$ and are given by the successive differences of the labels of the $1,\ldots, n+1$ lattice points.

First consider an index $k\notin[i_l,j_l]$ for all $l$, which is equivalent to $\mu_{k-1}\not=\mu_{k}\not=\mu_{k+1}$. Since $\mu_{k-1}\not=\mu_{k}$, the special cases C(iii) and D(iii) cannot occur along strip $k$. If $\nu_k=\lambda_k$, then the face A break path between strips $k$ and $k+1$ is a horizontal step, and remains so upon excavation, so $\rho_k=\mu_k$. Likewise, if $\nu_k=\lambda_k+1$, then the break path has a slanted step between strips $k$ and $k+1$. Since $\mu_k\not=\mu_{k+1}$, the special cases cannot occur along strip $k+1$, so $\rho_k=\mu_k-1$. In both cases, this coincides with the local condition.

Now consider the indices in the interval $[i_l,j_l]$ for some $l$, and restrict $\nu$ and $\mu$ to this interval. Since $\nu$ is dominant and the entries of $\mu$ are equal for these indices, it follows that $\nu_k=\mu_k+1$ for $k$ ranging in some initial subinterval $[i_l,h_l]$ for $i_l-1\leq h_l\leq j_l$, and $\nu_k=\mu_k$ for $k$ in $[h_l+1,j_l]$. Note that $h_l=i_l-1$ corresponds to this initial interval being empty. The break path of face B is horizontal for the steps between strips $i_l$ to $h_l+1$ and slanted between strips $h_l+1$ to $j_l+1$. See Figure $\ref{break paths of interval}$ where the strips $i_l$ to $j_l+1$ are bold.

If $\nu_k=\lambda_k$ for every $k\in[i_l,j_l]$, then the break path of face A is horizontal across the strips $i_l$ to $j_l+1$. The special cases of the proposition cannot occur for any of these strips, so their weight labels remain the same after every shave step. This gives $\rho_k=\mu_k$ for all $k\in[i_l,j_l]$ as desired.

Suppose that there is at least one index $k$ in $[i_l,j_l]$ such that $\nu_k=\lambda_k+1$. We would like to understand the possibilities for the break path of face A. Let $z=\mu_{i_l}$. Then $(\mu_{i_l},\ldots,\mu_{j_l})=(z,\ldots,z)$ and $(\nu_{i_l},\ldots,\nu_{j_l})=(z+1,\ldots,z+1,z,\ldots,z)$ where the entries are $z+1$ for indices $i_l$ to $h_l$. Then $(\lambda_{i_l},\ldots,\lambda_{j_l})=(z+1,\ldots,z+1,z,\ldots,z,z-1,\ldots,z-1)$ where the entries are $z+1$ up to some index $a\leq h_l$ and the entries are $z-1$ starting at some index $b> h_l$. This gives the intervals $[i_l,a], [a+1,h_l], [h_l+1,b], [b+1,j_l]$. The weights $\mu, \nu, \lambda$ on the interval $[i_l,j_l]$ can be visualized as outlines of Young diagrams with infinitely many boxes to the left in each row.
\begin{align*}
\mu= \begin{tikzpicture}[baseline=8ex,scale=.35]
\draw[ultra thick] (0,0) -- (0,10);
\draw (0,10) rectangle (-1,9);
\draw (0,9) rectangle (-1,8);
\draw (0,8) rectangle (-1,7);
\draw (0,7) rectangle (-1,6);
\draw (0,6) rectangle (-1,5);
\draw (0,5) rectangle (-1,4);
\draw (0,4) rectangle (-1,3);
\draw (0,3) rectangle (-1,2);
\draw (0,2) rectangle (-1,1);
\draw (0,1) rectangle (-1,0);
\end{tikzpicture}
\hspace{10mm}
\nu= \begin{tikzpicture}[baseline=8ex,scale=.35]
\draw[ultra thick] (0,0) -- (0,10);
\draw (0,10) rectangle (1,9);
\draw (0,9) rectangle (1,8);
\draw (0,8) rectangle (1,7);
\draw (0,7) rectangle (1,6);
\draw (0,6) rectangle (1,5);
\draw (0,10) rectangle (-1,9);
\draw (0,9) rectangle (-1,8);
\draw (0,8) rectangle (-1,7);
\draw (0,7) rectangle (-1,6);
\draw (0,6) rectangle (-1,5);
\draw (0,5) rectangle (-1,4);
\draw (0,4) rectangle (-1,3);
\draw (0,3) rectangle (-1,2);
\draw (0,2) rectangle (-1,1);
\draw (0,1) rectangle (-1,0);
\end{tikzpicture}
\hspace{10mm}
\lambda= \begin{tikzpicture}[baseline=8ex,scale=.35]
\draw[ultra thick] (1,0) -- (1,10);
\draw (0,10) rectangle (1,9);
\draw (0,9) rectangle (1,8);
\draw (0,8) rectangle (1,7);
\draw (0,7) rectangle (1,6);
\draw (0,6) rectangle (1,5);
\draw (0,5) rectangle (1,4);
\draw (0,4) rectangle (1,3);
\draw (1,10) rectangle (2,9);
\end{tikzpicture}
\end{align*}

The corresponding portion of face A's break path consists of $a-i_l+1$ horizontal steps, then $h_l-a$ slanted steps, followed by $b-h_l$ horizontal steps, and finally $j_l+b$ slanted steps. For the example above, there are $1$, $4$, $2$, $3$ such steps respectively. Certain consecutive segments may be empty, but there is at least one slanted step since we are assuming $\nu_k=\lambda_k+1$ for some $k\in [i_l,j_l]$. More importantly, there is at most one position where a slanted step is followed by a horizontal step, i.e. there is at most one out elbow. Furthermore, if there is such an out elbow on face A, then it must occur at strip $h_l+1$, the same strip at which the break path of face B contains an out elbow.

If there is no out elbow on face A, then the break path begins with consecutive horizontal steps (possibly none), followed by consecutive slanted steps. This means that $\nu_k=\lambda_k+1$ for indices $k$ in a terminal subinterval of $[i_l,j_l]$. Since the special cases of the previous proposition cannot arise, $\rho_k=\mu_k-1$ for all such $k$, which is dominant because $\mu_{j_l}>\mu_{j_l+1}$, so this agrees with the local condition.

\begin{figure}
\centering
\scalemath{.5}{
\begin{tikzpicture}
  \begin{ternaryaxis}[
    xmin=0,
    ymin=0,
    zmin=0,
    xmax=30,
    ymax=30,
    zmax=30,
    xtick={0, ..., 30},
    ytick={0, ..., 30},
    ztick={0, ..., 30},
    xticklabels={},
    yticklabels={},
    zticklabels={},
    major tick length=0,
        minor tick length=0,
            every axis grid/.style={gray},
    minor tick num=1,
    ]
 
   \addplot3[color=blue,ultra thick] coordinates {(11,19,0)(10,19,1)(10,17,3)(9,17,4)(9,14,7)(9,12,9)(6,12,12)(6,9,15)(4,9,17)(4,4,22)(0,4,26)(0,0,30)};
   
   \addplot3[thick] coordinates {(0,23,7)(23,0,7)};
   \addplot3[thick] coordinates {(0,22,8)(22,0,8)};
   \addplot3[thick] coordinates {(0,21,9)(21,0,9)};
   \addplot3[thick] coordinates {(0,20,10)(20,0,10)};
   \addplot3[thick] coordinates {(0,19,11)(19,0,11)};
   \addplot3[thick] coordinates {(0,18,12)(18,0,12)};
   \addplot3[thick] coordinates {(0,17,13)(17,0,13)};
   \addplot3[thick] coordinates {(0,16,14)(16,0,14)};
   \addplot3[thick] coordinates {(0,15,15)(15,0,15)};
   \addplot3[thick] coordinates {(0,14,16)(14,0,16)};
   \addplot3[thick] coordinates {(0,13,17)(13,0,17)};

  \end{ternaryaxis}
\end{tikzpicture}}

\vspace{-2.75mm}

\scalemath{.5}{
\begin{tikzpicture}
  \begin{ternaryaxis}[
  clip=false,
    xmin=0,
    ymin=0,
    zmin=0,
    xmax=30,
    ymax=30,
    zmax=30,
    yscale=-1,
    xtick={0, ..., 30},
    ytick={0, ..., 30},
    ztick={0, ..., 30},
    xticklabels={},
    yticklabels={},
    zticklabels={},
    major tick length=0,
        minor tick length=0,
            every axis grid/.style={gray},
    minor tick num=1,
    ]

   \addplot3[color=red,ultra thick] coordinates {(0,30,0)(7,23,0)(7,18,5)(14,11,5)(14,6,10)(16,4,10)(16,0,14)};

   \addplot3[thick] coordinates {(0,23,7)(7,23,0)};
   \addplot3[thick] coordinates {(0,22,8)(8,22,0)};
   \addplot3[thick] coordinates {(0,21,9)(9,21,0)};
   \addplot3[thick] coordinates {(0,20,10)(10,20,0)};
   \addplot3[thick] coordinates {(0,19,11)(11,19,0)};
   \addplot3[thick] coordinates {(0,18,12)(12,18,0)};
   \addplot3[thick] coordinates {(0,17,13)(13,17,0)};
   \addplot3[thick] coordinates {(0,16,14)(14,16,0)};
   \addplot3[thick] coordinates {(0,15,15)(15,15,0)};
   \addplot3[thick] coordinates {(0,14,16)(16,14,0)};
   \addplot3[thick] coordinates {(0,13,17)(17,13,0)};
   
   \addplot3[thick] coordinates {(0,18,12)(12,18,0)};

  \end{ternaryaxis}
\end{tikzpicture}}
\caption{Break paths with strips from $i_l$ to $j_l+1$ in bold. \label{break paths of interval}}
\end{figure}
%%% label the strip indices in the figure above

Now suppose that face A does have an out elbow. Let A$_1$ denote face A, and define A$_k$ so that the $k$th shave step excavates face A$_k$ to reveal the smaller face A$_{k+1}$. Face A$_{k+1}$ consists of the strips indexed by $k+1$ to $n+1$. In particular, the $k$th shave step reveals the $k$th step of the break path of the bottom left face and hence determines $\rho_k$ from $\mu_k$. If there is a slant between the $k$ and $k+1$ strips of face A$_k$ and the weight along the $k+1$ strip is not increased during the shave step, then $\rho_k=\mu_k-1$. If this happened for $k\in [\alpha_1+1,h_l]$, then $\mu$ would yield a nondominant $\rho$. We will show that this does not happen because the special cases of the shave steps perform the desired sorting.

At the $i_l$ shave step, the break path of B$_{i_l}$ is horizontal until strip $h_l+1$ and slanted until strip $j_l+1$.  For each $k\in [i_l,a]$, the $k$th step in A's break path is horizontal, so the $k$th shave step reveals that $\rho_k=\mu_k$. Consider the $a+1$ shave step. Since there is an out elbow at $h_l+1$ in the break paths of both A$_{a+1}$ and B$_{a+1}$ and $\mu_{k}=\mu_{k-1}$ for all $k$, the special case of C(iii) occurs at strip $h_l+1$ and triggers a string of special cases of D(iii) in strips $a+1$ to $h_l$. 

Thus, although there was a slant between the $a+1$ and $a+2$ strips, the $a+1$ shave step reveals that $\rho_{a+1}=\mu_{a+1}$. The break path of A$_{a+2}$ now has slanted steps from strip $a+2$ to $h_l+2$ and from $b+1$ to $j_l+1$. It has horizontal steps from strip $i_l$ to $a+2$ and from $h_l+2$ to $b+1$. The break paths of A$_{a+2}$ and B$_{a+2}$ now have out elbows at strip $h_l+2$, so this process repeats. This continues for $b-h_l$ steps with the slanted intervals of the A break path moving closer together each time until they meet.  

The special case of C(iii) can no longer occur because the break path of B does not have an out elbow at the $j_l+1$ strip. The end result is $\rho_k=\mu_{k}-1$ on the final $(h_l-\alpha_1)+(j_l-\alpha_2)$ positions of $[i_l,j_l]$ to yield a dominant weight $\rho$.
\end{proof}

\bibliographystyle{amsalpha}
\bibliography{biblio2.0-arXiv}
 
\end{document}